\documentclass[a4paper, titlepage, oneside, 11pt]{amsart}
\usepackage{amsmath,amssymb,amsfonts,amstext,amscd,cite,enumitem,geometry,graphicx,latexsym,mathptmx,xcolor,pstricks}
\usepackage[utf8]{inputenc}
\usepackage{newunicodechar}
\usepackage[all]{xy}

\newtheorem{lem}{Lemma}[section]
\newtheorem{prop}[lem]{Proposition}
\newtheorem{thm}[lem]{Theorem}
\newtheorem{cor}[lem]{Corollary}
\newtheorem{conj}[lem]{Conjecture}

\newtheorem{rem}[lem]{Remark}
\newtheorem{prob}[lem]{Problem}

\newcommand{\la}{\lambda}

\def\R{{\mathbb{R}}}
\def\N{{\mathbb{N}}}
\def\Z{{\mathbb{Z}}}

\newcommand{\vertiii}[1]{{\left\vert\kern-0.25ex\left\vert\kern-0.25ex\left\vert #1
    \right\vert\kern-0.25ex\right\vert\kern-0.25ex\right\vert}}

\setlist[description]{leftmargin=\parindent,labelindent=0cm}

\numberwithin{equation}{section}

\allowdisplaybreaks

\begin{document}

\title[Detailed balance and invariant measures for discrete KdV- and Toda-type systems]{Detailed balance and invariant measures\\for discrete KdV- and Toda-type systems}

\author[D.~A.~Croydon]{David A. Croydon}
\address{Research Institute for Mathematical Sciences, Kyoto University, Kyoto 606-8502, Japan}
\email{croydon@kurims.kyoto-u.ac.jp}

\author[M.~Sasada]{Makiko Sasada}
\address{Graduate School of Mathematical Sciences, University of Tokyo, 3-8-1, Komaba, Meguro-ku, Tokyo, 153--8914, Japan}
\email{sasada@ms.u-tokyo.ac.jp}

\begin{abstract}
In order to study the invariant measures of discrete KdV- and Toda-type systems, this article focusses on models, discretely indexed in space and time, whose dynamics are deterministic and defined locally via lattice equations. A detailed balance criterion is presented that, amongst the measures that describe spatially independent and identically/alternately distributed configurations, characterizes those that are temporally invariant in distribution. A condition for establishing ergodicity of the dynamics is also given. These results are applied to various examples of discrete integrable systems, namely the ultra-discrete and discrete KdV equations, for which it is shown that the relevant invariant measures are of exponential/geometric and generalized inverse Gaussian form, respectively, as well as the ultra-discrete and discrete Toda lattice equations, for which the relevant invariant measures are found to be of exponential/geometric and gamma form. Ergodicity is demonstrated in the case of the KdV-type models. Links between the invariant measures of the different systems are presented, as are connections with stochastic integrable models and iterated random functions. Furthermore, a number of conjectures concerning the characterization of standard distributions are posed.
\end{abstract}


%37K60   	Lattice dynamics; integrable lattice equations
%37K10   	Completely integrable infinite-dimensional Hamiltonian and Lagrangian systems, integration methods, integrability tests, integrable hierarchies (KdV, KP, Toda, etc.)
%37L40   	Invariant measures for infinite-dimensional dissipative dynamical systems
%60E05: Distributions: general theory
%60J10   	Markov chains (discrete-time Markov processes on discrete state spaces)

\date{\today}

\maketitle

\setcounter{tocdepth}{2}
\tableofcontents
\newpage

\section{Introduction}

So as to capture the local dynamics of discrete KdV- and Toda-type systems, we consider a system of lattice equations with the following two-dimensional structure:\\
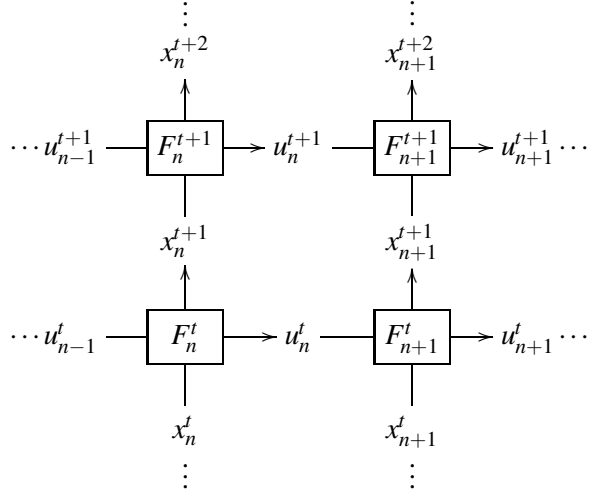
\centerline{\xymatrix@R-1pc@C-1pc{
&{\begin{array}{c}
   \vdots \\
   x_n^{t+2}
 \end{array}}\ar@{<-}[dd]&&{\begin{array}{c}
   \vdots \\
   x_{n+1}^{t+2}
 \end{array}}\ar@{<-}[dd]&\\
\cdots u_{n-1}^{t+1}\ar[rr]&\fcolorbox{black}{white}{\makebox[2em]{$F_{n}^{t+1}$\vphantom{$F_{n+1}^{t+1}$}}}&u_n^{t+1}\ar[rr]&\fcolorbox{black}{white}{\makebox[2em]{$F_{n+1}^{t+1}$}}&u_{n+1}^{t+1}\cdots\\
&x_n^{t+1}\ar@{<-}[dd]&&x_{n+1}^{t+1}\ar@{<-}[dd]&\\
\cdots u_{n-1}^t\ar[rr]&\fcolorbox{black}{white}{\makebox[2em]{$F_{n}^t$\vphantom{$F_{n+1}^{t+1}$}}}&u_n^t\ar[rr]&\fcolorbox{black}{white}{\makebox[2em]{$F_{n+1}^t$\vphantom{$F_{n+1}^{t+1}$}}}&u_{n+1}^t\cdots\\
&{\begin{array}{c}
   x_n^t\\
   \vdots
 \end{array}}&&{\begin{array}{c}
   x_{n+1}^t\\
   \vdots
   \end{array}}&}}
We will think of $n$ as the spatial coordinate, and $t$ as the temporal one. Moreover, the variables $(x^t_n)_{n\in\mathbb{Z}}$ will represent the configuration at time $t$, and $(u^t_n)_{n\in\mathbb{Z}}$ a collection of auxiliary variables through which the dynamics from time $t$ to $t+1$ are defined. As for the state spaces of the variables $(x^t_n,u^t_n)_{n,t\in\mathbb{Z}}$ and maps $(F^t_n)_{n,t\in\mathbb{Z}}$, we specialize to two cases:
\begin{description}
  \item[Type I (homogeneous) model] The variables $x_n^t$ take values in a common Polish space $\mathcal{X}_0$. Similarly, the variables $u_n^t$ take values in a common Polish space $\mathcal{U}_0$. Moreover, $F_n^t\equiv F$ for some involution $F:\mathcal{X}_0\times\mathcal{U}_0\rightarrow \mathcal{X}_0\times\mathcal{U}_0$.
  \item[Type II (alternating/bipartite) model] The variables $x_{n}^t$ take values in a Polish space $\mathcal{X}_0$ if $n+t=0$ (mod 2), and in Polish space $\tilde{\mathcal{X}}_0$ otherwise. Similarly, the variables $u_{n}^t$ take values in a Polish space $\tilde{\mathcal{U}}_0$ if $n+t=0$ (mod 2), and in Polish space ${\mathcal{U}}_0$ otherwise. Moreover, $F_n^t\equiv F_*$ for some bijection $F_*:\mathcal{X}_0\times\mathcal{U}_0\rightarrow \tilde{\mathcal{X}}_0\times\tilde{\mathcal{U}}_0$ if $n+t=0$ (mod 2), and $F_n^t\equiv F_*^{-1}$ otherwise.
\end{description}
This setting is rich enough to include a number of widely-studied discrete integrable systems, including the discrete and ultra-discrete KdV equations (which are examples of type I models), and the discrete and ultra-discrete Toda equations (which are examples of type II models). We highlight that these models are all important, fundamental examples of integrable systems that arise naturally within the Kadomtsev–Petviashvili hierarchy, which also includes the Korteweg-de Vries equation. See \cite{CSTkdv,TTud} for mathematical and physical background. As we will expand upon shortly, our interest will be in the evolution of such discrete integrable systems started from some random initial configuration. In particular, we give criteria for identifying spatially independent and identically/alternately-distributed (in the case of a type I/type II model, respectively) initial configurations that are distributionally invariant or ergodic in time under the dynamics of the system. These general results will be applied to each of the four aforementioned examples. Furthermore, in the latter part of the article, we discuss the relevance of our results to certain examples of stochastic integrable models, and to iterated random functions.

To give a more detailed description of our main results, let us proceed to define the dynamics associated with a type I/II model precisely. In particular, we start by letting $\mathcal{X}^*$ be the set of $(x_n)_{n\in\mathbb{Z}}$ in $\mathcal{X}_0^\mathbb{Z}$ for a type I model, or $(\mathcal{X}_0\times\tilde{\mathcal{X}}_0)^\mathbb{Z}$ for a type II model, for which there is a unique solution to the initial value problem:
\begin{equation}\label{initial}
\begin{cases}
F_{n}^t(x^{t}_n,u^t_{n-1})=(x^{t+1}_n, u^t_n),&\forall n,t \in \Z, \\
x^0_n =x_n,&\forall n \in \Z.
\end{cases}
\end{equation}
We then define a function $U$ on $\mathcal{X}^*$ by supposing $x=(x_n)_{n\in\Z}\mapsto (U_n(x))_{n\in\Z}:=(u_n^0)_{n\in\Z}$, where $(u_n^0)_{n\in\Z}$ is given by the unique solution of the initial value problem \eqref{initial} with $x^0_n=x_n$. For future convenience, we observe that $(U_n(x))_{n\in\Z}$ clearly solves
\begin{equation}\label{uniqueu}
\left(F_n^0\right)^{(2)}\left(x_n,U_{n-1}(x)\right)=U_{n}(x),\qquad \forall n\in\Z,
\end{equation}
where we use a superscript $(i)$ to represent the $i$th coordinate of a map. Finally, we define an operator $\mathcal{T}$ yielding the one time-step dynamics on $\mathcal{X}^*$
%DC - I don't think we can a priori assume the image is in $\mathcal{X}^*$?
%MS- Yes, you are right.
by supposing $\mathcal{T}(x)=(\mathcal{T}(x)_n)_{n\in\Z}$ is given by
\begin{equation}\label{dynamics}
\mathcal{T}(x)_n=
\begin{cases}
\left(F_n^0\right)^{(1)}\left(x_n,U_{n-1}(x)\right)=x^1_n,&\text{for a type I model,}\\
\left(F_{n+1}^0\right)^{(1)}\left(x_{n+1},U_{n}(x)\right)=x^1_{n+1},&\text{for a type II model,}\\
\end{cases}
\end{equation}
where $(x_n^1)_{n\in\Z}$ is given by the unique solution of the initial value problem \eqref{initial} with $x^0_n=x_n$. (The shift in the index $n$ is included in type II models to ensure that the elements of $x^1$ and $x^0$ that are in the spaces $\mathcal{X}_0$ and $\tilde{\mathcal{X}}_0$ are the same.) Note that we define the one time-step dynamics similarly on the set $\mathcal{X}^{\exists !}$ of configurations $(x_n)_{n\in\mathbb{Z}}$ for which there is a unique solution $(U_n(x))_{n\in\Z}$ to \eqref{uniqueu}. (NB.\ It is neither the case that $\mathcal{X}^{\exists !}\subseteq\mathcal{X}^{*}$ nor $\mathcal{X}^{*}\subseteq\mathcal{X}^{\exists !}$ in general, though on $\mathcal{X}^{*}\cap\mathcal{X}^{\exists !}$ the two definitions of $\mathcal{T}$ agree.)

Given that the global dynamics of the system arise from locally-defined maps, it is natural to ask whether it is possible to determine which measures supported on $\mathcal{X}^*$ are invariant under $\mathcal{T}$ based on local considerations. In our first result, we show that this is indeed the case for homogeneous/alternating product measures. Before stating the result, we introduce a notion of detailed balance in our setting.
\begin{description}
  \item[Detailed balance condition for a type I model] A pair of probability measures $(\mu,\nu)$ on $\mathcal{X}_0$ and $\mathcal{U}_0$ is said to satisfy the \emph{detailed balance condition} if
      \[F(\mu \times \nu)=\mu \times \nu,\]
      where we define $F(\mu\times\nu):=(\mu\times\nu)\circ F^{-1}$.
  \item[Detailed balance condition for a type II model] A quadruplet of probability measures ($\mu$,$\nu$, $\tilde{\mu}$,$\tilde{\nu}$) on $\mathcal{X}_0$,  $\mathcal{U}_0$, $\tilde{\mathcal{X}}_0$ and $\tilde{\mathcal{U}}_0$ is said to satisfy the \emph{detailed balance condition} if
      \[F_*(\mu \times \nu)=\tilde{\mu} \times \tilde{\nu}.\]
\end{description}
We then have the following characterization of independent and identically/alternately-distrib- uted configurations, which will be proved in Section \ref{generalsec}.
%DC - A forced hyphen here (remove if formatting does not require).

\begin{thm}[Detailed balance criteria for invariance]\label{detailedbalancenew}
\hspace{10pt}
\begin{enumerate}
  \item[(a)] \emph{Type I model.} Suppose $\mu$ is a probability measure on $\mathcal{X}_0$ and $\mu^{\Z}(\mathcal{X}^*)=1$. It is then the case that $\mathcal{T} \mu^{\Z}=\mu^{\Z}$ if and only if there exists a probability measure $\nu$ on $\mathcal{U}_0$ such that the pair $(\mu,\nu)$ satisfies the detailed balance condition. Moreover, when this holds, $\nu$ is the distribution of $U_{n}(x)$ for each $n$, where $x$ is distributed according to $\mu^{\Z}$.
  \item[(b)] \emph{Type II model.} Suppose $\mu$, $\tilde{\mu}$ are probability measures on $\mathcal{X}_0$, $\tilde{\mathcal{X}}_0$ and $(\mu \times \tilde{\mu})^{\Z}(\mathcal{X}^*)=1$. It is then the case that $\mathcal{T} (\mu \times \tilde{\mu})^{\Z}=(\mu \times \tilde{\mu})^{\Z}$ if and only if there exists probability measures $\nu$, $\tilde{\nu}$ on $\mathcal{U}_0$, $\tilde{\mathcal{U}}_0$, respectively, such that the quadruplet of probability measures $(\mu,\nu,\tilde{\mu},\tilde{\nu})$ satisfies the detailed balance condition. Moreover, when this holds, then $\nu$, $\tilde{\nu}$ are the distributions of $U_{2n-1}(x)$, $U_{2n}(x)$, respectively, for each $n$, where $x$ is distributed according to $(\mu \times \tilde{\mu})^{\Z}$.
\end{enumerate}
\end{thm}

We remark that the above theorem does not in itself provide a truly local criteria for invariance of homogeneous/alternating product measures under $\mathcal{T}$. Indeed, the condition that $\mu^{\Z}(\mathcal{X}^*)=1$ or $(\mu \times \tilde{\mu})^{\Z}(\mathcal{X}^*)=1$ depends on knowledge of the global dynamics, and in particular a suitably accessible description of $\mathcal{X}^*$. We do not present a universal approach to this problem here. However, for the KdV- and Toda-type systems already mentioned, the existence and uniqueness of solutions to the initial value problem \eqref{initial} was studied in detail in \cite{CSTkdv}, where it was shown that the associated dynamics could be interpreted in terms of certain `Pitman-type transformations' of related path encodings of the configurations. In this article, we will incorporate as a key ingredient the results of \cite{CSTkdv} when applying Theorem \ref{detailedbalancenew} to these examples. (NB.\ A brief introduction to the results of \cite{CSTkdv} is presented in \cite{CS4}.)

To prove Theorem \ref{detailedbalancenew}, we proceed in two steps. Firstly, we establish a weaker version (see Theorem \ref{detailedbalanceold} below), in which the invariance of $\mu^\mathbb{Z}$ or $(\mu\times\tilde{\mu})^\mathbb{Z}$ under $\mathcal{T}$ is shown to be equivalent to the detailed balance condition holding with $\nu$, $\tilde{\nu}$ given by the relevant marginals of $(U_{n}(x))_{n\in\mathbb{Z}}$. Since it is not trivial to deduce the distribution of $U_{n}(x)$ from $\mu$ or $\mu\times\tilde{\mu}$ in general, the latter version of the result is far from straightforward to apply in examples. Towards dealing with this issue, we show that invariant measures on $\mathcal{X}^*$ of homogeneous/alternating product form induce stationary/alternating measures of $(x_n^t, u_n^t)_{n,t \in \Z}$ satisfying Burke's property (see Subsection \ref{burkesec} below), and moreover they are the only such measures satisfying this property. Namely, Burke's property is equivalent to the detailed balance condition $F(\mu \times \nu) = \mu \times \nu$ or $F_*(\mu \times \nu) = \tilde{\mu} \times \tilde{\nu}$. Combining this observation with Theorem \ref{detailedbalanceold} yields our main result, i.e.\ Theorem \ref{detailedbalancenew}. See Section \ref{generalsec}, where a sufficient condition for establishing ergodicity of such invariant measures for type I models is also given, for details.

The abstract results discussed above are applied to our concrete KdV- and Toda-type examples of discrete integrable systems in Sections \ref{kdvsec} and \ref{todasec}, respectively. In particular, we show that spatially independent and identically/alternately distributed configurations that are also temporally invariant are of exponential/geometric form for the ultra-discrete KdV equation, of generalized inverse Gaussian form for the discrete KdV equation, of exponential/geometric form for the ultra-discrete Toda lattice, and of gamma form for the discrete Toda lattice. Our proofs for checking detailed balance for the various models depends on some well-known characterizations of certain standard distributions, including the exponential, geometric, gamma and generalized inverse Gaussian distributions \cite{Cr,Fe1,Fe2,L,LW}. Let us also highlight that the lattice structure of the Toda examples is not immediately covered by the framework of this article, with each being based on a map with three inputs and three outputs. Nonetheless, in both the discrete and ultra-discrete cases, it is possible to describe a type II model for which the involution $F:\tilde{\mathcal{X}}_0 \times  \mathcal{X}_0 \times  \mathcal{U}_0 \to  \tilde{\mathcal{X}}_0 \times  \mathcal{X}_0 \times  \mathcal{U}_0 $ defined by
\begin{equation}\label{threeinvolution}
F(a,b,c):=\left(F_*^{(1)}(b,c), F_*^{-1}\left(a, F_*^{(2)}\left(b,c\right)\right)\right)
\end{equation}
gives the appropriate dynamics. For a general involution of this form, we show that invariance under $F$, i.e.\
\begin{equation}\label{3invar}
F\left( \tilde{\mu} \times \mu \times \nu\right)=\tilde{\mu} \times \mu \times \nu,
\end{equation}
is equivalent to the detailed balance condition for $F_*$, i.e.\ $F_*(\mu \times \nu) = \tilde{\mu} \times \tilde{\nu}$ for some $\tilde{\nu}$, and indeed that both these conditions are equivalent to
\begin{equation}\label{23inv}
F^{(2,3)}\left( \tilde{\mu} \times \mu \times \nu\right)=\mu \times \nu.
\end{equation}
The detailed balance solutions that we derive in our examples yield corresponding invariant measures of the form described above. Our results yield that these satisfy Burke's property, and we also explore ergodicity for the KdV (type I) models. Moreover, in Section \ref{linksec}, we discuss natural relationships between the detailed balance solutions/invariant measures of the systems in question, which are based on an ultra-discretization procedure, and a certain KdV-Toda correspondence. See Figure \ref{dis} below for a summary of these results.

Although in this article we restrict to the case when the maps are deterministic, it is also possible to consider stochastic models, in which the maps $F_n^t$ themselves are random. In Section \ref{stochsec}, we provide some comments on generalizations of our results to this setting, and present links with certain stochastic integrable (solvable) lattice models, specifically last passage percolation, random polymers and higher spin vertex models. We note in particular that the relation at \eqref{23inv} is closely related to Burke's property for two-dimensional stochastic solvable models in integrable probability.

Another strand of literature to which the present article connects is that regarding iterated random functions. Indeed, one can understand \eqref{uniqueu} as a map $U_{n-1}\mapsto U_n$ based on the random function $f_{n,x_n}:=(F_n^0)^{(2)}(x_n,\cdot)$. Such systems arise in many settings, and there are a number of important problems that arise for them, such as the $(x_m)_{m\leq n}$-measurability of $U_n$. Moreover, if $(x_n)_{n\in\mathbb{Z}}$ is an independent sequence, then $U_n$ is a Markov chain (homogeneous for type I models, and with alternating transition probabilities for type II models), and one can ask questions about corresponding invariant measures and ergodicity for this process (or suitable variations for type II models). We will discuss how our results can be understood in this context in Section \ref{irfsec}.

Finally, in Section \ref{oqsec}, we summarize some of the open problems that are left open by this study, and present some conjectures on the characterization of some standard distributions that arise naturally from this study. We also include an appendix containing definitions of some of the probability distributions that appear in earlier sections.

\section{Setting and abstract results}
\label{generalsec}

In this section, we prove the abstract results outlined in the introduction. We continue to apply the definitions of a type I/II model, the set $\mathcal{X}^*$ of configurations for which there exists a unique solution to the initial value problem \eqref{initial}, the function $U$, and the operator $\mathcal{T}$, as given there. In Subsection \ref{dbsec}, we prove the weaker version of Theorem \ref{detailedbalancenew} discussed in the introduction. Moreover, in the type II setting, we establish the characterization of solutions to the detailed balance condition in terms of the conditions at \eqref{3invar} and \eqref{23inv}. In Subsection \ref{burkesec}, we present our conclusions concerning Burke's theorem in the present context. These allow us to strengthen the relevant result in Subsection \ref{dbsec}, and thereby obtain Theorem \ref{detailedbalancenew}. As noted above, this provides our means for checking invariance of homogeneous/alternating product measures under $\mathcal{T}$ in examples. Finally, in Subsection \ref{ergsec}, we develop an argument for checking the ergodicity of such invariant measures under $\mathcal{T}$ for type I models.

\subsection{The detailed balance condition and invariance}\label{dbsec}

Recalling the definition of the detailed balance condition for type I/II models from the introduction, the first goal of this subsection is to prove the following variation on Theorem \ref{detailedbalancenew}, which provides a link between detailed balance solutions and invariant measures.

\begin{thm}\label{detailedbalanceold}
\hspace{10pt}
\begin{enumerate}
  \item[(a)] \emph{Type I model.} Suppose $\mu$ is a probability measure on $\mathcal{X}_0$ and $\mu^{\Z}(\mathcal{X}^*)=1$. Let $\nu$ be
the distribution of $U_{-1}(x)$, where $x$ is distributed according to $\mu^{\Z}$. It is then the case that $\mathcal{T} \mu^{\Z}=\mu^{\Z}$ if and only if the pair $(\mu,\nu)$ satisfies the detailed balance condition.
  \item[(b)] \emph{Type II model.} Suppose $\mu$, $\tilde{\mu}$ are probability measures on $\mathcal{X}_0$, $\tilde{\mathcal{X}}_0$ and $(\mu \times \tilde{\mu})^{\Z}(\mathcal{X}^*)=1$. Let $\nu$, $\tilde{\nu}$ be the distributions of $U_{-1}(x)$, $U_0(x)$, respectively, where $x$ is distributed according to $(\mu \times \tilde{\mu})^{\Z}$. It is then the case that $\mathcal{T} (\mu \times \tilde{\mu})^{\Z}=(\mu \times \tilde{\mu})^{\Z}$ if and only if the quadruplet of probability measures $(\mu,\nu,\tilde{\mu},\tilde{\nu})$ satisfies the detailed balance condition.
\end{enumerate}
\end{thm}

\begin{rem}\label{remrem}
Let $\mathcal{X}^{U}$ be a set of configurations $(x_n)_{n\in\mathbb{Z}}$ for which there is a solution $(U_n(x))_{n\in\Z}$ to \eqref{uniqueu} for which $U_n$ is a function of $(x_m)_{m \le n}$ for all $n$, and $U_n (x) = \theta^n U_0 (x) =U_0( \theta^n x)$  for a type I model, and $U_{2n}=\theta^{2n}U_0=U_0\theta^{2n}$, $U_{2n+1}=\theta^{2n}U_1=U_1\theta^{2n}$ for a type II model, where $\theta$ is the usual shift operator. Moreover, assume that $\mathcal{T}\mathcal{X}^{U}\subseteq \mathcal{X}^{U}$, and $\mathcal{R}\mathcal{X}^{U}=\mathcal{X}^{U}$, where $\mathcal{T}=\mathcal{T}^U$ depends on $U$ through \eqref{dynamics}, and $\mathcal{R}x_n:=x_{1-n}$ for a type I model and $\mathcal{R}x_n:=x_{-n}$ for a type II model. If $\mathcal{T}\mathcal{R}\mathcal{T}\mathcal{R}$ is the identity map on $\mathcal{X}^U$, then Theorem \ref{detailedbalanceold} holds when we replace $\mathcal{X}^*$ by $\mathcal{X}^U$. It might be easier to find a space $\mathcal{X}^U$ than $\mathcal{X}^*$ in some cases.
\end{rem}

Towards proving Theorem \ref{detailedbalanceold}, we start by setting out a lemma on the measurability of $x^t_n$ and $u^t_n$ in terms of the initial configuration $x_n$. This is stated in terms of functions $X^t_n$ and $U^t_n$ on $\mathcal{X}^*$ that are defined via the relation
\[\left(X^t_n (x), U^t_n (x)\right)=(x^t_n,u^t_n),\qquad \forall x\in\mathcal{X}^*,\:n,t\in\Z,\]
where $(x^t_n,u^t_n)_{n,t\in\Z}$ is the unique solution of \eqref{initial} with initial condition $x$.

\begin{lem}\label{measurability}
Let $m \in \Z$.
\begin{enumerate}
  \item [(a)] For any $n \le m$ and $t \ge 0$, $X^t_n$ and $U^{t}_n$ are measurable with respect to $(x_n)_{n \le m}$.
  \item [(b)] For any $n \ge m+1$ and $t \le 0$, $X^t_n$ and $U^{t-1}_{n-1}$ are measurable with respect to $(x_n)_{n \ge m+1}$.
\end{enumerate}
\end{lem}
\begin{proof}
(a) Suppose there exist ${x}=(x_n)_{n\in\Z}$ and $y=(y_n)_{n\in\Z}$ in $\mathcal{X}^*$ such that $x_n=y_n$ for all $n \le m$, but $X^t_n({x}) \neq X^t_n({y})$ or $U^{t}_n ({x}) \neq U^{t}_n  ({y})$ for some $n \le m$, $t \ge 0$. We then define:
\begin{equation*}
\left\{
  \begin{array}{ll}
    \bar{x}^t_n:= X^t_n({y}),\: \bar{u}^{t}_n:= U^{t}_n({y}), & {n \le m,\: t \ge 0;} \\
    \bar{x}^t_n:=x_n, & { n > m,\: t=0;} \\
    \bar{x}^t_n:= X^t_n({x}),\:\bar{u}^{t}_n:= U^{t}_n({x}), & {n\in \Z,\: t < 0.}
  \end{array}
  \right.
\end{equation*}
Moreover, for $n >m$, $t>0$, it is clear from the lattice structure that there is a unique solution to $(\bar{x}^t_n, \bar{u}^{t-1}_n) : =F_{n}^{t-1}( \bar{x}^{t-1}_{n}, \bar{u}^{t-1}_{n-1})$ that is consistent with the previous definitions. Recursively, we have that $(\bar{x}^t_n, \bar{u}^t_n)_{n,t\in\Z}$ solves \eqref{initial} with initial condition $x$. Since $\bar{x}^t_n \neq x^t_n$ or $\bar{u}^{t}_n \neq u^{t}_n$ for some $n \le m$, $t \ge 0$ by assumption, this contradicts the uniqueness of the solution of \eqref{initial} for ${x}\in \mathcal{X}^*$. Hence we conclude that $X^t_n$ and $U^t_n$ are measurable with respect to $(x_n)_{n \le m}$.\\
(b) Appealing to the symmetry of the map $(x^t_n,u^t_n) \to (x^{1-t}_{1-n}, u^{-t}_{-n})$, we can apply the same proof as for part (a).
\end{proof}

In the next lemma, we rephrase spatial/temporal invariance of the law of an initial configuration as invariance under appropriate shifts of the induced law on variables on the entire lattice. Specifically, for a probability measure $P$ supported on $\mathcal{X}^*$, we denote by $\mathbf{P}_P$ the probability distribution of $(x^t_n,u^t_n)_{n,t\in\Z}$, as defined by the initial value problem \eqref{initial}, for which the marginal of $(x^0_n)_{n\in\Z}$ is given by $P$. We define a spatial shift $\theta$ on lattice variables by setting
\[\theta \left((x^t_{n},u^t_{n})_{n,t\in\mathbb{Z}}\right):=\left(x^t_{n+1},u^t_{n+1}\right)_{n,t\in\Z}.\]
Slightly abusing notation, for elements $x\in\mathcal{X}^*$, we similarly suppose $\theta (x)_n=x_{n+1}$. The corresponding temporal shift $T$ is given by
\[T \left((x^t_{n},u^t_{n})_{n,t\in\mathbb{Z}}\right):=\left(x^{t+1}_{n},u^{t+1}_{n}\right)_{n,t\in\Z}.\]
Note that if we consider $T$ as the map on $\mathcal{X}^*$ given by $T(x)_n=x^1_n$, then the definition of the dynamics at \eqref{dynamics} means that, for $x\in\mathcal{X}^*$,
\[\left\{\begin{array}{ll}
      \mathcal{T}(x)=T(x), & \hbox{for a type I model;} \\
      \mathcal{T}(x)=\theta \circ T(x), & \hbox{for a type II model.}
    \end{array}\right.\]
NB. From this description, it is easy to see that $\mathcal{T}$ is a bijection, with inverse operation $\mathcal{T}^{-1}=\mathcal{R}\mathcal{T}\mathcal{R}$, where $\mathcal{R}$ is defined as in Remark \ref{remrem}.

\begin{lem}\label{translation} Let $P$ be a probability measure supported on $\mathcal{X}^*$.
\begin{enumerate}
\item[(a)] For a type I model, $\mathcal{T} P=P$ if and only if $T \mathbf{P}_P= \mathbf{P}_P$. Also, $\theta P=P$ if and only if $\theta \mathbf{P}_P= \mathbf{P}_P$.
\item[(b)] For a type II model, $\mathcal{T} P=P$ if and only if $\theta \circ T\mathbf{P}_P= \mathbf{P}_P$. Also, $\theta^2 P=P$ if and only if $\theta^2 \mathbf{P}_P= \mathbf{P}_P$.
\end{enumerate}
\end{lem}
\begin{proof}
(a) If $\mathcal{T} P=P$ or $T \mathbf{P}_P= \mathbf{P}_P$ holds, then $\mathcal{T}P(\mathcal{X}^*)=1$, and so $\mathbf{P}_{\mathcal{T}P}$ is well-defined. The claim then follows from the fact that $T\mathbf{P}_P= \mathbf{P}_{\mathcal{T}P}$. The same argument works for $\theta$.\\
(b) Again, the same argument works.
\end{proof}

Combining the previous two lemmas, we have the following.

\begin{cor}\label{stationarymeasurability}
Let $P$ be a probability measure supported on $\mathcal{X}^*$, and suppose $\mathcal{T}P=P$. It is then the case that there is a subset of two-dimensional configurations $(x^t_n,u^t_n)_{n,t\in\Z}$ such that, with probability one on this subset, for any $m,s\in\Z$:
\begin{enumerate}
\item[(a)] for any $n \le m$ and $t \ge s$, $X^t_n$ and $U^{t}_n$ are measurable with respect to $(x_n^s)_{n \le m}$;
\item[(b)] for any $n \ge m+1$ and $t \le s$, $X^t_n$ and $U^{t-1}_{n-1}$ are measurable with respect to $(x_n^s)_{n \ge m+1}$.
\end{enumerate}
\end{cor}
\begin{proof}
For type I models, it is possible to deduce from Lemma \ref{translation}(a) that $x^s=(x^s_n)_{n\in\Z} \in \mathcal{X}^*$ for all $s \in \Z$, $\mathbf{P}_P$-a.s. Since $X^t_n(x)=X^{t-s}_n(x^s)$ and $U^t_n(x)=U^{t-s}_n(x^s)$ when $x^s \in \mathcal{X}_*$, Lemma \ref{measurability} completes the proof. The same argument works for type II model.
\end{proof}

Before proceeding, we note the following consequence of the above measurability results, which is somewhat related to Burke's property, as will be introduced in the next subsection. The particular statement will not be used later, but we believe it is of independent interest to observe that we do not require spatial stationarity of the initial configuration to establish temporal independence of the random variables $(u^t_0)_{t\in\Z}$.

\begin{cor}
Let $P$ be a probability measure $P$ supported on $\mathcal{X}^*$, and suppose $(x_n)_{n\in\Z}$ is an independent sequence under $P$.
\begin{enumerate}
\item[(a)] For a type I model, if it holds that $\mathcal{T}P=P$, then $(u^t_0)_{t\in\Z}$ is an independent and identically distributed (i.i.d.) sequence under $\mathbf{P}_P$.
\item[(b)] For a type II model, if it holds that $\mathcal{T}P=P$, then $(u^t_0)_{t\in\Z}$ is an independent and alternately-distributed sequence under $\mathbf{P}_P$.
\end{enumerate}
\end{cor}
\begin{proof} (a) Since $u^t_0=U_0(x^t)$, it readily follows that the sequence $(u^t_0)_{t\in\Z}$ is stationary. As for the independence claim, we note that, by Corollary \ref{stationarymeasurability}, $u^t_0$ is a measurable function of $(x^t_{n})_{n\leq 0}$, and $(u^s_0)_{s<t}$ is a measurable function of $(x^t_{n})_{n>0}$. Since $(x^t_{n})_{n\leq 0}$ and $(x^t_{n})_{n>0}$ are independent, the result follows.\\
(b) The proof is similar.
\end{proof}

We are nearly read to prove Theorem \ref{detailedbalanceold}. As the final ingredient, we give an elementary lemma regarding independence of sigma-algebras.

\begin{lem}\label{independence}
Let $\mathcal{G}_1,\mathcal{G}_2,\mathcal{G}_3$ be sigma-algebras on a probability space. If $\mathcal{G}_1$ and  $\mathcal{G}_2$ are independent, and $\sigma ( \mathcal{G}_1 \cup  \mathcal{G}_2)$ and $\mathcal{G}_3$ are independent, then $\mathcal{G}_1$ and $\sigma ( \mathcal{G}_2 \cup  \mathcal{G}_3)$ are independent.
\end{lem}
\begin{proof} Denoting by $P$ the probability measure on the relevant space, we have that, for any $E_i \in \mathcal{G}_i$, $i=1,2,3$, $P(E_1 \cap E_2 \cap E_3)= P(E_1 \cap E_2)P(E_3)= P(E_1) P(E_2)P(E_3)$. The result follows.
\end{proof}

\begin{proof}[Proof of Theorem \ref{detailedbalanceold}]
(a) Suppose $\mathcal{T} \mu^{\Z}=\mu^{\Z}$. By definition, we have that $x_0^0\sim\mu$ and $u_{-1}^0\sim\nu$. Moreover, by invariance under $\mathcal{T}$, we have that $x^1_0\sim\mu$. And, since $\theta \mu^{\Z}=\mu^{\Z}$, Lemma \ref{translation} yields that $\theta \mathbf{P}_{\mu^{\Z}}= \mathbf{P}_{\mu^{\Z}}$, and so the distribution of $u^0_0$ is also $\nu$. Now, by Corollary \ref{stationarymeasurability}, we have that $u_{-1}^0$ is a measurable function of $(x_{n}^0)_{n\leq -1}$, and $u_{0}^0$ is a measurable function of $(x_{n}^1)_{n\geq 1}$. In particular, it follows that $u_{-1}^0$ is independent of $x^0_0$, and $u_0^0$ is independent of $x_0^1$, i.e.\ it holds that $(x^0_0,u_{-1}^0)\sim\mu\times\nu$ and $(x^1_0,u_0^0)\sim\mu\times\nu$. Since $F(x_0^0, u_{-1}^0)=(x_0^1,u_0^0)$, we thus obtain that $\mu\times\nu$ satisfies the detailed balance condition in this case.

Next, suppose that $F(\mu \times \nu)=\mu \times \nu$. By Lemma \ref{measurability}, $u_{n-1}^0$ is measurable with respect to $(x_m^0)_{m \le n-1}$, so $x_n^0$ and $u_{n-1}^0$ are independent for all $n \in \Z$. By assumption $x_n^0\sim\mu$. Moreover, by assumption and the invariance $\theta \mathbf{P}_P=\mathbf{P}_P$ given by Lemma \ref{translation}, $u_{n-1}^0\sim\nu$. Hence the distribution of $x^1_n=F^{(1)}(x_n^0,u^1_{n-1})$ is $\mu$, and also $x^1_n$ and $u^0_n$ are independent, for all $n\in\Z$. Since, by Lemma \ref{measurability}, $u^0_0$ and $x^1_0$ are both measurable with respect to $(x_n)_{n \le 0}$, it follows from Lemma \ref{independence} that $x^1_0$ and $\sigma (u^0_0, x^0_1, x^0_2,x^0_3,\dots)$ are independent. Therefore, since $(x^1_n)_{n \ge 1}$ is measurable with respect to $\sigma (u^1_0, x^0_1, x^0_2,x^0_3,\dots)$, it must be the case that $x^1_0$ and $(x^1_n)_{n \ge 1}$ are independent. Finally, since $\theta \mathbf{P}_{\mu^{\Z}}= \mathbf{P}_{\mu^{\Z}}$ by Lemma \ref{translation}, we obtain that $(x^1_n)_{n\in\Z}$ is an i.i.d.\ sequence with marginal distribution $\mu$, and so $\mathcal{T}\mu^{\Z}=\mu^{\Z}$. \\
(b) Essentially the same argument as for part (a) applies.
\end{proof}

We complete the subsection by proving the alternative characterizations of the detailed balance condition for type II models that were presented in the introduction.

\begin{prop}\label{twothree}
Let $F_* :\mathcal{X}_0 \times  \mathcal{U}_0 \to  \tilde{\mathcal{X}}_0 \times  \tilde{\mathcal{U}}_0$ be a bijection, and define the involution $F : \tilde{ \mathcal{X}}_0 \times  \mathcal{X}_0 \times  \mathcal{U}_0 \to  \tilde{ \mathcal{X}}_0 \times  \mathcal{X}_0 \times  \mathcal{U}_0 $ as at \eqref{threeinvolution}. For a triplet of probability measures $(\mu, \nu, \tilde{\mu})$ on $\mathcal{X}_0$,  $\mathcal{U}_0$ and $\tilde{\mathcal{X}}_0$, the following three conditions are then equivalent.
\begin{enumerate}
\item[(a)] $F^{(2,3)}( \tilde{\mu} \times \mu \times \nu)=\mu \times \nu$.
\item[(b)] $F( \tilde{\mu} \times \mu \times \nu)=\tilde{\mu} \times \mu \times \nu$.
\item[(c)] There exists a probability measure $\tilde{\nu}$ on $\tilde{\mathcal{U}}_0$ such that the quadruplet of probability measures $(\mu,\nu,\tilde{\mu},\tilde{\nu})$ satisfies the detailed balance condition with respect to $F_*$.
\end{enumerate}
\end{prop}
\begin{proof}
{(b) $\Rightarrow$ (a)}: This is obvious.\\
{(c) $\Rightarrow$ (b)}: Let $X_0 \sim \mu$, $U_0 \sim \nu$, $\tilde{X}_0 \sim \tilde{\mu}$ be independent random variables, and define $(\tilde{X}'_0, \tilde{U}_0):=F_*(X_0,U_0)$. By (c), $(\tilde{X}'_0, \tilde{U}_0)\sim\tilde{\mu}\times\tilde{\nu}$. Moreover, by Lemma \ref{independence}, $ \tilde{X}_0, \tilde{X}'_0$ and $\tilde{U}_0$ are independent. Now, by definition, $F(\tilde{X}_0, X_0,U_0)=(\tilde{X}'_0, F_*^{-1}(\tilde{X}_0, \tilde{U}_0))$, and, by the detailed balance condition, $F_*^{-1}(\tilde{\mu} \times \tilde{\nu})=\mu \times \nu$, so (b) holds.\\
{(a) $\Rightarrow$ (c)}: Let $\tilde{\nu}:=F_*^{(2)}(\mu \times \nu)$, and $X_0 \sim \mu$, $U_0 \sim \nu$, $\tilde{X}_0 \sim \tilde{\mu}$ be independent random variables. Since $F^{(2,3)}(\tilde{X}_0, X_0,U_0)=F_*^{-1}(\tilde{X_0}, F_*^{(2)}(X_0,U_0))$ and the distribution of $(\tilde{X_0}, F_*^{(2)}(X_0,U_0))$ is $\tilde{\mu} \times \tilde{\nu}$, (a) implies $F_*^{-1}(\tilde{\mu} \times \tilde{\nu})=\mu \times \nu$.
\end{proof}

\subsection{Burke's property}\label{burkesec} Burke's theorem is a classical result in queueing theory, which states that, for an $M/M/1$ queue, the departure process at stationarity has the same law as the arrivals process, and that the departure process prior to a given time is independent of the current queue length \cite{Bq}. This result has been generalized to many settings, see Section \ref{stochsec} for discussion in the context of stochastic integrable systems in particular. In this subsection we present a definition of Burke's property for our model, and relate it to the study of the detailed balance condition and invariant homogeneous/alternating product measures. This allows us to complete the proof of Theorem \ref{detailedbalancenew}.

\begin{description}
  \item[Burke's property for a type I model] We say that a distribution supported on configurations $(x^t_{n},u^t_{n})_{n,t \in \Z}$ satisfying $F_{n}^t(x_n^{t},u_{n-1}^t)=(x_n^{t+1},u_n^t)$ satisfies \emph{Burke's property} if:
\begin{itemize}
\item the sequences $(x_n^0)_{n \geq1}$ and $(u_0^t)_{t \geq 0}$ are each i.i.d., and independent of each other;
\item the distribution of $(x^t_{n},u^t_{n})_{n,t \in \Z}$ is translation invariant, that is, for any $m,s \in \Z$,
    \[T^s\theta^m\left(\left(x^t_{n},u^t_{n}\right)_{n,t \in \Z}\right)\buildrel{d}\over{=}
   \left(x^t_{n},u^t_{n}\right)_{n,t \in \Z}.\]
\end{itemize}
\item[Burke's property for a type II model] We say that a distribution supported on configurations $(x^t_{n},u^t_{n})_{n,t \in \Z}$ satisfying $F_{n}^t(x_n^{t},u_{n-1}^t)=(x_n^{t+1},u_n^t)$ satisfies \emph{Burke's property} if:
\begin{itemize}
\item the sequences $(x_{2n}^0)_{n \geq1}$, $(x_{2n-1}^0)_{n \geq1}$, $(u_0^{2t})_{t \geq0}$ and $(u_0^{2t-1})_{t \geq 1}$ are each i.i.d., and independent of each other;
\item the distribution of $(x^t_{n},u^t_{n})_{n,t \in \Z}$ is translation invariant, that is, for any $m,s \in \Z$ such that $m+s=0$ (mod 2),
    \[T^s\theta^m\left(\left(x^t_{n},u^t_{n}\right)_{n,t \in \Z}\right)\buildrel{d}\over{=}
   \left(x^t_{n},u^t_{n}\right)_{n,t \in \Z}.\]
\end{itemize}
\end{description}
We make the obvious remark that, in the case of a type I model, if the distribution of $(x^t_{n},u^t_{n})_{n,t \in \Z}$ satisfies Burke's property, then $(x_n^t)_{n\in\Z}$ is i.i.d.\ for each $t \in \Z$, and $(u^t_n)_{t\in\Z}$ is i.i.d.\ for each $n\in\mathbb{N}$. A similar property holds for type II models.

In the main result of this subsection, we show that the existence of a solution to the detailed balance condition implies the existence of a distribution satisfying Burke's property. Moreover, in the case that the relevant marginal of this measure is supported on configurations for which \eqref{initial} has a unique solution, we are able to describe both the distributions of $x_n^t$ and $u_n^t$ in terms of the detailed balance solution.

\begin{prop}[Burke's property]\label{Burke}\hspace{10pt}
\begin{enumerate}
\item[(a)] Type I: If a pair of probability measures $(\mu,\nu)$ satisfies the detailed balance condition, then there exists a distribution supported on configurations $(x^t_{n},u^t_{n})_{n,t \in \Z}$ satisfying $F_{n}^t(x_n^{t},u_{n-1}^t)=(x_n^{t+1},u_n^t)$ for which Burke's property holds. Moreover, if it holds that $\mu^{\Z}(\mathcal{X}^*)=1$, then $u^0_{-1} \sim \nu$ and $\mathbf{P}_{\mu^\Z}$ satisfies Burke's property.
\item[(b)] Type II: If a quadruplet of probability measures $(\mu,\nu,\tilde{\mu}, \tilde{\nu})$ satisfies the detailed balance condition, then there exists a distribution supported on configurations $(x^t_{n},u^t_{n})_{n,t \in \Z}$ satisfying $F_{n}^t(x_n^{t},u_{n-1}^t)=(x_n^{t+1},u_n^t)$ for which Burke's property holds. Moreover, if it holds that $(\mu \times \tilde{\mu})^{\Z}(\mathcal{X}^*)=1$, then $u^{0}_{-1} \sim \nu$, $u^0_0 \sim \tilde{\nu}$, and $\mathbf{P}_{(\mu \times \tilde{\mu})^{\Z}}$ satisfies Burke's property.
\end{enumerate}
\end{prop}
\begin{proof}
(a) Let $(x_n^0,u_0^t)_{n\geq 1,t \geq 0}$ be independent random variables satisfying $x_n^0 \sim \mu$ and $u_0^t \sim \nu$. For $n, t \in \N$, define
\[\left(x_n^{t}, u_n^{t-1}\right):=F\left(x_{n}^{t-1}, u_{n-1}^{t-1}\right)\]
recursively. By induction and the detailed balance condition, one readily obtains that $x^1_n \sim \mu$, $u^0_n \sim \nu$ and $x^1_n$ and $u^{0}_n$ are independent for all $n \in \N$. Moreover, for any $n \in \N$, $x^1_n$ and $u^0_n$ are measurable with respect to $\sigma(u^0_0, x_1^0,x_2^0,\dots, x_n^0)$, and $(x_m^1)_{m \ge n+1}$ is measurable with respect to $\sigma(u^0_n, x_{n+1}^0,x_{n+2}^0,\dots)$. So, applying Lemma \ref{independence}, we find that $x^1_n$ and $(x_m^1)_{m \ge n+1}$ are independent. Hence $(x^1_n)_{n \in \N}$ is an i.i.d.\ sequence with the marginal $\mu$. Now, since $(x^1_n)_{n \in \N}$ is measurable with respect to $\sigma(u^0_0, (x_n^0)_{n \in \N})$, it further holds that $(x^1_n)_{n \in \N}$ and $(u^t_0)_{t \ge 1}$ are independent. Letting $y^t_n:=x^{t+1}_n$ and $v^t_n:=u^{t+1}_n$, we thus have that $(y_n^0,v_0^t)_{n\geq 1,t \geq0}$ are independent random variables satisfying $y_n^0 \sim \mu$, $v_0^t \sim \nu$ and
\[\left(y_n^t, v_n^{t-1}\right)=F\left(y_{n}^{t-1}, v_{n-1}^{t-1}\right)\]
for all $n,t \in \N$. In particular, $(x_n^t, u_n^t)_{n\geq 1,t \geq0}  \buildrel{d}\over{=}  (y_n^t, v_n^t)_{n\geq1,t\geq0}$, which implies
\[ \left(x_n^{t+1}, u_n^{t+1}\right)_{n\geq 1,t \geq0}\buildrel{d}\over{=}\left(x_n^t, u_n^t\right)_{n\geq 1,t \geq0}.\]
By the same argument, one can show that
\[ \left(x_{n+1}^{t}, u_{n+1}^{t}\right)_{n\geq 1,t \geq0}\buildrel{d}\over{=}\left(x_n^t, u_n^t\right)_{n\geq 1,t \geq0},\]
and so
\[ \left(x_{n+m}^{t+s}, u_{n+m}^{t+s}\right)_{n\geq 1,t \geq0}\buildrel{d}\over{=}\left(x_n^t, u_n^t\right)_{n\geq 1,t \geq0},\]
for any $m,s \in \N$. Finally, by constructing the distributions of $(x^t_n,u^t_n)_{n\geq k+1,t \ge k}$ for each $k \in \Z$ by translation, we can construct the distribution of $(x^t_n,u^t_n)_{n,t \in \Z}$ by applying the Daniell-Kolmogorov extension theorem, see \cite[Theorem 5.14]{Kall}, for example. (This is the one place in our arguments where we require the state spaces to be Polish.)
%MS ; Do we need to assume that \mathcal{X}_0, \mathcal{U}_0 to be a Polish space?
%DC - Yes, I think so, and added this to the model definition. To do without this, it would be enough to know that $(x_n^0,u_0^t)_{n\geq 1,t \geq 0}$ is a measurable function of $(x_n^t,u_n^t)_{n\geq 2,t \geq 1}$ and apply ionescu-tulcea extension theorem, for example.
Moreover, if $\mu^{\Z}(\mathcal{X}^*)=1$, then there is a unique distribution of $(x^t_n,u^t_n)_{n,t \in \Z}$ that is supported on configurations satisfying $F_{n}^t(x_n^{t},u_{n-1}^t)=(x_n^{t+1},u_n^t)$ and with marginal $(x_n^0)_{n\in\Z} \sim \mu^{\Z}$. Hence it must be the one satisfying Burke's property, as constructed above.
In particular, $u^t_n \sim \nu$ for all $n,t \in \Z$. \\
(b) The same argument as for part (a) works.
\end{proof}

\begin{proof}[Proof of Theorem \ref{detailedbalancenew}] Combine Theorem \ref{detailedbalanceold} and Proposition \ref{Burke}.
\end{proof}

We conclude the subsection with a corollary that establishes, when the marginal of $(x_n^0)_{n\in\Z}$ is supported on $\mathcal{X}^*$, Burke's property is actually equivalent to the detailed balance condition. As with Theorem \ref{detailedbalancenew}, it readily follows from Theorem \ref{detailedbalanceold} and Proposition \ref{Burke}.

\begin{cor}\hspace{10pt}
\begin{enumerate}
\item[(a)] Type I: Suppose that $\mu$ is a probability measure on $\mathcal{X}_0$ such that $\mu^{\Z}(\mathcal{X}^*)=1$. Let $\nu$ be the distribution of $U_{-1}(x)$, where $x\sim\mu^{\Z}$. It is then the case that there exists a distribution of $(x_n^t, u_n^t)_{n,t\in\Z}$ satisfying $(x_n^0)_{n\in\Z} \sim \mu^{\Z}$ and Burke's property if and only if $(\mu,\nu)$ satisfies the detailed balance condition.
\item[(b)] Type II: Suppose that $\mu\times\tilde{\mu}$ is a probability measure on $\mathcal{X}_0\times\tilde{\mathcal{X}}_0$ such that $(\mu \times \tilde{\mu})^{\Z}(\mathcal{X}^*)=1$. Let $\nu$, $\tilde{\nu}$ be the distributions of $U_{-1}(x)$, $U_0(x)$, respectively, where $x\sim(\mu \times \tilde{\mu})^{\Z}$. It is then the case that if there exists a distribution of $(x_n^t, u_n^t)_{n,t\in\Z}$ satisfying $(x_n^0)_{n\in\Z} \sim (\mu \times \tilde{\mu})^{\Z}$ and Burke's property if and only if $(\mu,\nu,\tilde{\mu},\tilde{\nu})$ satisfies the detailed balance condition.
\end{enumerate}
\end{cor}
\begin{proof}
(a) The `if' part is shown in Proposition \ref{Burke}. We prove the `only if' part. Suppose that there exists a distribution of $(x_n^t, u_n^t)_{n,t\in\Z}$ satisfying $(x_n^0)_{n\in\Z} \sim \mu^{\Z}$ and Burke's property. Since $\mu^{\Z}(\mathcal{X}^*)=1$, the measure must be $\mathbf{P}_{\mu^{\Z}}$. By the second condition of Burke's property, $T\mathbf{P}_{\mu^{\Z}}=\mathbf{P}_{\mu^{\Z}}$ holds. Hence, by Lemma \ref{translation}, we must have that $\mathcal{T} \mu^{\Z}=\mu^{\Z}$ holds. Consequently, by Theorem \ref{detailedbalanceold}, the detailed balance condition holds.\\
(b) The same argument as for part (a) works.
\end{proof}

\subsection{Ergodicity}\label{ergsec}

We now turn our attention to the issue of ergodicity. In this part of the article, we consider only type I models. Our main result gives a sufficient condition for the ergodicity of $\mathcal{T}$ for i.i.d.\ invariant measures. To state the result, we introduce an involution $\check{F} : \mathcal{U}_0 \times \mathcal{X}_0 \to \mathcal{U}_0 \times \mathcal{X}_0$ by setting
\[\check{F}=\pi \circ F \circ \pi,\]
where $\pi(u,x):=(x,u)$. We consider $\check{F}$ the \emph{dual} of $F$.

\begin{thm}\label{ergodicthm}
Suppose we have a type I model, and that $\mu$ is a probability measure on $\mathcal{X}_0$ such that $\mu^{\Z}(\mathcal{X}^*)=1$ and $\mathcal{T}\mu^{\Z}=\mu^{\Z}$. If it holds that, for $\mathbf{P}_{\mu^{\Z}}$-a.e.\ $u_0=(u_0^t)_{t\in\mathbb{Z}}$, there exists at most one $x=(x^t)_t \in \mathcal{X}_0^\Z$ such that
\[\check{F}^{(2)}(u_0^t,x^{t})=x^{t+1},\qquad \forall t \in \Z,\]
then $\mu^{\Z}$ is ergodic under $\mathcal{T}$.
\end{thm}

\begin{rem}\label{urem} We note that, by Theorem \ref{detailedbalanceold} and Proposition \ref{Burke}, under $\mathbf{P}_{\mu^{\Z}}$, $u_0=(u_0^t)_{t\in\mathbb{Z}}$ has law $\nu^{\Z}$, where $\nu$ is the distribution of $u_{-1}^0$ under $\mathbf{P}_{\mu^{\Z}}$. In particular, one could replace `$\mathbf{P}_{\mu^{\Z}}$-a.e.' with `$\nu^{\Z}$-a.e.' in the above statement.
\end{rem}

\begin{rem}
Under the assumptions of Theorem \ref{ergodicthm}, in addition to ergodicity, the same proof gives the measure-preserving transformation $\mathcal{T}$ is metrically isomorphic to a two-sided Bernoulli shift, cf.\ \cite{KO}.
\end{rem}

The proof of the above theorem will depend on the following lemma. For the statement of this, we define a function $\Lambda : \mathcal{X}^* \to \mathcal{U}_0^\Z$ by setting
\[\Lambda (x):= (u^t_0)_{t\in\Z},\]
where $(x_n^t,u_n^t)_{n,t\in\Z}$ is the unique solution of \eqref{initial} with initial condition $x$. Note that, as is consistent with the idea that $T$ is a temporal shift, we set $T((u^t_0)_{t\in\Z}):=(u^{t+1}_0)_{t\in\Z}$.

\begin{lem}\label{ergodic}
Let $P$ be a distribution on $\mathcal{X}^*$. Suppose there exists a set $\mathcal{U}^* \subseteq \mathcal{U}_0^\Z$ and a function
\[\tilde{\Lambda} : \mathcal{U}^* \to \mathcal{X}_0^\Z\]
such that $\Lambda P (\mathcal{U}^*)=1$ and $\tilde{\Lambda} \circ \Lambda$ is the identity map on the set $\{x\in \mathcal{X}^*:\: \Lambda(x)\in\mathcal{U}^*\}$. The following statements then hold.
\begin{enumerate}
\item[(a)] $P$ is invariant under $\mathcal{T}$ if and only if $\Lambda P$ is invariant under $T$.
\item[(b)] $P$ is invariant and ergodic under $\mathcal{T}$ if and only if $\Lambda P$ is invariant and ergodic under $T$.
\end{enumerate}
\end{lem}
\begin{proof}
(a) Define $\mathcal{X}^{**}:=\{x\in \mathcal{X}^*:\: \Lambda(x)\in\mathcal{U}^*\}$ and $\mathcal{U}^{**}:=\Lambda(\mathcal{X}^*) \cap  \mathcal{U}^*$. We first check that $\Lambda : \mathcal{X}^{**} \to   \mathcal{U}^{**}$ is a bijection with inverse function $\tilde{\Lambda}$. Clearly $\Lambda(\mathcal{X}^{**})\subseteq   \mathcal{U}^{**}$. Moreover, by assumption, $\tilde{\Lambda} \circ \Lambda (x)=x$ for all $x \in \mathcal{X}^{**}$. Hence it remains to show that
\[\Lambda  \circ  \tilde{\Lambda} (u)=u,\qquad \forall u \in \mathcal{U}^{**}.\]
For any $u \in \mathcal{U}^{**} \subseteq \Lambda(\mathcal{X}^*)$, there exists $x_u \in \mathcal{X}^{**}$ such that $\Lambda(x_u)=u$. It follows that
\[\Lambda\circ\tilde{\Lambda}(u)=\Lambda\circ\tilde{\Lambda}\circ\Lambda(x_u)=\Lambda(x_u)=u,\]
as required. Next, since $P(\mathcal{X}^*)=\Lambda P(\mathcal{U}^*)=1$, we have that $P(\mathcal{X}^{**})=1$, and thus also $\Lambda P (\mathcal{U}^{**})=1$. Consequently, if $\mathcal{T}P=P$, then it $P$-a.s.\ holds that $x:=(x_n)_{n\in\Z}$ and $\mathcal{T}(x)$ take values in $\mathcal{X}^{**}$, and so
\[\Lambda(\mathcal{T}(x)) = (u^{t+1}_0)_{t\in\Z} =T\left((u^{t}_0)_{t\in\Z}\right) = T \Lambda (x).\]
It follows that $T\Lambda P=\Lambda \mathcal{T} P=\Lambda P$. On the other hand, if $T\Lambda P=\Lambda \mathcal{T} P=\Lambda P$, then it $\Lambda P$-a.s.\ holds that $u:=(u^t_0)_{t\in\Z}$ and $T(u)$ takes values in $\mathcal{U}^{**}$, and so
\[\tilde{\Lambda} (T(u))= \tilde{\Lambda} \left((u^{t+1}_0)_{t\in\Z}\right) = \mathcal{T} (x)= \mathcal{T} \tilde{\Lambda} (u).\]
Hence $\mathcal{T} P= \mathcal{T} \tilde{\Lambda}\Lambda P =\tilde{\Lambda}T\Lambda P=\tilde{\Lambda} \Lambda P=P$.\\
(b) By the proof of (a), for any subset $E \subseteq \mathcal{X}^{**}$, $ \Lambda (\mathcal{T}(E))=T (\Lambda( E))$, and so $\mathcal{T}E=E$ is equivalent to $T \Lambda E= \Lambda E$. The claim follows.
\end{proof}

\begin{rem}
The same result was shown in \cite{CS} in the setting of the box-ball system of finite box and/or carrier capacity.
\end{rem}

\begin{proof}[Proof of Theorem \ref{ergodicthm}]
As per Remark \ref{urem}, we know that $\Lambda(\mu^{\Z})=\nu^{\Z}$. Moreover, $\nu^{\Z}$ is clearly invariant and ergodic under $T$. Hence, by Lemma \ref{ergodic}, we only need to show the existence of a set $\mathcal{U}^* \subseteq \mathcal{U}_0^\Z$ and a function $\tilde{\Lambda} : \mathcal{U}^* \to \mathcal{X}_0^\Z$ such that $\nu^{\Z} (\mathcal{U}^*)=1$ and $\tilde{\Lambda} \circ \Lambda$ is the identity map on the set $\{x\in\mathcal{X}^*:\:\Lambda(x)\in\mathcal{U}^*\}$. To this end, let $\mathcal{U}^{*,0} \subseteq \mathcal{U}_0^\Z$ be the set of $u=(u^t)_{t\in\Z}$ such that there is at most one $x=(x^t)_{t\in\Z} \in \mathcal{X}_0^\Z$ satisfying
\[\check{F}^{(2)}(u^t,x^{t})=x^{t+1},\qquad\forall t \in \Z.\]
By assumption, $\nu^{\Z}( \mathcal{U}^{*,0})=1$. Since $\mathcal{R} \nu^{\Z} =\nu^{\Z}$, where $\mathcal{R} u^t := u^{-t}$, and $u_n:=(u_n^t)_{t\in\Z} \sim \nu^{\Z}$ under $\mathbf{P}_{\mu^{\Z}}$ for all $n$, it follows that
\begin{equation}\label{as}
\mathbf{P}_{\mu^{\Z}}\left(u_n \in \mathcal{U}^{*,0} \cap \mathcal{R} \mathcal{U}^{*,0},\forall n\in\Z\right)=1.
\end{equation}
Now, define $\mathcal{X}^{**}$ to be the set of $x \in \mathcal{X}^*$ such that $u_n(x) \in   \mathcal{U}^{*,0} \cap \mathcal{R} \mathcal{U}^{*,0}$ for all $n$, where $u_n=(u_n^t)_{t\in\Z}$ is given by the solution of the initial value problem \eqref{initial} with initial condition $x$. Moreover, set $\mathcal{U}^*:=\Lambda(\mathcal{X}^{**})$, and note that, by \eqref{as}, we have that $\nu^{\Z} ( \mathcal{U}^*)=1$. We next claim that for any $u \in \mathcal{U}^*$, there is a unique $x \in \mathcal{X}^{*}$ such that $\Lambda(x)=u$, and moreover that $x\in\mathcal{X}^{**}$. Indeed, if $x\in \mathcal{X}^{**}$, $x'\in\mathcal{X}^*$ and $\Lambda(x)=u=\Lambda(x')$, then
\[F\left(x_1^{t},u^t_0\right)=\left(x_1^{t+1}, {u}_1^{t}\right),  \qquad F\left({x'}_1^{t},{u'}^t_0\right)=\left({x'}_1^{t+1}, {u'}_1^t\right),\]
where $(x^t_n,u^t_n)_{n,t\in\Z}$ and $({x'}^t_n, {u'}^t_n)_{n,t\in\Z}$ are the solutions of the initial value problem \eqref{initial} with initial conditions $x$ and $x'$, respectively. Hence,
\[\check{F}^{(2)}\left(u^t_0, x_1^{t}\right)=x_1^{t+1}, \qquad \check{F}^{(2)}\left({u'}^t_0, {x'}_1^{t}\right)={x'}_1^{t+1}.\]
Since $(u^t_0)_{t\in\Z}=\Lambda(x)= \Lambda(x')= ({u'}^t_0)_{t\in\Z}$ is an element of $\mathcal{U}^{*,0}$, it must therefore be the case that $x_1^t={x'}_1^t$ for all $t \in \Z$. It moreover follows that $u_1^t={u'}_1^t$ for all $t \in \Z$. Since $x \in  \mathcal{X}^{**}$ implies $u_n^t \in \mathcal{U}^{*,0}$ for all $n$,
iterating this argument yields that $x_n^t={x'}_n^t$ for all $t \in \Z$ and $n \ge 0$. To deal with negative $n$, note that
\[F\left(x_{0}^{t},u^{t}_{-1}\right)=\left(x_0^{t+1}, {u}_0^t\right),  \qquad F\left({x'}_{0}^{t},{u'}^t_{-1}\right)=\left({x'}_0^{t+1}, {u'}_0^t\right),\]
is equivalent to
\[\left(x_{0}^{t},u^{t}_{-1}\right)=F\left(x_0^{t+1}, {u}_0^t\right),  \qquad \left({x'}_{0}^{t},{u'}^t_{-1}\right)=F\left({x'}_0^{t+1}, {u'}_0^t\right),\]
and so
\[\check{F}^{(2)}\left(u^t_0, x_0^{t+1}\right)=x_{0}^{t}, \qquad \check{F}^{(2)}\left({u'}^t_0, {x'}_0^{t+1}\right)={x'}_0^{t}.\]
Applying the reflection $\mathcal{R}$ thus yields
\[\check{F}^{(2)}\left(u^{-t}_0, x_0^{-t+1}\right)=x_{0}^{-t}, \qquad \check{F}^{(2)}\left({u'}^{-t}_0, {x'}_0^{-t+1}\right)={x'}_0^{-t}.\]
Since $\mathcal{R}u_0 \in \mathcal{U}^{*,0}$, this implies $x_0^t={x'}_0^t$ for all $t\in\Z$. Again, we can iterate this argument to conclude that $x^t_n={x'}^t_n$ for all $t,n \in \Z$, as desired. Hence the function $\tilde{\Lambda} :  \mathcal{U}^{*} \to \mathcal{X}^{**}$ given by $\Lambda(x)\mapsto x$ is well-defined, and $\tilde{\Lambda} \circ \Lambda (x)=x$ for all $x \in\mathcal{X}^{**}$. Moreover, we have from the above argument that $\mathcal{X}^{**}=\{x\in\mathcal{X}^*:\:\Lambda(x)\in\mathcal{U}^*\}$, and so the proof is complete.
\end{proof}

\section{Type I examples: KdV-type discrete integrable systems}\label{kdvsec}

Two important examples of discrete integrable systems are the discrete and ultra-discrete KdV equations, which are obtained from the original KdV equation by natural discretization and ultra-discretization procedures. See \cite{CSTkdv,TTud} and the references therein for background. Both are examples of type I systems, and the aim of this section is to explain how our general results for such can be applied to identify examples of invariant and ergodic measures for them.

\subsection{Ultra-discrete KdV equation}

\subsubsection{The model}

The (modified) ultra-discrete KdV equation incorporates two parameters, $J,K \in \R \cup \{ \infty \}$, and is based on the following lattice map:
\begin{align}
\lefteqn{F^{(J,K)}_{udK}(x,u)}\tag{udKdV}\label{UDKDV}\\
&:=\left(u-\max\{x+u-J,0\}+\max\{x+u-K,0\},x-\max\{x+u-K,0\}+\max\{x+u-J,0\}\right),\nonumber
\end{align}
where the variables $x$ and $u$ are $\R$ valued. When the variables are positive, one can think of $x$ as the amount of mass currently at a lattice site, which has capacity $J$. Moreover, $u$ represents the amount of mass that a `carrier', which has capacity $K$, is bringing to this site. Simultaneously, the carrier deposits what it can, i.e.\ $\min\{u,J-x\}$, and collects what it can, i.e.\ $\min\{x,K-u\}$. This leaves a mass of
\[x+\min\{u,J-x\}-\min\{x,K-u\}=\left(F^{(J,K)}_{udK}\right)^{(1)}(x,u)\]
at the site, and the carrier moves forward (rightwards) to the next lattice site carrying a mass of
\[u-\min\{u,J-x\}+\min\{x,K-u\}=\left(F^{(J,K)}_{udK}\right)^{(2)}(x,u);\]
one discrete time step of the lattice dynamics is given by a complete pass of the carrier from $n=-\infty$ to $n=+\infty$. We note that the original udKdV equation corresponds to setting $K=\infty$. We also highlight that if $J,K \in \N$ and we restrict the possible values of the variables so that $x \in \{0,1,\dots,J\}$ and $u \in \{0,1,\dots,K\}$, then the dynamics associated with $F^{(J,K)}_{udK}$ correspond to the box-ball system with box capacity $J$ and carrier capacity $K$, which we denote by BBS($J$,$K$).

\begin{rem}\label{symrem}
Similarly to the discussion for BBS($J$,$K$) in \cite{CS}, the map \eqref{UDKDV} admits various symmetries, including the following.
\begin{description}
\item[Involution] For any $(x,u)\in \mathbb{R}^2$, it holds that
  \begin{equation}\label{involution}
  F_{udK}^{(J,K)}\circ F_{udK}^{(J,K)}(x,u)=(x,u).
  \end{equation}
\item[Configuration-carrier duality] If $\pi(x,u):=(u,x)$, then
\begin{equation}\label{ccdual}
F^{(J,K)}_{udK} = \pi  \circ F_{udK}^{(K,J)} \circ \pi.
\end{equation}
\item[Empty space-particle duality] Suppose $J,K<\infty$. If $\sigma_{J,K}(x,u):=(J-x,K-u)$, then
\begin{equation}\label{spdual}
F^{(J,K)}_{udK} = \sigma_{J,K}  \circ F^{(J,K)}_{udK} \circ \sigma_{J,K}.
\end{equation}
\item[Shift invariance] If $r\in\mathbb{R}$, then for any $(x,u) \in \mathbb{R}^2$ it holds that
      \begin{equation}\label{shift}
       F_{udK}^{(J-2r,K-2r)}(x-r,u-r)=\left(\left(F_{udK}^{(J,K)}\right)^{(1)}(x,u)-r,\left(F_{udK}^{(J,K)}\right)^{(2)}(x,u)-r\right).
        \end{equation}
\item[Scale invariance] If $\lambda\in\mathbb{R}$, then for any $(x,u) \in \mathbb{R}^2$ it holds that
      \begin{equation}\label{scale}
       F_{udK}^{(\lambda J,\lambda K)}(\lambda x,\lambda u)=\lambda F_{udK}^{(J,K)}(x,u).
        \end{equation}
\end{description}
Note that, whilst we will not dwell on it here, the property \eqref{involution} implies that the time-reversal of the \eqref{UDKDV} system can be studied in exactly the same way as the original system. As for \eqref{ccdual}, this means that it will suffice to solve the detailed balance equation for $J\leq K$. Properties \eqref{spdual}, \eqref{shift} and \eqref{scale} yield corresponding relationships between solutions of the detailed balance equation for \eqref{UDKDV} of various parameters.
\end{rem}

\subsubsection{Detailed balance solutions}

We now address the detailed balance equation for \eqref{UDKDV}; as per Remark \ref{symrem}, it will be enough to do this for $J\leq K$. We give two results. The first, Proposition \ref{udkdvmeas} lists a number of solutions of the detailed balance equation. We highlight that the detailed balance equation was completely solved for the BBS($J$,$K$) in \cite{CS}, and the discrete part of the following result (i.e.\ (a)(ii)) is essentially a restatement of the result from that paper. We refer the reader to the appendix for definitions of the probability distributions that appear. Our second result, Proposition \ref{udkdvmeasall} shows, up to a technical condition, that these are all the solutions of the detailed balance equation in this setting.

\begin{prop}\label{udkdvmeas}
The following product measures $\mu\times\nu$ satisfy $F^{(J,K)}_{udK}(\mu\times\nu)=\mu\times\nu$.
\begin{enumerate}
  \item[(a)] Suppose $J,K \in \R \cup \{ \infty \}$.
             \begin{enumerate}
              \item[(i)] For $\lambda\in\mathbb{R}$ if $\max\{J,K\}<\infty$, or $\lambda>0$ if $\max\{J,K\}=\infty$, and finite $c<\min\{\frac{J}{2},\frac{K}{2}\}$,
\[\mu\times\nu=\mathrm{stExp}(\lambda,c,J-c)\times \mathrm{stExp}(\lambda,c,K-c).\]
\item[(ii)] For finite $c<\min\{\frac{J}{2},\frac{K}{2}\}$ and $m>0$ such that $c,J,K\in m\mathbb{Z}\cup\{\infty\}$,
\[\mu\times\nu=\mathrm{sstbGeo}\left(1-\theta,\frac{c}{m},\frac{J-c}{m},\kappa,m\right)\times\mathrm{sstbGeo}\left(1-\theta,\frac{c}{m},\frac{K-c}{m},\kappa,m\right),\]
where it is further supposed that: either $J-2c,K-2c\in m\mathbb{Z}\cup\{\infty\}$, $\theta\in(0,1)$, $\kappa=1$; or $J-2c,K-2c\in 2m\mathbb{Z}\cup\{\infty\}$, $\theta\in(0,1)$, $\kappa\in(0,\infty)\backslash\{1\}$; or $J-2c,K-2c\in m\mathbb{Z}$, $\theta\geq 1$, $\kappa=1$; or $J-2c,K-2c\in 2m\mathbb{Z}\cup\{\infty\}$, $\theta\geq 1$, $\kappa\in(0,\infty)\backslash\{1\}$.
\end{enumerate}
  \item[(b)] Suppose $J=K$. For any measure $m$ on $\mathbb{R}$,
\[\mu\times\nu=m\times m.\]
\item[(c)] Suppose $J<K$.
 \begin{enumerate}
 \item[(i)] For any measure $m$ supported on $(-\infty,\frac{J}{2}]$,
\[\mu\times\nu=m\times m.\]
\item[(ii)] For any measure $m$ supported on $[\frac{J}{2},K-\frac{J}{2}]$,
\[\mu\times\nu=\delta_{\frac{J}{2}} \times m,\]
where for $x\in\mathbb{R}$, $\delta_x$ is the probability measure placing all of its mass at $x$.
\item[(iii)] Suppose further that $K<\infty$. For any measure $m$ supported on $[\frac{J}{2},\infty)$,
\[\mu\times\nu=m\times (m +L),\]
where $L:=K-J$ and $(m+L)(A)=m (\{x-L:\: x \in A\})$.
\end{enumerate}
\end{enumerate}
\end{prop}
\begin{proof}
Since $F^{(J,K)}_{udK}$ preserves mass, i.e.
\[\left(F_{udK}^{(J,K)}\right)^{(1)}(x,u)+\left(F_{udK}^{(J,K)}\right)^{(2)}(x,u)=x+u,\]
and the absolute value of the associated Jacobian determinant is equal to one (Lebesgue almost-everywhere), part (a)(i) is straightforward to check. As already noted, part (a)(ii) was proved in \cite{CS}. Parts (b) and (c) readily follow from the definition of $F^{(J,K)}_{udK}$, and so their proofs are omitted.
\end{proof}

\begin{prop}\label{udkdvmeasall}
\hspace{10pt}
\begin{enumerate}
  \item[(a)] Suppose $J=K$. It is then the case that the product measures given in Proposition \ref{udkdvmeas}(b) are the only solutions to $F^{(J,K)}_{udK}(\mu\times\nu)=\mu\times\nu$.
\item[(b)] Suppose $J < K$ and a product measure $\mu\times\nu$ satisfies $F^{(J,K)}_{udK}(\mu\times\nu)=\mu\times\nu$. It is then the case that one of the following statements hold.
 \begin{enumerate}
 \item[(i)] The product measure $\mu\times\nu$ is given in Proposition \ref{udkdvmeas}(c).
 \item[(ii)] There exists $c \in [-\infty, \frac{J}{2})$ such that
 \[\inf \mathrm{supp} (\mu)=\inf \mathrm{supp} (\nu)=c,\]
 \[\sup \mathrm{supp} (\mu)=J-c, \qquad \sup \mathrm{supp} (\nu)=K-c,\]
 where $\mathrm{supp} (\mu)$ and $\mathrm{supp} (\nu)$ are the support of $\mu$ and $\nu$, respectively.
 \end{enumerate}
Moreover, if (ii) holds and $\mu$ and $\nu$ have smooth (twice differentiable), strictly positive densities on the intervals $[c,J-c]$ and $[c,K-c]$ respectively, then they given by Proposition \ref{udkdvmeas}(a)(i). And, if (ii) holds and neither $\mathrm{supp}(\mu)$ nor $\mathrm{supp}(\nu)$ contains an accumulation point, then they are given by Proposition \ref{udkdvmeas}(a)(ii).
\end{enumerate}
\end{prop}
\begin{proof}
(a) Since $F^{(J,J)}_{udK}(x,u)=(u,x)$, this part of the result is obvious.\\
(b) Let $a_1:=\inf \mathrm{supp} (\mu)$, $a_2:=\sup \mathrm{supp} (\mu)$,  $b_1:=\inf \mathrm{supp} (\nu)$, $b_2:=\sup \mathrm{supp} (\nu)$. Since
\[0 \le \max\{x+u-J,0\}-\max\{x+u-K,0\} \le L,\qquad \forall (x,u)\in\mathbb{R}^2,\]
where $L:=K-J$, $F^{(J,K)}_{udK}(x,u)=(y,v)$ implies $u-L \le y \le u$ and $x \le v \le x+L$. Thus it holds that
\[a_1 \le b_1 \le a_1+L,\qquad a_2 \le b_2 \le a_2 +L.\]
Also, by definition, it holds that:
\[F^{(J,K)}_{udK}(x,u)=\left\{\begin{array}{ll}
                               (u,x),  & \mbox{if }x+u \le J, \\
                               (J-x, u+2x-J), & \mbox{if }J \le x+u \le K, \\
                               (u-L,x+L), & \mbox{if } x+u \ge K,
                             \end{array}
\right.\]
and, in particular, $F^{(J,K)}_{udK}(x,u)$ is continuous with respect to $(x,u)$. We now consider three cases separately: (I) $a_1+b_1 < J$, (II) $J \le a_1+b_1 < K$, (III) $a_1+b_1 \ge K$.
\begin{enumerate}
\item[(I)] If $a_1+b_1 < J$, then $F^{(J,K)}_{udK}(a_1,b_1)=(b_1,a_1)$. This implies $a_1 \le b_1$, $b_1 \le a_1$, and so $a_1 =b_1 < \frac{J}{2}$.
\item[(II)] If $J \le a_1+b_1 < K$, then $F^{(J,K)}_{udK}(a_1,b_1)=(J-a_1, b_1+2a_1-J)$. Hence $a_1 \le J-a_1$, $b_1 \le b_1+2a_1-J$, which implies in turn that $a_1=\frac{J}{2}$ and $\frac{J}{2} \le b_1 < K-\frac{J}{2}$. If $a_1=a_2$, namely $\mu$ is the measure $\delta_{J/2}$, then $\nu$ must be concentrated on $[b_1, K-\frac{J}{2}]$. If $a_1 <a_2$, then there exist $\varepsilon >0$, $\varepsilon' \ge 0$ such that $a_1 +\varepsilon \in  \mathrm{supp} (\mu)$, $b_1+\varepsilon' \in \mathrm{supp}( \nu)$. In particular, we can take $\varepsilon'$ small enough so that $a_1+b_1+\varepsilon' < K$. If $a_1+b_1+\varepsilon+\varepsilon' \le K$, then, $J-(a_1+\varepsilon)= \frac{J}{2}-\varepsilon \in \mathrm{supp} (\mu)$, but this contradicts with the fact that $a_1 =\frac{J}{2}$. On the other hand, if $a_1+b_1+\varepsilon+\varepsilon'  >K$, then $b_1+\varepsilon'-L \in  \mathrm{supp} (\mu)$. However, $b_1+\varepsilon' -L <K-a_1-L=\frac{J}{2}$, which again contradicts with $a_1 =\frac{J}{2}$. Thus we have shown that it is not possible that $a_1<a_2$. Consequently, in this case, if $F^{(J,K)}_{udK}(\mu\times\nu)=\mu\times\nu$ holds, then $\mu=\delta_{\frac{J}{2}}$ and $\mathrm{supp} (\nu) \subseteq [\frac{J}{2}, K-\frac{J}{2}]$.

\item[(III)] If  $a_1+b_1 \ge K$, then $F^{(J,K)}_{udK}(x,u)=(u-L,x+L)$ for all $x,u \in [a_1,a_2] \times [b_1,b_2]$, so $F^{(J,K)}_{udK}(\mu\times\nu)=\mu\times\nu$ holds if and only if $\nu=\mu+L$.
\end{enumerate}
We next consider the corresponding three cases for the suprema of the support: (I') $a_2+b_2 \le J$, (II') $J < a_2+b_2 \le  K$, (III') $a_2+b_2 > K$. By a similar argument to above, we have the following.
\begin{enumerate}
\item[(I')] If $a_2+b_2 \le J$, then $F^{(J,K)}_{udK}(\mu\times\nu)=\mu\times\nu$ holds if and only if $\nu=\mu$.
\item[(II')]  If $J < a_2+b_2 \le K$ and $F^{(J,K)}_{udK}(\mu\times\nu)=\mu\times\nu$ holds, then $\mu=\delta_{\frac{J}{2}}$, $\mathrm{supp}(\nu)\subseteq [\frac{J}{2}, K-\frac{J}{2}]$.
\item[(III')] If $a_2+ b_2 > K$, then $b_2=a_2+L$ and $a_2 > \frac{J}{2}$.
\end{enumerate}
Putting together the above discussion, there are only four possible cases: (I''-1) $\mu=\nu$ and $a_2 \le \frac{J}{2}$; (I''-2) $\mu=\delta_{\frac{J}{2}}$, $\mathrm{supp} (\nu) \subseteq [\frac{J}{2}, K-\frac{J}{2}]$; (I''-3) $\mu + L = \nu$ and $a_1 \ge \frac{J}{2}$; (II'')  $a_1= b_1$, $a_2=b_2-L$ and $a_1 < \frac{J}{2}$, $a_2 > \frac{J}{2}$. The cases (I''-1), (I''-2), (I''-3) correspond to Proposition \ref{udkdvmeas}(c)(i), (ii), (iii), respectively. It remains to check that the case (II'') corresponds to part (b)(ii) of the current proposition. In this case, there exist $c_1, c_2>0$ such that $a_1=b_1=\frac{J}{2}-c_1$ and $a_2=b_2-L=\frac{J}{2}+c_2$. Suppose $c_1 > c_2$. Then, $a_1+b_2= J-c_1+c_2+L=K -c_1+c_2 < K$. If $J \le a_1+b_2 < K$, then $F^{(J,K)}_{udK}(a_1,b_2)=(J-a_1, b_2+2a_1-J)$, and so $J-a_1 \le a_2$. The latter inequality is equivalent to $c_1\leq c_2$, which contradicts $c_1 >c_2$. If $a_1+b_2 < J$, then $F^{(J,K)}_{udK}(a_1,b_2)=(b_2, a_1)$, which implies $b_2 \le a_2$. However, this contradicts $a_2=b_2-L$. Hence $c_1 \le c_2$. A similar argument allows one to deduce the reverse inequality, and thus we obtain $c_1=c_2$. In conclusion, letting $c=\frac{J}{2}-c_1$, we obtain the desired result.

To complete the proof, we study the special cases where $\mu$ and $\nu$ have densities, or they are discrete. Let $f_{\mu}$, $f_{\nu}$ be densities of $\mu$ and $\nu$. For $x \in [c, J-c]$ and $ u \in [c,K-c]$, we then have that
\[f_{\mu}(x)f_{\nu}(u)=\left\{\begin{array}{ll}
                                f_{\mu}(u)f_{\nu}(x), & \mbox{if }x+u \le J,\\
                                f_{\mu}(J-x)f_{\nu}(u+2x-J), & \mbox{if }J \le x+u \le K, \\
                                f_{\mu}(u-L)f_{\nu}(x+L), & \mbox{if }x+u \ge K.
                              \end{array}
\right.\]
Letting $h_{\mu}(x):=\log f_{\mu}(x)$ and $h_{\nu}(u):=\log f_{\nu}(u)$ and taking derivatives of the relation
\[h_{\mu}(x)+h_{\nu}(u)=h_{\mu}(J-x)+h_{\nu}(u+2x-J)\]
with respect to $x$ first and then with respect to $u$, for $(x,u)$ satisfying $J \le x+u \le K$, we have $h''_{\nu}(u+2x-J)=0$. For any $v \in [c, K-c]$, by letting $\varepsilon:=\frac{v-c}{K-2c} \in [0,1]$ and
\[x =c+ \varepsilon (J-2c),\qquad u =c+ (1-\varepsilon)(J-2c)+\varepsilon (K-J),\]
we have $x \in [c,J-c]$, $u \in [c, K-c]$, $J \le x+u \le K$ and $v=u+2x-J$, so $h''_{\nu}(v)=0$ for all $v \in [c, K-c]$. Therefore, there exists $\lambda \in \R$ such that $h'_{\nu}(u)=\lambda$ for all $u \in [c,K-c]$. Also, by taking the derivative of $h_{\mu}(x)+h_{\nu}(u)=h_{\mu}(u)+h_{\nu}(x)$ with respect to $x$, for $(x,u)$ satisfying $x+u \le J$, we have $h'_{\mu}(x)=h'_{\nu}(x)$. Since for any $x \in [c,J-c]$, by taking $u=J-x$, we have $u \in [c,K-c]$ and $x+u \le J$, hence $h'_{\mu}(x)=\lambda$ for all $x \in [c,J-c]$. Therefore, since $\mu$ and $\nu$ are probability measures, $c$ must be finite. Moreover, if $K= \infty$, then $\lambda$ must be positive.

Finally, we consider the case where $\mu$ and $\nu$ are discrete. We first prove that $c>- \infty$.
If $c=-\infty$, then for any $x \in \mathrm{supp}( \mu)$, there exists $u \in \mathrm{supp}(\nu)$ such that
$x+u \le J$, and vice versa. Since for such $x,u$ we have $F_{udK}^{(J,K)}(x,u)=(u,x)$, we conclude that $\mathrm{supp} (\mu) = \mathrm{supp} (\nu)$.  Moreover, by noting $\mu(\{x\})\nu(\{u\})=\mu (\{u\})\nu(\{x\})$ for $x + u \le J$, it is an elementary exercise to check that $\mu=\nu$. Next, note that if $K< \infty$, then for any $x+u \ge K$ with $x,u \in \mathrm{supp}(\mu)$, it holds that $x+L,u-L \in \mathrm{supp} (\mu)$
and
\[\frac{\mu(\{x\})}{\mu(\{x+L\})}=\frac{\mu(\{u-L\})}{\mu(\{u\})}.\]
However, since $\sup \mathrm{supp} (\mu) = \infty$, this implies that for any $x \in \mathrm{supp} (\mu)$, $x+nL \in \mathrm{supp} (\mu)$ for all $n \in \Z$ and $\mu(\{x+nL\})=\mu(\{x\}) \la^{n}$ for some $\la \neq 0$, which can not happen since $\mu$ is a probability measure. Similarly, if $K= \infty$, then for any $x, n, m$ satisfying $\frac{J}{2}+nx, \frac{J}{2}+mx \in \mathrm{supp} (\mu)$ and $n+m \ge 0$, we have $\frac{J}{2}-nx$, $\frac{J}{2}+(2n+m)x \in \mathrm{supp} (\mu)$ and
\[\frac{\mu(\{\frac{J}{2}+nx\})}{\mu(\{\frac{J}{2}+(2n+m)x\})}=\frac{\mu(\{\frac{J}{2}-nx\})}{\mu(\{\frac{J}{2}+mx\})}.\]
In particular, applying this relation with $x \ge 0$ satisfying $\frac{J}{2}+x \in \mathrm{supp}(\mu)$ with $n=m=1$, we have $\frac{J}{2}-x, \frac{J}{2}+3x \in \mathrm{supp} (\mu)$. Iterating this argument yields
$ \frac{J}{2}+(2n+1)x \in \mathrm{supp} (\mu)$ for all $n \in \Z$, and
\[\frac{\mu(\{\frac{J}{2}+x\})}{\mu(\{\frac{J}{2}-x\})}=\frac{\mu(\{\frac{J}{2}+(2n+1)x\})}{\mu(\{\frac{J}{2}+(2n-1)x\})}\]
for all $n \ge 0$. Moreover, since for $n \le -1$,
\[\frac{\mu(\{\frac{J}{2}+(2n-1)x\})}{\mu(\{\frac{J}{2}+(2n+1)x\})} = \frac{\mu(\{\frac{J}{2}-(2n-1)x\})}{\mu(\{\frac{J}{2}-(2n-3)x\})},\]
we have $\mu(\{\frac{J}{2}+(2n+1)x\})=\mu(\{\frac{J}{2}+x\}))\la^n$ for all $n\in\mathbb{Z}$, for some $\la \neq 0$. Again, this can not happen since $\mu$ is a probability measure. We can therefore conclude that $c>-\infty$.

First suppose $K<\infty$. We then have $\mathrm{supp} (\mu) =\{x_0,x_1,\dots,x_n\}$ for some $n$ with $c=x_0 <x_1<\dots<x_n=J-c$, and $\mathrm{supp} (\mu)\subset c+b\mathbb{Z}$ for some $b>0$. Additionally, $\mathrm{supp} (\nu)  =\{u_0,u_1,\dots, u_{m}\}$ for some $m$ with $c=u_0 <u_1<\dots<u_m=K-c$, and $\mathrm{supp} (\nu)\subset c+b'\mathbb{Z}$ for some $b'>0$. By a similar argument to the previous paragraph, it is possible to check that, for an appropriate choice of $b$, one may take $b'=b$, and moreover $x_i=c+ib$, $u_i=c+ib$ for each $i$. Hence, by making the change of variables $x \to x-c, u \to u-c, J \to J-2c, K \to K-2c$, we can apply \cite[Lemma 4.5]{CS} to complete the proof. To establish the result when $K=\infty$, one can proceed in the same way to check that $\mathrm{supp} (\nu)\cap[c,J-c]=\mathrm{supp}(\mu)=(c+b\mathbb{Z})\cap[c,J-c]$ for some $b>0$, and then use the identity $\mu(\{c\})\nu(\{J+x\})=\mu(\{J-c\})\nu(\{x+2c\})$ for $x\geq -c$ to derive the full support of $\nu$, from which point one can again apply \cite[Lemma 4.5]{CS} to obtain the desired result.
\end{proof}

\subsubsection{Invariant measures}

Much of the hard work for identifying invariant product measures for \eqref{UDKDV} has now been done. Indeed, up to the technical restriction of Proposition \ref{udkdvmeasall}, Theorem \ref{detailedbalancenew} tells us that the marginals of invariant product measures must be described within the statement of Proposition \ref{udkdvmeas} (as $\mu$ in the case $J\leq K$, and $\nu$ in the case $J\geq K$).

We start by restricting our attention to $J\leq K$. The reason for this is that it allows us to apply the approach of \cite{CKST,CS,CSTkdv}, which provides a description of the dynamics in terms of certain Pitman-type transformations of path encodings of configurations, to give an explicit set upon which the initial value problem \eqref{initial} has a unique solution. In particular, we will now consider the initial value problem \eqref{initial} with $F_{n,t}=F^{(J,K)}_{udK}$ for all $n,t$, where $J \le K$. For $J=K$, we set $\mathcal{X}^*_{J,K}:=\mathbb{R}^\mathbb{Z}$. For $J<K=\infty$, we take
\[\mathcal{X}^*_{J,K}:=\left\{ (x_n)_{n\in\mathbb{Z}}:\: \lim_{|n| \to \infty}\frac{\sum_{k=1}^n \left(J-2x_k\right)}{n} >0\right\},\]
where for $n<0$, the sum $\sum_{k=1}^n$ should be interpreted as $-\sum^{0}_{n+1}$, and, for $J<K <\infty$,
\[\mathcal{X}^*_{J,K}:=\left\{ (x_n)_{n\in\mathbb{Z}}:\: \limsup_{n \to \pm \infty}\left|\sum_{k=1}^n \left(J-2x_k\right)\right|=\infty \right\}.\]
From results of \cite{CS,CSTkdv}, we then have the following.

\begin{lem}\label{ivpudkdv}
Suppose $J \le K$. If $(x_n)_{n\in\mathbb{Z}} \in \mathcal{X}^*_{J,K}$, then there exists a unique solution of \eqref{initial} with $F_{n,t}=F^{(J,K)}_{udK}$ for all $n,t$.
\end{lem}

\begin{proof}
In the case $J =K$, we have $ F^{(J,K)}_{udK}(x,u)=(u,x)$, and so the result is clear. In the case $J<K=\infty$, the result is given by \cite[Theorem 2.1]{CSTkdv}. For the case $J <K <\infty$, the result is given for BBS($J$,$K$) in \cite{CS}, i.e.\ for $J,K \in \N$ and $x \in \{0,1,2,\dots, J\}$, $u \in \{0,1,2,\dots, K\}$. The same proof applies in the more general case.
\end{proof}

To handle the case $\infty=J>K$, we consider the set
\begin{equation}\label{sync1}
\mathcal{X}^{!}_{J,K}:=\left\{ (x_n)_{ n\in\mathbb{Z}}:\: \limsup_{n \to -\infty}\mathbf{1}_{\{x_n +x_{n+1} \le K \}}=1 \right\},
\end{equation}
and for $\infty >J > K$, the set
\begin{equation}\label{sync2}
\mathcal{X}^{!}_{J,K}:=\left\{ (x_n)_{ n\in\mathbb{Z}}:\: \limsup_{n \to -\infty}\mathbf{1}_{\{x_n +x_{n+1} \le K \}\cup\{x_n +x_{n+1} \ge 2J-K \}}=1 \right\}.
\end{equation}
The subsequent result gives that if we start from a configuration within these sets, then it is not possible to give multiple definitions for the one time-step dynamics.

\begin{lem}\label{udkdvlem}
Suppose $J >K$. If $(x_n)_{ n\in\mathbb{Z}} \in \mathcal{X}^{!}_{J,K}$, then there exists at most one sequence $(u_n)_{ n\in\mathbb{Z}} $ such that
\begin{equation}\label{relation}
\left(F^{(J,K)}_{udk}\right)^{(2)}(x_n, u_{n-1})=u_n,\qquad \forall n\in\mathbb{Z}.
\end{equation}
\end{lem}
\begin{proof}
We first prove that if $x_n+x_{n+1} \le K$, then $u_{n+1}=x_{n+1}$. Since $J >K$ and
\[u_n=x_n-\max\{x_n+u_{n-1}-K, 0\}+\max\{x_n+u_{n-1}-J, 0\},\]
it must hold that $u_n \le x_n$. Hence $x_{n+1}+u_n \le x_{n+1}+x_n \le K$, and so
\[u_{n+1}=x_{n+1}-\max\{x_{n+1}+u_n-K, 0\}+\max\{x_{n+1}+u_n-J, 0\}=x_{n+1}.\]
Similarly, if $\infty >J > K$ and $x_n+x_{n+1} \ge 2J-K$, then $u_{n+1}=x_{n+1}-L$, where $L=J-K$. Indeed, since $u_n\geq x_n-L$, in this case we have that $x_{n+1}+u_n\geq 2J-K-L=J$, and the result follows. As a consequence, if $(x_n)_{ n\in\mathbb{Z}} \in \mathcal{X}^{!}_{J,K}$, then there exists a sequence $n_k \downarrow -\infty$ such that $u_{n_k}$ is determined by $x_{n_k}$. For $n \notin \{n_k:k\geq 1\}$, the relation \eqref{relation} means that $u_n$ is uniquely defined by $u_{n_k}$ such that $n_k < n $ and $(x_m)_{n_k+1 \le m \le n}$, and so the proof is complete.
\end{proof}

Putting together Theorem \ref{detailedbalancenew}, Proposition \ref{udkdvmeas}, Lemma \ref{ivpudkdv} and Lemma \ref{udkdvlem}, we complete this section by describing a number of invariant product measures for \eqref{UDKDV}. We write $\mathcal{T}_{udK}^{(J,K)}$ for the dynamics given by $F_{udK}^{(J,K)}$, as defined at \eqref{dynamics}.

\begin{thm}\label{udkdvinv} The product measure $\mu^\mathbb{Z}$ satisfies $\mathcal{T}_{udK}^{(J,K)}\mu^{\Z}=\mu^{\Z}$ for the following measures $\mu$.
\begin{enumerate}
  \item[(a)] Suppose $J=K$. Any measure $\mu$ on $\mathbb{R}$.
  \item[(b)] Suppose $J<K$. Excluding $\mu=\delta_{J/2}$, any measure $\mu$ given by Proposition \ref{udkdvmeas}(a) or Proposition \ref{udkdvmeas}(c).
  \item[(c)] Suppose $J>K$. Excluding $\mu=\delta_{K/2}$ and $\mu=\delta_{J-K/2}$, any measure $\mu$ given by Proposition \ref{udkdvmeas}(a) or supported on $(-\infty,\frac{K}{2}]$ or $[J-\frac{K}{2},\infty)$.
\end{enumerate}
\end{thm}
\begin{proof}
(a) The case $J=K$ is obvious.\\
(b) In the case $J<K=\infty$, for one of the measures $\mu^\mathbb{Z}$ from Proposition \ref{udkdvmeas} to satisfy $\mathcal{T}\mu^{\Z}=\mu^{\Z}$, it will suffice to check that $\mu^{\Z}(\mathcal{X}^*_{J,K})=1$. For this, the law of large numbers tells us that it is sufficient for $\int x\mu(dx)<J/2$. The measures given in the statement of the theorem are readily checked to satisfy this requirement. Finally, for $J<K<\infty$, it will again be enough to determine measures $\mu^\mathbb{Z}$ from Proposition \ref{udkdvmeas} that satisfy $\mu^{\Z}(\mathcal{X}^*_{J,K})=1$. The latter constraint simply rules out the trivial measure $\mu=\delta_{J/2}$, and so the result readily follows.\\
(c) Let us continue for the moment to suppose that $J<K$. We will appeal to the configuration-carrier duality of \eqref{ccdual} to prove the result, and as a first step we take $\mu$ to be one of the measures identified in part (b). If $(x_n^t,u_n^t)_{n,t\in\mathbb{Z}}$ is given by solving the initial value problem \eqref{initial} with initial condition $(x_n)_{n\in\mathbb{Z}}\sim\mu^\mathbb{Z}$, it then readily follows from Proposition \ref{Burke} that, for each $n\in\mathbb{Z}$, $(u_n^t)_{t\in\mathbb{Z}}$ is i.i.d., with marginal given by the corresponding $\nu$ from Proposition \ref{udkdvmeas}. Now, as long as $\nu((-\infty,\frac{J}{2}]\cup [K-\frac{J}{2},\infty))>0$, then it is clear that $\nu^\mathbb{Z}(\mathcal{X}_{K,J}^!)=1$. This means that, $\mu^\mathbb{Z}$-a.s., $(u_n^t)_{t\in\mathbb{Z}}$ uniquely determines $(x_{n+1}^t,u_{n+1}^t)_{t\in\mathbb{Z}}$, with $u_{n+1}=\mathcal{T}_{udK}^{(K,J)}u_{n}$, where $\mathcal{T}_{udK}^{(K,J)}$ represents the dynamics given by $F_{udK}^{(K,J)}$ (cf.\ the proof of Theorem \ref{ergodicthm}). In particular, we have demonstrated that $\mathcal{T}_{udK}^{(K,J)}\nu^\mathbb{Z}=\nu^\mathbb{Z}$. Reversing the role of $J$ and $K$ gives the result.
\end{proof}

\subsubsection{Ergodicity}

Finally, we study the ergodicity of the operator $\mathcal{T}_{udK}^{(J,K)}$. The next result is an immediate application of Theorem \ref{ergodicthm}, together with the observations we made in the proof of Theorem \ref{udkdvinv}, and so we simply state the conclusion.

\begin{thm} Suppose $J\leq K$. Let $\mu\times \nu$ be a product measure satisfying $F_{udK}^{(J,K)}(\mu\times \nu)=\mu\times \nu$, as given by Proposition \ref{udkdvmeas}, with $\mu\neq \delta_{J/2}$ and $\nu((-\infty,\frac{J}{2}]\cup [K-\frac{J}{2},\infty))>0$. It is then the case that $\mu^{\Z}$ is ergodic under $\mathcal{T}_{udK}^{(J,K)}$, and  $\nu^{\Z}$ is ergodic under $\mathcal{T}_{udK}^{(K,J)}$.
\end{thm}

\subsection{Discrete KdV equation}

\subsubsection{The model}

Our next model, the (modified) discrete KdV equation also incorporates two parameters, in this case given by $\alpha,\beta \ge 0$, and is based on the following lattice map:
\begin{align}
F^{(\alpha,\beta)}_{dK}(x,u)=\left(\frac{u(1+\beta xu)}{1+\alpha xu}, \frac{x(1+\alpha xu)}{1+\beta xu}\right),\tag{dKdV}\label{DKDV}
\end{align}
where we now assume the variables $x$ and $u$ are $(0,\infty)$ valued. We note that $F^{(\alpha,\beta)}_{dK}$ satisfies the Yang-Baxter relation, and may be derived from the 3d-consistency condition of the discrete potential KdV equation or the discrete BKP equation, see \cite{PTV, KNW}. Moreover, if $\beta=0$, then $F^{(\alpha,\beta)}_{dK}$ gives the discrete KdV equation.

\begin{rem}
Similarly to Remark \ref{symrem}, the lattice map \eqref{DKDV} admits a number of symmetries.
\begin{description}
\item[Involution] For any $(x,u)\in (0,\infty)^2$, it holds that
  \[F^{(\alpha,\beta)}_{dK}\circ F^{(\alpha,\beta)}_{dK}(x,u)=(x,u).\]
\item[Configuration-carrier duality] If $\pi(x,u):=(u,x)$, then
\[F^{(\alpha,\beta)}_{dK} = \pi  \circ F^{(\beta,\alpha)}_{dK} \circ \pi.\]
\item[Empty space-particle duality] Suppose  $\alpha,\beta>0$. If $\sigma_{\alpha,\beta}(x,u):=(\frac{1}{\alpha x},\frac{1}{\beta u})$, then
\[F^{(\alpha,\beta)}_{dK} = \sigma_{\alpha,\beta}  \circ F^{(\alpha,\beta)}_{dK} \circ \sigma_{\alpha,\beta}.\]
\item[Scale invariance] If $\lambda>0$, then for any $(x,u) \in  (0,\infty)^2$ it holds that
       \[F^{(\lambda^{-2}\alpha,\lambda^{-2}\beta)}_{dK} (\lambda x,\lambda u)=\lambda F^{(\alpha,\beta)}_{dK}(x,u).\]
\end{description}
We note that scale invariance in this setting corresponds to the shift invariance of \eqref{UDKDV}.
\end{rem}

\subsubsection{Detailed balance solutions}

For \eqref{DKDV}, we are unable to characterize the solutions of the detailed balance equation, even up to a technical condition as we did for \eqref{UDKDV}. Nonetheless, we are able to describe a family of solutions based on the GIG distribution. As we explain in Section \ref{linksec}, this family naturally corresponds to the stExp solutions of the \eqref{UDKDV} detailed balance equation, as presented in Proposition \ref{udkdvmeas}.

\begin{prop}\label{dkdvmeas}
The following product measures $\mu\times\nu$ satisfy $F^{(\alpha,\beta)}_{dK}(\mu\times\nu)=\mu\times\nu$.
\begin{enumerate}
\item[(a)] For any $\lambda\in\mathbb{R}$ if $\alpha\beta>0$, or $\lambda>0$ if $\alpha\beta=0$, and $c>0$,
\[\mu\times\nu=\mathrm{GIG}(\lambda,c\alpha,c)\times \mathrm{GIG}(\lambda,c\beta,c).\]
\item[(b)] Suppose $\alpha=\beta$. For any measure $m$ on $(0,\infty)$,
\[\mu\times\nu=m\times m.\]
\end{enumerate}
In the case $\alpha\beta=0$, there are no other non-trivial (i.e.\ non-Dirac measure) solutions to the detailed balance equation.
\end{prop}
\begin{proof}
(a) To verify the claim, given that absolute value of the associated Jacobian determinant of $F_{dK}^{(\alpha,\beta)}$ is equal to one, it suffices to check that the following relation between joint densities:
\[x^{-\lambda-1}e^{-c\alpha x-cx^{-1}} u^{-\lambda-1}e^{-c\beta u-cu^{-1}} = y^{-\lambda-1}e^{-c\alpha y-cy^{-1}} v^{-\lambda-1}e^{-c\beta v-c v^{-1}},\]
where $y=\frac{u(1+\beta xu)}{1+\alpha xu}$ and $v=\frac{x(1+\alpha xu)}{1+\beta xu}$. This is a simple consequence of the identities $xu=yv$ and $\alpha x+ x^{-1} + \beta u+ u^{-1}=\alpha  y+y^{-1}+\beta  v+v^{-1}$, which can be checked directly. \\
(b) Since $F^{(\alpha,\alpha)}_{dK}(x,u)=(u,x)$, the result is obvious.\\
For the final part of the result, suppose $\alpha>0=\beta$. In this case, the map of interest becomes
\[F^{(\alpha,0)}_{dK}(x,u)=\left(\frac{u}{1+\alpha xu}, x(1+\alpha xu)\right)=\left(\frac{1}{\alpha x +u^{-1}}, \alpha^{-1}\left(\frac{1}{\alpha x}-\frac{1}{\alpha x +u^{-1}}\right)^{-1}\right).\]
Now, in \cite[Theorem 4.1]{LW}, it is shown that if $X$ and $Y$ are strictly positive independent random variables such that at least one of $X$ and $Y$ has a non-trivial distribution, and $(X+Y)^{-1}$ and $X^{-1}-(X+Y)^{-1}$ are also independent, then $X$ must have a generalized inverse Gaussian distribution and $Y$ must have a gamma distribution with related parameters. (NB. This result builds on \cite{LS}.) Considering the form of the map $F^{(\alpha,0)}_{dK}$ as given above, and applying \cite[Theorem 4.1]{LW} with $X=\alpha x$, $Y=u^{-1}$ yields the result.
\end{proof}

\subsubsection{Invariant measures}

We now show how the measures of Proposition \ref{dkdvmeas} yield invariant product measures for $\mathcal{T}_{dK}^{\alpha,\beta}$, that is, the operator describing the \eqref{DKDV} dynamics. Apart from the trivial case $\alpha=\beta$, we restrict our attention to the case when $\alpha\beta=0$. (We list the case $\alpha\beta>0$ amongst the open problems in Section \ref{oqsec}.) The reason for this is that it will allow the application of the path encoding results from \cite{CSTkdv} concerning the initial value problem \eqref{initial}. In particular, consider the latter problem with $F_{n,t}=F^{(\alpha,\beta)}_{dK}$ for all $n,t$, where $\alpha>0$ and $\beta=0$. Letting
\[\mathcal{X}^*_{\alpha}:=\left\{ (x_n)_{n\in\mathbb{Z}}\in (0,\infty)^{\Z}:\: \lim_{|n| \to \infty}\frac{\sum_{k=1}^n (- \log \alpha-2 \log x_k)}{n} >0  \right\},\]
we have the following result (see \cite[Theorem 2.2]{CSTkdv}).

\begin{lem}\label{dkdvsol} Suppose $\alpha>0$. If $(x_n)_{n\in\mathbb{Z}} \in \mathcal{X}^*_{\alpha}$, then there exists a unique solution of \eqref{initial} with $F_{n,t}=F^{(\alpha,0)}_{dK}$ for all $n,t$.
\end{lem}

In the case $\alpha=0$, $\beta>0$, we consider the set
\[\mathcal{X}^{\exists!}_{\beta}:=\left\{ (x_n)_{n\in\mathbb{Z}} \in (0,\infty)^{\Z}:\: \sum_{n=-\infty}^0x_n^{-1}=\lim_{n\rightarrow\infty}-S_{-n}=\infty,\:\lim_{n\rightarrow-\infty}(\log x_n)S_{n}^{-1}=0 \right\},\]
where $S_{-n}:=\sum_{m=-n+1}^0(- \log \beta-2 \log x_m)$. The parallel to Lemma \ref{udkdvlem} that we apply in the discrete setting is the following.

\begin{lem}\label{contfracsol}
Suppose $\alpha=0$, $\beta>0$. If $(x_n)_{n\in\mathbb{Z}} \in \mathcal{X}^{\exists!}_{\beta}$, then there exists precisely one sequence $(u_n)_{n\in\mathbb{Z}}\in (0,\infty)^{\Z}$ such that
\begin{equation}\label{one-step}
\left(F^{(0,\beta)}_{dK}\right)^{(2)}(x_n, u_{n-1})=u_n, \qquad \forall n\in\mathbb{Z},
\end{equation}
which is explicitly given by the infinite continued fraction
\[u_n= \frac{1}{\sqrt{\beta}} \frac{1}{(\sqrt{\beta} x_{n})^{-1}+ \frac{1}{(\sqrt{\beta} x_{n-1})^{-1}+ \dots}}.\]
\end{lem}
\begin{proof}
The relation  \eqref{one-step} can be written as
\[u_n=\frac{x_n}{1+\beta x_n u_{n-1}},\]
which is equivalent to
\[\sqrt{\beta} u_n= \frac{1}{(\sqrt{\beta} x_n)^{-1}+\sqrt{\beta} u_{n-1}}.\]
Hence, the sequence defined by
\[\sqrt{\beta}  u_n= \frac{1}{(\sqrt{\beta} x_{n})^{-1}+ \frac{1}{(\sqrt{\beta} x_{n-1})^{-1}+ \dots}}\]
satisfies \eqref{one-step}. Indeed, the condition $\sum_{n=-\infty}^0x_n^{-1}=\infty$ ensures that the infinite continued fraction converges in $(0,\infty)$ (see \cite[Chapter 8]{Loya}, for example). Suppose that we have another solution $(\tilde{u}_n)_{n\in\mathbb{Z}}\in (0,\infty)^{\Z}$ to \eqref{one-step}. It is then the case that
\[\left|u_n-\tilde{u}_n\right|=u_n\tilde{u}_n\left|u_{n}^{-1}-\tilde{u}_{n}^{-1}\right|=\beta u_{n}\tilde{u}_{n}\left|u_{n-1}-\tilde{u}_{n-1}\right|\leq\beta x_{n}^2\left|u_{n-1}-\tilde{u}_{n-1}\right| .\]
Iterating this, we find that for any $m\leq n$,
\[\left|u_n-\tilde{u}_n\right|\leq \prod_{k=m}^n(\beta x_k^2)\times x_{m-1}=\exp\left(\sum_{k=m}^n(\log \beta+2\log x_k)+\log x_{m-1}\right).\]
Taking the limit as $m\rightarrow-\infty$, the defining properties of $\mathcal{X}^{\exists!}_{\beta}$ imply that $u_n=\tilde{u}_n$, as desired.
\end{proof}

Arguing as for Theorem \ref{udkdvinv}, we have that Theorem \ref{detailedbalancenew}, Proposition \ref{dkdvmeas} and Lemmas \ref{dkdvsol} and \ref{contfracsol} yield the subsequent result. For the proof of part (c) of the result, the one additional useful observation is that if $\mu=GIG(\lambda,c\alpha,c)$ and $\nu=IG(\lambda,c)$, then $2\int \log(x)\mu(dx)\leq 2\int \log(x)\nu(dx)$ (this ensures that the given condition is enough to ensure that both marginals of the solution to the relevant detailed balance equation satisfy the required logarithmic moment bound).

\begin{thm} The product measure $\mu^\mathbb{Z}$ satisfies $\mathcal{T}_{dK}^{(\alpha,\beta)}\mu^{\Z}=\mu^{\Z}$ for the following measures $\mu$.
\begin{enumerate}
  \item[(a)] Suppose $\alpha=\beta$. Any measure $\mu$ on $\mathbb{R}$.
  \item[(b)] Suppose $\alpha>0$, $\beta=0$. The measure $\mu=GIG(\lambda,c\alpha,c)$ for any parameters $\lambda,c>0$ such that $2\int \log(x)\mu(dx)<-\log\alpha$.
  \item[(c)] Suppose $\alpha=0$, $\beta>0$. The measure $\mu=IG(\lambda,c)$ for any parameters $\lambda,c>0$ such that $2\int \log(x)\mu(dx)<-\log\beta$.
\end{enumerate}
\end{thm}

\subsubsection{Ergodicity}

Regarding the ergodicity of $\mathcal{T}_{dK}^{(\alpha,\beta)}$, combining the results of the previous section with Theorem \ref{ergodicthm} gives the next result.

\begin{thm} Suppose $\alpha>0$, $\beta=0$. Let $\mu\times \nu$ be a product measure satisfying $F_{dK}^{(\alpha,0)}(\mu\times \nu)=\mu\times \nu$, as given by Proposition \ref{dkdvmeas} (i.e.\ $\mu\times\nu = GIG(\lambda,c\alpha,c)\times IG (\lambda,c)$. If it holds that $2\int \log(x)\nu(dx)<-\log\alpha$, it is then the case that $\mu^{\Z}$ is ergodic under $\mathcal{T}_{dK}^{(\alpha,0)}$, and  $\nu^{\Z}$ is ergodic under $\mathcal{T}_{dK}^{(0,\alpha)}$.
\end{thm}

\section{Type II examples: Toda-type discrete integrable systems}\label{todasec}

The type II examples that we study arise from two other important discrete integrable systems, namely the discrete and ultra-discrete Toda equations. Again, see [7, 15] and the references therein for background. As in the previous section, our aim is to identify solutions of the corresponding detailed balance equations and invariant measures. For type II systems, we do not have a strategy for checking ergodicity.

\subsection{Ultra-discrete Toda equation}

\subsubsection{The model} The ultra-discrete Toda equation is described as follows:
    \begin{equation}
    \begin{cases}
    Q_{n}^{t+1}=\min \{U_{n}^t,E_n^t\}, \\
    E_{n}^{t+1}=Q_{n+1}^t+E_{n}^t-Q_{n}^{t+1},\\
    U_{n+1}^t=U_{n}^t+Q_{n+1}^t-Q_{n}^{t+1},
    \end{cases}
    \tag{udToda}\label{UDTODA}
\end{equation}
where $(Q_n^t,E_n^t,U_n^t)_{n,t\in\mathbb{Z}}$ take values in $\mathbb{R}$. We summarise this evolution as $(Q_{n}^{t+1},E_n^{t+1},U_{n+1}^t)=F_{udT}(Q_{n+1}^t,E_n^t,U_{n}^t)$, highlighting that $F_{udT}$ is an involution on $\mathbb{R}^3$, and represent the lattice structure diagrammatically as
    \begin{equation}\label{qelattice}
    \xymatrix@C-15pt@R-15pt{ & Q_n^{t+1} & E_n^{t+1}& \\
            U_n^{t} \ar[rrr] & && U_{n+1}^t.\\
             & E_n^{t}\ar[uu]& Q_{n+1}^{t}\ar[uu]&}
    \end{equation}
Whilst this system might not immediately appear to link with \eqref{UDKDV} or the BBS, we note that if we restrict to non-negative integer-valued variables, and view $Q_{n}^{t}$ as the length of the $n$th interval containing balls, $E_n^t$ as the length of the $n$th empty interval (at time $t$), and $U_n^t$ as the carrier load at the relevant lattice location, then the dynamics of these variables coincides with that given by the BBS. (In the case of infinite balls, there is an issue of how to enumerate the intervals.) Moreover, although the lattice structure at \eqref{qelattice} does not immediately fit into our general framework, it is possible to decompose the single map $F_{udT}$ with three inputs and three outputs into two maps $F_{udT^*}$ and $F_{udT^*}^{-1}$, each with two inputs and two outputs:
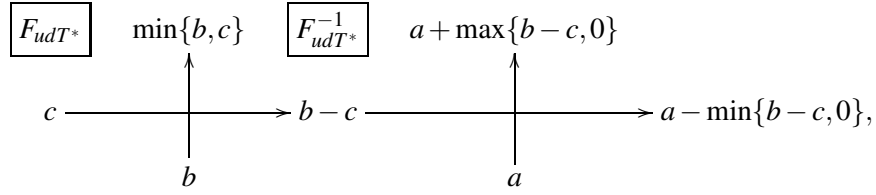
\[\hspace{40pt}\xymatrix@C-15pt@R-15pt{\boxed{F_{udT^*}\vphantom{F_{udT^*}^{-1}}} & \min\{b,c\} &\boxed{F_{udT^*}^{-1}}& a+\max\{b-c,0\}&\\
c \ar[rr] && b-c\ar[rr]&& a-\min\{b-c,0\},\\
& b \ar[uu]&&a\ar[uu]&}\]
where we generically take $(a,b,c)=(Q_{n+1}^t,E_n^t,U_{n}^t)$. Including the additional lattice variables, we can thus view the system as type II locally-defined dynamics, as defined in the introduction, with the maps alternating between the bijection $F_{udT^*} : \R^2 \to \R^2$ and its inverse, which are given explicitly by
\[F_{udT^*}(x,u)=\left(\min\{x,u\},x-u\right), \quad  F_{udT^*}^{-1}(x,u)=\left(x+\max\{u,0\},x-\min\{u,0\}\right).\]
Note that the decomposition of $F_{udT}$ into $F_{udT^*}$ and $F_{udT^*}^{-1}$ is not unique. The form of $F_{udT^*}$ chosen here is slightly simpler than the corresponding map in \cite{CSTkdv} (see also \cite{CS4}), since we do not need to satisfy the additional constraint that yields a ‘Pitman-type transformation map’.

\subsubsection{Detailed balance solutions} For $F_{udT^*}$, we are able to completely solve the detailed balance equation, see Proposition \ref{UDTODAmeas}. In the subsequent result, Proposition \ref{udTodainv}, we show how this yields a complete solution to the corresponding problem for $F_{udT}$.

\begin{prop}\label{UDTODAmeas} The following measures $\mu,\nu,\tilde{\mu},\tilde{\nu}$ satisfy $F_{udT^*}(\mu \times \nu)=\tilde{\mu}\times \tilde{\nu}$.
\begin{enumerate}
\item[(a)] For any $\lambda_1, \lambda_2 >0$ and $c \in \R$,
\[\mu=\mathrm{sExp}(\la_1,c),\quad \nu=\mathrm{sExp}(\la_2,c),\quad \tilde{\mu} =\mathrm{sExp}(\la_1+\la_2,c),\quad \tilde{\nu}= \mathrm{AL} (\la_1,\la_2).\]
\item[(b)] For any $\theta_1, \theta_2  \in (0,1)$, $M \in \Z$ and $m \in (0,\infty)$,
\[\mu=\mathrm{ssGeo}(1-\theta_1,M,m),\quad\nu=\mathrm{ssGeo}(1-\theta_2,M, m),\]
\[\tilde{\mu}=\mathrm{ssGeo}(1-\theta_1\theta_2,M,m),\quad\tilde{\nu}= \mathrm{sdAL}(1-\theta_1,1-\theta_2,m).\]
\item[(c)] For any $c_1,c_2\in\mathbb{R}$ and measure $m$ supported on $[0,\infty)$,
\begin{enumerate}
\item[(i)] $\mu=\delta_{c_1}$, $\nu=\delta_{c_2}$, $\tilde{\mu}=\delta_{\min\{c_1,c_2\}}$, $\tilde{\nu}=\delta_{c_1-c_2}$,
\item[(ii)] $\mu=\delta_{c_1}$, $\nu=m(\cdot -c_1)$, $\tilde{\mu}=\delta_{c_1}$, $\tilde{\nu}=m(-\cdot)$,
\item[(iii)] $\mu=m(\cdot -c_1)$, $\nu=\delta_{c_1}$, $\tilde{\mu}=\delta_{c_1}$, $\tilde{\nu}=m$.
\end{enumerate}
NB.\ Case (c)(i) is contained in cases (c)(ii) and (c)(iii).
\end{enumerate}
It is further the case that there are no other quadruples of probability measures $(\mu,\nu,\tilde{\mu},\tilde{\nu})$ that satisfy $F_{udT^*}(\mu \times \nu)=\tilde{\mu}\times \tilde{\nu}$.
\end{prop}
\begin{proof}
The first part follows by direct computation. The uniqueness claim relies on a well-known fact \cite{Fe1,Fe2,Cr} about exponential and geometric distributions. Namely, suppose that $X$ and $Y$ are two non-constant, independent random variables. It is then the case that $\min\{X,Y\}$ and $X-Y$ are independent if and only if $X$ and $Y$ are
sExp-distributed random variables with the same location parameter or ssGeo-distributed random variables with the same location and scale parameters. The trivial solutions of part (c) are covered by \cite[Theorem 1 (and the following comment)]{Fe1}.
\end{proof}

By construction, we have that
\[F_{udT}(a,b,c) = \left( F_{udT^*}^{(1)} (b,c), F_{udT^*}^{-1} \left(a,F_{udT^*}^{(2)} (b,c)\right)\right).\]
This enables us to deduce from Propositions \ref{twothree} and \ref{UDTODAmeas} the subsequent result.

\begin{prop}\label{udTodainv}
The following product measures $\tilde{\mu}\times\mu\times\nu$ satisfy $F_{udT}(\tilde{\mu}\times\mu\times\nu)=\tilde{\mu}\times\mu\times\nu$.
\begin{enumerate}
\item[(a)] For any $\lambda_1, \lambda_2 >0$ and $c \in \R$,
\[\tilde{\mu}\times\mu\times\nu=\mathrm{sExp}(\la_1+\la_2,c) \times \mathrm{sExp}(\la_1,c) \times \mathrm{sExp}(\la_2,c).\]
\item[(b)] For any $\theta_1, \theta_2  \in (0,1)$, $M \in \Z$ and $m \in (0,\infty)$,
\[\tilde{\mu}\times\mu\times\nu=\mathrm{ssGeo}(1-\theta_1\theta_2,M,m) \times \mathrm{ssGeo}(1-\theta_1,M,m) \times \mathrm{ssGeo}(1-\theta_2,M, m).\]
\item[(c)] For any $c_1,c_2\in\mathbb{R}$ and measure $m$ supported on $[0,\infty)$,
\begin{enumerate}
\item[(i)] $\tilde{\mu}\times\mu\times \nu=\delta_{\min\{c_1,c_2\}}\times\delta_{c_1}\times\delta_{c_2}$,
\item[(ii)] $\tilde{\mu}\times\mu\times \nu=\delta_{c_1}\times\delta_{c_1}\times m(\cdot -c_1)$,
\item[(iii)] $\tilde{\mu}\times\mu\times \nu=\delta_{c_1}\times m(\cdot -c_1)\times \delta_{c_1}$.
\end{enumerate}
NB.\ Again, case (c)(i) is contained in cases (c)(ii) and (c)(iii).
\end{enumerate}
Moreover, if a product measure is invariant under $F_{udT}$, then it must be one of the above.
\end{prop}
%DC - could also note F_{udT^*}^{(2)} (b,c) is the `intermediate' carrier variable.
\begin{proof}
The first part follows directly from Propositions \ref{twothree} and \ref{UDTODAmeas}. To show uniqueness, let $X,Y$ and $Z$ be independent random variables satisfying
\[F_{udT}(X,Y,Z) \buildrel{d}\over{=}(X,Y,Z).\]
Let $W:=F_{udT^*}^{(2)} (Y,Z)$, then, by assumption,
\[F_{udT^*}^{-1}(X, W) =F_{udT^*}^{-1}\left(X,  F_{udT^*}^{(2)} (Y,Z)\right)  \buildrel{d}\over{=} (Y,Z).\]
Hence $ F_{udT^*}(Y,Z) \buildrel{d}\over{=} (X, W)$.
Since $X,Y,Z$ are independent, $X$ and $W$ are independent. Therefore the marginals of $(Y,Z,X,W)$ must be given by one of the collections $(\mu,\nu,\tilde{\mu},\tilde{\nu})$ described in Proposition \ref{UDTODAmeas}.
\end{proof}

\subsubsection{Invariant measures} The initial value problem for the ultra-discrete Toda equation that we consider is: for $(Q^0,E^0)\in(\mathbb{R}^2)^\mathbb{Z}$, find $(Q^t_n,E^t_n,U^t_n)_{n,t\in\mathbb{Z}}$ such that \eqref{UDTODA} holds for all $n,t$. This was solved in \cite{CSTkdv} for initial conditions in the set
\[\mathcal{X}_{udT}:=\left\{(Q,E)\in(\mathbb{R}^2)^\mathbb{Z}\::\:
 \begin{array}{l}
   \lim_{n \to \infty}\frac{\sum_{m=1}^{n}(Q_m-E_m)}{n} =  \lim_{n \to \infty}\frac{\sum_{m=1}^{n}(Q_m-E_m)+Q_{n+1}}{n} <0,\\
\lim_{n \to -\infty}\frac{\sum_{m=1}^{n}(Q_m-E_m)}{n}=  \lim_{n \to -\infty}\frac{\sum_{m=1}^{n}(Q_m-E_m)+E_{n}}{n} <0
 \end{array}\right\}.\]
In particular, the subsequent result was established.

\begin{lem}[{\cite[Theorem 2.3]{CSTkdv}}]\label{udtodasol} If $(Q^0,E^0)\in\mathcal{X}_{udT}$, then there exists a unique collection $(Q^t_n,E^t_n,U^t_n)_{n,t\in\mathbb{Z}}$ such that \eqref{UDTODA} holds for all $n,t$.
\end{lem}

In the case when a unique solution to \eqref{UDTODA} exists, it makes sense to define the dynamics of the system similarly to \eqref{dynamics}, i.e.\ set
\[\mathcal{T}_{udT}(Q^0,E^0):=(Q^1,E^1).\]
In what is the main result of this section, we characterize invariant product measures for the resulting evolution.

\begin{thm}\label{udtodathm} Suppose that $(Q^0_n,E^0_n)_{n\in\mathbb{Z}}$ is an i.i.d.\ sequence with marginal given by $\tilde{\mu}\times {\mu}$, where one of the following holds:
\begin{enumerate}
\item[(a)] for some $\lambda_1, \lambda_2 >0$ and $c \in \R$,
\[\tilde{\mu}\times\mu=\mathrm{sExp}(\la_1+\la_2,c) \times \mathrm{sExp}(\la_1,c);\]
\item[(b)] for some $\theta_1, \theta_2  \in (0,1)$, $M \in \Z$ and $m \in (0,\infty)$,
\[\tilde{\mu}\times\mu=\mathrm{ssGeo}(1-\theta_1\theta_2,M,m) \times \mathrm{ssGeo}(1-\theta_1,M,m);\]
\item[(c)] for some $c\in\mathbb{R}$ and measure $m$ supported on $[c,\infty)$ with $m\neq \delta_{c}$,
\[\tilde{\mu}\times\mu=\delta_{c}\times m(\cdot -c).\]
\end{enumerate}
It is then the case that $\mathcal{T}_{udT}(Q^0,E^0)\buildrel{d}\over=(Q^0,E^0)$. Moreover, there are no other non-trivial measures such that $(Q^0_n,E^0_n)_{n\in\mathbb{Z}}$ is an i.i.d.\ sequence, with $Q_n^0$ independent of $E_n^0$, and $\mathcal{T}_{udT}(Q^0,E^0)\buildrel{d}\over=(Q^0,E^0)$.
\end{thm}
\begin{proof} If $(Q^0_n,E^0_n)_{n\in\mathbb{Z}}$ is an i.i.d.\ sequence with marginal $\tilde{\mu}\times {\mu}$ of one of the given forms, then it is a simple application of the law of large numbers to check that, $(\tilde{\mu}\times {\mu})^\mathbb{Z}$-a.s., $(Q^0,E^0)\in\mathcal{X}_{udT}$. It readily follows from Lemma \ref{udtodasol} that, $(\tilde{\mu}\times {\mu})^\mathbb{Z}$-a.s., the corresponding type II lattice equations have a unique solution with initial condition $(x_n)_{n\in\mathbb{Z}}$, where $x_{2n}:=E^0_n$ and $x_{2n+1}:=Q^0_{n+1}$. Thus we can apply Theorem \ref{detailedbalancenew} and Proposition \ref{UDTODAmeas} to deduce the result.
\end{proof}

\subsection{Discrete Toda equation}

\subsubsection{The model} The discrete Toda equation is given by:
    \begin{equation}
    \begin{cases}
    I_n^{t+1}=J_n^t+U_n^t,\\
    J_n^{t+1}={I_{n+1}^{t}J_n^{t}}(I_n^{t+1})^{-1},\\
    U_{n+1}^t={I_{n+1}^{t}U_n^{t}}(I_n^{t+1})^{-1}.
    \end{cases}
    \tag{dToda}\label{DTODA}
    \end{equation}
    Here, the variables $(I_n^t,J_n^t,U_n^t)_{n,t\in\mathbb{Z}}$ take values in $(0,\infty)$, and we can summarise the above dynamics by $(I_{n}^{t+1},J_n^{t+1},U_{n+1}^t)=F_{dT}(I_{n+1}^t,J_n^t,U_{n}^t)$, where $F_{dT}$ is an involution on $(0,\infty)^3$. Similarly to \eqref{qelattice}, in this case we have a lattice structure
    \[\xymatrix@C-15pt@R-15pt{ &I_n^{t+1}&J_n^{t+1} &\\
            U_n^{t} \ar[rrr]& & & U_{n+1}^{t},\\
     & J_n^t \ar[uu]&I_{n+1}^t\ar[uu]&}\]
     which can be decomposed into two maps, $F_{dT^*}$ and $F_{dT^*}^{-1}$, as follows:
     \[\xymatrix@C-15pt@R-15pt{\boxed{F_{dT^*}\vphantom{F_{udT^*}^{-1}}} & b+c &\boxed{F_{dT^*}^{-1}}& \frac{ab}{b+c}&\\
            c \ar[rr] &&\frac{b}{b+c} \ar[rr]&& \frac{ac}{b+c},\\
             & b \ar[uu]&&a\ar[uu]&}\]
     where we generically take $(a,b,c)=(I_{n+1}^t,J_n^t,U_{n}^t)$. So, again including the additional lattice variables, we can view the system as type II locally-defined dynamics, as defined in the introduction, with the maps alternating between the bijection $F_{dT^*}:(0,\infty)^2\rightarrow(0,\infty)^2$ and its inverse, which are given explicitly by:
\[F_{dT^*}(x,y)=\left(x+y,\frac{x}{x+y}\right), \qquad F_{dT^*}^{-1}(x,y)=\left(xy,x(1-y)\right).\]
As in the ultra-discrete case, we note that the decomposition of $F_{dT}$ into $F_{dT^*}$ and $F_{dT^*}^{-1}$ is not unique, with the form of $F_{dT^*}$ chosen here being slightly simpler than the corresponding map in \cite{CSTkdv} (see also \cite{CS4}).

\subsubsection{Detailed balance solutions} As in the ultra-discrete case, we are also able to completely solve the detailed balance equation for $F_{dT^*}$, see Proposition \ref{DTODAmeas}. In the subsequent result, Proposition \ref{dTodainv}, we apply this to deduce a complete solution to the corresponding problem for $F_{dT}$.

\begin{prop}\label{DTODAmeas} The following measures $\mu,\nu,\tilde{\mu},\tilde{\nu}$ satisfy $F_{dT^*}(\mu \times \nu)=\tilde{\mu}\times \tilde{\nu}$.

\begin{enumerate}
\item[(a)] For any $\lambda_1, \lambda_2 >0$ and $c>0$,
\[\mu=\mathrm{Gam}(\la_1,c),\quad \nu=\mathrm{Gam}(\la_2,c),\quad \tilde{\mu} =\mathrm{Gam}(\la_1+\la_2,c),\quad \tilde{\nu}= \mathrm{Be} (\la_1,\la_2).\]
\item[(b)] For any $c_1,c_2\in(0,\infty)$, $\mu=\delta_{c_1}$, $\nu=\delta_{c_2}$, $\tilde{\mu}=\delta_{c_1+c_2}$, $\tilde{\nu}=\delta_{c_1/(c_1+c_2)}$.
\end{enumerate}
It is further the case that there are no other quadruples of probability measures $(\mu,\nu,\tilde{\mu},\tilde{\nu})$ that satisfy $F_{dT^*}(\mu \times \nu)=\tilde{\mu}\times \tilde{\nu}$.
\end{prop}
\begin{proof}
The first part follows by direct computation. The uniqueness relies on a well-known fact \cite{L} about gamma distributions. Namely, suppose that $X$ and $Y$ are two non-constant, independent, positive random variables. Then $X+Y$ and $\frac{X}{X+Y}$ are independent if and only if $X$ and $Y$ are gamma-distributed random variables with the same scale parameter. Applying the fact that $X$ and $1/X$ are independent if and only if $X$ is a constant random variable, the trivial solutions of part (b) are readily checked to be the only other option.
\end{proof}

In this case, by construction, we have that
\[F_{dT}(a,b,c) := \left( F_{dT^*}^{(1)} (b,c), F_{dT^*}^{-1} \left(a,F_{dT^*}^{(2)} (b,c)\right)\right).\]
This enables us to deduce from Propositions \ref{twothree} and \ref{DTODAmeas} the following result.

\begin{prop}\label{dTodainv}
The following product measures $\tilde{\mu}\times\mu\times\nu$ satisfy $F_{dT}(\tilde{\mu}\times\mu\times\nu)=\tilde{\mu}\times\mu\times\nu$.
\begin{enumerate}
\item[(a)] For any $\lambda_1, \lambda_2 >0$ and $c >0$,
\[\tilde{\mu}\times\mu\times\nu=\mathrm{Gam}(\la_1+\la_2,c) \times \mathrm{Gam}(\la_1,c) \times \mathrm{Gam}(\la_2,c).\]
\item[(b)] For any $c_1,c_2\in(0,\infty)$,
\[\tilde{\mu}\times\mu\times\nu=\delta_{c_1+c_2} \times \delta_{c_1} \times \delta_{c_2}.\]
\end{enumerate}
Moreover, if a product measure is invariant under $F_{dT}$, then it must be one of the above.
\end{prop}
\begin{proof}
The proof is same as that of Proposition \ref{udTodainv}.
\end{proof}

\subsubsection{Invariant measures} The initial value problem for the discrete Toda equation that we consider is: for $(I^0,J^0)\in((0,\infty)^2)^\mathbb{Z}$, find $(I^t_n,J^t_n,U^t_n)_{n,t\in\mathbb{Z}}$ such that \eqref{DTODA} holds for all $n,t$. This was solved in \cite{CSTkdv} for initial conditions in the set
\begin{align*}
\lefteqn{\mathcal{X}_{dT}:=}\\
&\left\{(I,J)\in((0,\infty)^2)^\mathbb{Z}\::\:
 \begin{array}{l}
   \lim_{n \to \infty}\frac{\sum_{m=1}^{n}(\log J_m-\log I_m)}{n} =  \lim_{n \to \infty}\frac{\sum_{m=1}^{n}(\log J_m-\log I_m)-\log I_{n+1}}{n} <0,\\
\lim_{n \to -\infty}\frac{\sum_{m=1}^{n}(\log J_m-\log I_m)}{n}=  \lim_{n \to -\infty}\frac{\sum_{m=1}^{n}(\log J_m-\log I_m)-\log J_{n}}{n} <0
 \end{array}\right\}.
 \end{align*}
In particular, the following result was established.

\begin{lem}[{\cite[Theorem 2.5]{CSTkdv}}] If it holds that $(I^0,J^0)\in\mathcal{X}_{dT}$, then there exists a unique collection $(I^t_n,J^t_n,U^t_n)_{n,t\in\mathbb{Z}}$ such that \eqref{DTODA} holds for all $n,t$.
\end{lem}

As in the ultra-discrete case, in the case when a unique solution to \eqref{DTODA} exists, it makes sense to define the dynamics of the system similarly to \eqref{dynamics}, i.e.\ set
\[\mathcal{T}_{dT}(I^0,J^0):=(I^1,J^1).\]
In what is the main result of this section, we characterize invariant product measures for the resulting evolution.

\begin{thm} Suppose that $(I^0_n,J^0_n)_{n\in\mathbb{Z}}$ is an i.i.d.\ sequence with marginal given by $\tilde{\mu}\times {\mu}$, where the following holds: for some $\lambda_1, \lambda_2 >0$ and $c \in \R$,
\[\tilde{\mu}\times\mu=\mathrm{Gam}(\la_1+\la_2,c) \times \mathrm{Gam}(\la_1,c).\]
It is then the case that $\mathcal{T}_{dT}(I^0,J^0)\buildrel{d}\over=(I^0,J^0)$.  Moreover, there are no other non-trivial measures such that $(I^0_n,J^0_n)_{n\in\mathbb{Z}}$ is an i.i.d.\ sequence, with $I_n^0$ independent of $J_n^0$,  and $\mathcal{T}_{dT}(I^0,J^0)\buildrel{d}\over=(I^0,J^0)$.
\end{thm}
\begin{proof} The proof is the same as that of Theorem \ref{udtodathm}.
\end{proof}

\section{Links between discrete integrable systems}\label{linksec}

In this section, we explain how the well-known links between the systems \eqref{UDKDV}, \eqref{DKDV}, \eqref{UDTODA} and \eqref{DTODA} extend to invariant measures. Our results are summarised in Figure \ref{dis}.

\begin{figure}
\centerline{\xymatrix{\framebox(150,45){\parbox{140pt}{\centering{Discrete KdV $(\alpha,\beta)$: $\mathrm{GIG}(\lambda,c\alpha,c)\times$\\$\mathrm{GIG}(\lambda,c\beta,c)$}}}\ar[dddd]|<<<<<<<<<<<<<<<<<{\mbox{
\begin{tabular}{c}
  \small{\emph{Ultra-discretization:}} \\
  $\lambda(\varepsilon) = \lambda\varepsilon$\\
  $c(\varepsilon)=e^{c/\varepsilon}$\\
  $\alpha(\varepsilon)=e^{-J/\varepsilon}$\\
  $\beta(\varepsilon)=e^{-K/\varepsilon}$
\end{tabular}}}&\hspace{30pt}&\framebox(150,45){\parbox{140pt}{\centering{Discrete Toda:\\ $\mathrm{Gam}(\la_1+\la_2,c)\times$\\$\mathrm{Gam}(\la_1,c) \times \mathrm{Gam}(\la_2,c)$ }}}\ar[dddd]|<<<<<<<<<<<<<<<<<{\mbox{
\begin{tabular}{c}
  \small{\emph{Ultra-discretization:}} \\
 $\lambda_1(\varepsilon) = \lambda_1\varepsilon$\\
 $\lambda_2(\varepsilon) = \lambda_2\varepsilon$\\
  $c(\varepsilon)=e^{c/\varepsilon}$\\
\end{tabular}}}\ar@{-->}@<-8ex>[ll]^{\mbox{\begin{tabular}{c}
  \small{\emph{Self-convolution:}} \\
 $\beta=0$,\\
 $(\lambda,c\sqrt{\alpha})\leftrightarrow(\lambda_2,c)$
\end{tabular}}}\\
&&\\
&&\\
&&\\
\framebox(150,45){\parbox{140pt}{\centering{Ultra-discrete KdV $(J,K)$: $\mathrm{stExp}(\lambda,c,J-c)\times \mathrm{stExp}(\lambda,c,K-c)$}}}\ar@{<--}@<0ex>[rr]^{\mbox{\begin{tabular}{c}
  \small{\emph{Self-convolution:}} \\
 $K=\infty$,\\
 $(\lambda,c-\frac{J}{2})\leftrightarrow(\lambda_2,c)$
\end{tabular}}}&&\framebox(150,45){\parbox{140pt}{\centering{Ultra-discrete Toda:\\$\mathrm{sExp}(\la_1+\la_2,c)\times$\\ $\mathrm{sExp}(\la_1,c) \times \mathrm{sExp}(\la_2,c)$}}}}}
\caption{Links between some of the product invariant measures of Propositions \ref{udkdvmeas}, \ref{dkdvmeas}, \ref{udTodainv} and \ref{dTodainv}, as discussed in Section \ref{linksec}. In particular, the two solid arrows are essentially given by the weak convergence statements of Proposition \ref{udprop}, see Remark \ref{udrem}. The two dashed arrows indicate how particular conditionings of the invariant measures for the Toda-type systems give rise to the invariant measures for the KdV-type systems, see Subsection \ref{kdvtodalinksubsec} for details.}\label{dis}
\end{figure}
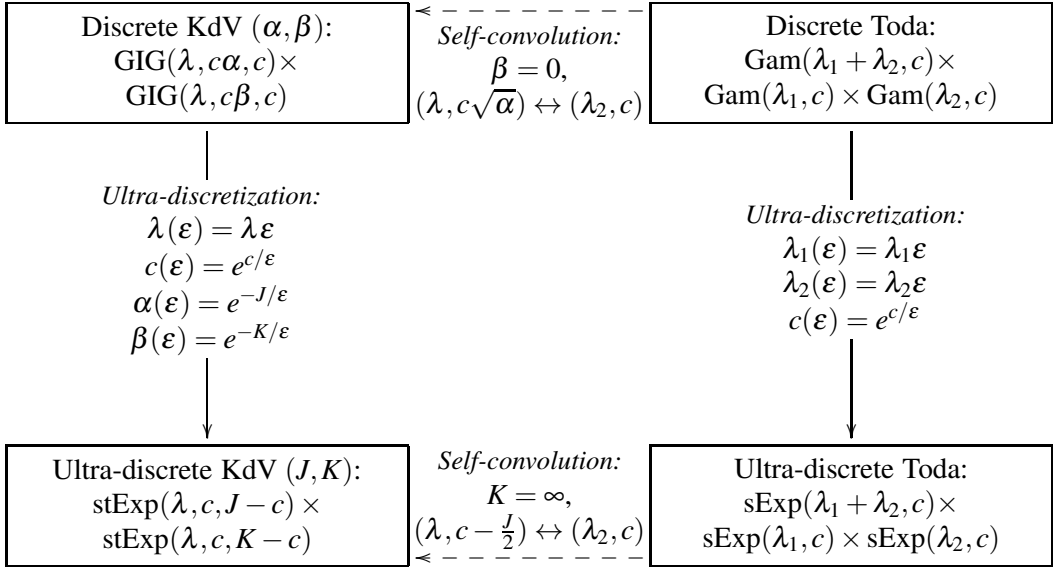

\subsection{Ultra-discretization}

The systems \eqref{UDKDV} and \eqref{UDTODA} arise as ultra-discrete limits of \eqref{DKDV} and \eqref{DTODA}, respectively. In particular, it is straightforward to check that if
\[x=\lim_{\varepsilon \downarrow 0}\varepsilon \log x(\varepsilon),\qquad u=\lim_{\varepsilon \downarrow 0}\varepsilon \log u(\varepsilon),\]
\[J=\lim_{\varepsilon \downarrow 0}-\varepsilon \log \alpha(\varepsilon),\qquad K=\lim_{\varepsilon \downarrow 0}-\varepsilon \log \beta(\varepsilon),\]
then
\begin{equation}\label{udlimkdv}
\lim_{\varepsilon \downarrow 0}\varepsilon \log \left(F_{dK}^{(\alpha(\varepsilon),\beta(\varepsilon))}\right)^{(i)}(x(\varepsilon),u(\varepsilon)) =\left(F_{udK}^{(J,K)}\right)^{(i)}(x,u),\qquad i=1,2.
\end{equation}
Similarly, if
\[a=\lim_{\varepsilon \downarrow 0}-\varepsilon \log a(\varepsilon),\qquad b=\lim_{\varepsilon \downarrow 0}-\varepsilon \log b(\varepsilon),\qquad c=\lim_{\varepsilon \downarrow 0}-\varepsilon \log c(\varepsilon),\]
then
\begin{equation}\label{udlimtoda}
\lim_{\varepsilon \downarrow 0}-\varepsilon \log F_{dT}^{(i)}(a(\varepsilon),b(\varepsilon),c(\varepsilon)) =F_{udT}^{(i)}(a,b,c),\qquad i=1,2,3.
\end{equation}
As a consequence of the following proposition, we have that making corresponding changes of parameters for certain invariant measures for $F_{dK}^{(\alpha,\beta)}$ and $F_{dT}$ yields invariant measures for $F_{udK}^{(J,K)}$ and $F_{udT}$ (see Remark \ref{udrem}).

\begin{prop}\label{udprop}\hspace{10pt}
\begin{enumerate}
\item[(a)] Suppose that $X(\varepsilon)\sim \mathrm{GIG}(\varepsilon\lambda,c(\varepsilon)\alpha(\varepsilon),c(\varepsilon))$, where $c(\varepsilon):=e^{c/\varepsilon}$ and $\alpha(\varepsilon):=e^{-L/\varepsilon}$, for some $L\in\mathbb{R}\cup\{\infty\}$, $\lambda\in\mathbb{R}$ if $L<\infty$,  $\lambda>0$ if $L=\infty$, and $c<L/2$. It then holds that
\[\lim_{\varepsilon \downarrow 0}\varepsilon \log X(\varepsilon)=X\]
in distribution, where $X\sim \mathrm{stExp}(\lambda,c,L-c)$.
\item[(b)] Suppose that $X(\varepsilon)\sim \mathrm{Gam} (\varepsilon\lambda,c(\varepsilon))$, where $c(\varepsilon):=e^{c/\varepsilon}$, for some $\lambda>0$ and $c\in\mathbb{R}$. It then holds that
\[\lim_{\varepsilon \downarrow 0}-\varepsilon \log X(\varepsilon)=X\]
in distribution, where $X\sim \mathrm{sExp}(\lambda,c)$.
\end{enumerate}
\end{prop}
\begin{proof} (a) Write $Y(\varepsilon):=\varepsilon \log X(\varepsilon)$. By making a standard change of variables, we see that this has density proportional to
\[f_\varepsilon(y):=e^{-\lambda y}\exp\left(-c(\varepsilon)\alpha(\varepsilon)e^{y/\varepsilon}-c(\varepsilon)e^{-y/\varepsilon}\right),\qquad y\in\mathbb{R}.\]
Observe that, for $y\in(c,L-c)$, we have that $f_\varepsilon(y)\rightarrow e^{-\lambda y}$. Hence, by the dominated convergence theorem, for any compact interval $I\subseteq [c,L-c]$, we have
\[\int_I f_\varepsilon(y)dy\rightarrow \int_Ie^{-\lambda y}dy.\]
Moreover, if $L<\infty$, the monotone convergence theorem yields that
\[\int_{L-c}^\infty f_\varepsilon(y)dy\leq\int_{L-c}^\infty e^{-\lambda y} \exp\left(-e^{(y-L+c)/\varepsilon}\right)dy\rightarrow0.\]
Similarly,
\[\int_{-\infty}^c f_\varepsilon(y)dy\leq\int_{-\infty}^c e^{-\lambda y} \exp\left(-e^{-(y-c)/\varepsilon}\right)dy\rightarrow0.\]
Combining the previous three limits, the result readily follows.\\
(b) Writing $Z(\varepsilon):=-\varepsilon \log X(\varepsilon)$, we find that $Z(\varepsilon)$ has density proportional to
\[g_\varepsilon(y):=e^{-\lambda y}\exp\left(-c(\varepsilon)e^{-y/\varepsilon}\right),\qquad y\in\mathbb{R}.\]
Given this, arguing similarly to the first part of the proof gives the desired conclusion.
\end{proof}

\begin{rem}\label{udrem} Applying Proposition \ref{dkdvmeas}(a), \eqref{udlimkdv} and Proposition \ref{udprop}(a) gives another proof of Proposition \ref{udkdvmeas}(a)(i). Similarly, applying Proposition \ref{dTodainv}, \eqref{udlimtoda} and Proposition \ref{udprop}(b) gives another proof of Proposition \ref{udTodainv}.
\end{rem}

\subsection{KdV-Toda correspondence}\label{kdvtodalinksubsec}

In \cite{CSTkdv}, a correspondence was established between one time-step solutions of the ultra-discrete Toda equation with a particular symmetry and solutions of the ultra-discrete KdV equation (with $K=\infty$), and similarly for the discrete models. Here, we use these relations to connect invariant measures for the various systems.

\subsubsection{Ultra-discrete case} To describe the story in the ultra-discrete case, first observe that the $F_{udT}$ preserves the space $\{(a,b,c)\in\mathbb{R}^3:\:a+b=0\}$. In particular, we have that
\[F_{udT}(-a, a,b)=\left(\min\{a,b\}, -\min\{a,b\}, b-a-\min\{a,b\}\right).\]
Combining the first two coordinates, we introduce an involution $K_{udT}:\mathbb{R}^2\rightarrow\mathbb{R}^2$ by setting $K_{udT}(a,b):=(F_{udT}^{(2)}(-a,a,b),F_{udT}^{(3)}(-a,a,b))$, or equivalently,
\[K_{udT}(a,b)=\left(-\min\{a,b\},b-{a}-{\min\{a,b\}}\right).\]
Moreover, we note that this is simply a change of coordinates from $F_{udK}^{(J,\infty)}$. Indeed, if $A^{(J)}(x,u):=(\frac{J}{2}-x,u-\frac{J}{2})$, then we have that
\[F_{udK}^{(J,\infty)} = (A^{(J)})^{-1} \circ K_{udT}  \circ A^{(J)}.\]
The above sequence of operations incorporates the `self-convolution' procedure of \cite[Section 6, and Proposition 6.5 in particular]{CSTkdv}, with the reverse procedure from $F_{udK}^{(J,\infty)}$ to $F_{udT}(-a, a,b)$ involving the `splitting' operation of \cite[Section 6]{CSTkdv}. NB.\ The presentation of this article differs by a unimportant factor of 2 from that of \cite{CSTkdv}, where such a factor was needed to define a `Pitman-type transformation map'. Now, it is an elementary exercise to check that the invariant measure $\mathrm{stExp}(\lambda,c,J-c)\times \mathrm{sExp}(\lambda,c)$ (with $\lambda>0$ and $c\leq J/2$) for $F_{udK}^{(J,\infty)}$ of Proposition \ref{udkdvmeas} corresponds to the following invariant measure for $K_{udT}$:
\begin{equation}\label{abdist}
\mathrm{stExp}\left(-{\lambda},c-\frac{J}{2},\frac{J}{2}-c\right)\times \mathrm{sExp}\left(\lambda,c-\frac{J}{2}\right),
\end{equation}
Returning to the coordinates of the \eqref{UDTODA} system, this gives that if $(A,B)$ has the above distribution, then $(-A,A,B)$ is invariant for $F_{udT}$. We note that this solution relates to the product invariant measure of Proposition \ref{udTodainv}. Indeed, it is readily checked that if $(A,B,C)\sim \mathrm{sExp}(\la_1+\la_2,c) \times \mathrm{sExp}(\la_1,c) \times \mathrm{sExp}(\la_2,c)$ with $\lambda_1,\lambda_2>0$ and $c<0$, then
\begin{equation}\label{eps}
\left(A,B,C\right)\:\vline\:\left\{|A+B|\leq \varepsilon\right\}\:\rightarrow \left(-\tilde{A},\tilde{A},\tilde{B}\right),
\end{equation}
in distribution as $\varepsilon\rightarrow 0$, where $(\tilde{A},\tilde{B})\sim \mathrm{stExp}(-\lambda_2,c,-c)\times \mathrm{sExp}(\la_2,c)$. Since it holds that $a+b=F^{(1)}_{udT}(a,b,c)+F^{(2)}_{udT}(a,b,c)$, the left-hand side of \eqref{eps}  has a distribution that is invariant under $F_{udT}$, and thus the continuous mapping theorem implies that so does the right-hand side. Reparameterising gives that $(\tilde{A},\tilde{B})$ has distribution as at \eqref{abdist}, which establishes that, in the case $K=\infty$, Proposition \ref{udkdvmeas}(a)(i) can alternatively be obtained as a consequence of Proposition \ref{udTodainv}.

\subsubsection{Discrete case}
The discrete case is similar to the ultra-discrete one. Indeed, $F_{dT}$ preserves the space $\{(a,b,c)\in(0,\infty)^3:\:ab=1\}$, with
\[F_{dT}\left(a^{-1}, a,b\right)=\left(a+b, (a+b)^{-1},\frac{b}{a(a+b)}\right).\]
In this case, we introduce an involution $K_{dT}:(0,\infty)^2\rightarrow(0,\infty)^2$ by setting
\[K_{dT}(a,b):=\left(F_{dT}^{(2)}(-a,a,b),F_{dT}^{(3)}(-a,a,b)\right)=\left(\frac{1}{a+b},\frac{b}{a(a+b)}\right),\]
and note that if $A^{(\alpha)}(x,u):=(x\sqrt{\alpha},\frac{1}{u\sqrt{\alpha}})$, then
\[F^{(\alpha,0)}_{dK} = (A^{(\alpha)})^{-1} \circ K_{dT}  \circ A^{(\alpha)}.\]
Again, these operations essentially describe the self-convolution procedure of \cite[Section 6]{CSTkdv}, with the reverse procedure from $F^{(\alpha,0)}_{dK}$ to $F_{dT}(a^{-1}, a,b)$ involving the splitting procedure of \cite[Section 6]{CSTkdv}. The invariant measure $\mathrm{GIG}(\lambda,c\alpha,c)\times \mathrm{IG}(\lambda,c)$ (with $\lambda,c>0$) for $F^{(\alpha,0)}$ of Proposition \ref{dkdvmeas} corresponds to the following invariant measure for $K_{dT}$:
\begin{equation}\label{dist0}
\mathrm{GIG}\left(\lambda,c\sqrt{\alpha},c\sqrt{\alpha}\right)\times \mathrm{Gam}\left(\lambda,{c}{\sqrt{\alpha}}\right).
\end{equation}
Hence, if $(A,B)$ has the above distribution, then the law of $(A^{-1},A,B)$ is invariant under $F_{dT}$. Moreover, it is possible to check that the solution relates to the product invariant measure of Proposition \ref{dTodainv}. For, if $(A,B,C)\sim
\mathrm{Gam}(\la_1+\la_2,c) \times \mathrm{Gam}(\la_1,c) \times \mathrm{Gam}(\la_2,c)$, then one may verify that
\begin{equation}\label{eps2}
\left(A,B,C\right)\:\vline\:\left\{|AB-1|\leq \varepsilon\right\}\:\rightarrow \left(\tilde{A}^{-1},\tilde{A},\tilde{B}\right),
\end{equation}
in distribution as $\varepsilon\rightarrow 0$, where $(\tilde{A},\tilde{B})\sim  \mathrm{GIG}(\la_2,c,c)\times  \mathrm{Gam}(\la_2,c)$. Since it holds that $ab=F^{(1)}_{dT}(a,b,c)F^{(2)}_{dT}(a,b,c)$, the left-hand side of \eqref{eps2}  has a distribution that is invariant under $F_{dT}$, and thus the continuous mapping theorem implies that so does the right-hand side. Reparameterising gives that $(\tilde{A},\tilde{B})$ has distribution as at \eqref{dist0}, which establishes that, in the case $\beta=0$, Proposition \ref{dkdvmeas}(a) can alternatively be obtained as a consequence of Proposition \ref{dTodainv}.

\section{Connection to stochastic integrable models}\label{stochsec}

In this section, we discuss links between our framework and results, and studies on stochastic integrable models. To expand slightly, stochastic two-dimensional lattice integrable (explicitly solvable) models have been intensively studied in recent years in the context of KPZ universality. These include last passage percolation with exponential/geometric weights, the log-gamma, strict-weak, beta, and inverse-beta directed random polymer models, and also higher spin vertex models. An important common property of these systems is that they admit stationary measures that satisfy an appropriate version of Burke's property. We will describe how the arguments of Subsection \ref{burkesec} can be extended to cover the stochastic setting, and explain how this applies in a number of examples. We highlight that we are able to make explicit connections between the Toda-type systems of Section \ref{todasec} and certain polymer models. This part of the study is continued in \cite{CSrp}, wherein the techniques of this article are used to explore the stationary solutions of random polymer models and their zero-temperature limits.

A typical setting for the stochastic models of interest here is the following: for a given boundary condition $(X^0_n,U^t_0)_{n\geq 1,t\geq 0}$, where the $X^0_n$ are random variables taking values in a space $\mathcal{X}_0$ and the $U^t_0$ are random variables taking values in a space $\mathcal{U}_0$, the random variables $(X^t_n,U^t_n)_{n\geq 1,t\geq 0}$ are defined recursively via the equations:
\begin{equation}\label{stochasticmodel}
\left(X_n^{t+1},U_{n}^t\right)=R\left(\tilde{X}_n^t, X^{t}_n,U^t_{n-1}\right)
\end{equation}
where $R: \tilde{\mathcal{X}}_0 \times \mathcal{X}_0 \times  \mathcal{U}_0 \to  \mathcal{X}_0 \times  \mathcal{U}_0$ is a deterministic function, and $(\tilde{X}_n^t)_{n\geq 1,t\geq 0}$ are i.i.d.\ random variables, independent of $(X^0_n,U^t_0)_{n\geq 1,t\geq 0}$. In particular, for a given realization of the variables $(\tilde{X}_n^t)_{n\geq 1,t\geq 0}$, we have a two-dimensional system of equations of the form described in the first sentence of the article with $F_n^t=R(\tilde{X}_n^t,\cdot,\cdot)$. For these models, we define the following notion of Burke's property.

\begin{description}
  \item[Burke's property for a stochastic model] We say that the distribution of the random variables $(X^t_n,U^t_n)_{n\geq 1,t\geq 0}$ satisfies \emph{Burke's property} if:
      \begin{itemize}
      \item the sequences $(X_n^0)_{n \geq1}$ and $(U_0^t)_{t \geq 0}$ are each i.i.d., and independent of each other;
      \item the distribution of $(X^t_{n},U^t_{n})_{n\geq 1,t\geq 0}$ is translation invariant, that is, for any $m,s \in \Z_+$,
    \[\left(X^{s+t}_{m+n},U^{s+t}_{m+n}\right)_{n\geq 1,t\geq 0}\buildrel{d}\over{=}\left(X^t_{n},U^t_{n}\right)_{n\geq 1,t\geq 0}.\]
\end{itemize}
\end{description}
\bigskip

By applying the same argument as that used to prove Proposition \ref{Burke}, we can obtain the following.

\begin{prop}[Burke's property for a stochastic model]
Suppose $\tilde{\mu}, \mu, \nu$ are probability measures on $\tilde{\mathcal{X}}_0, \mathcal{X}_0, \mathcal{U}_0$ respectively satisfying \begin{equation}\label{rdb}
R(\tilde{\mu} \times \mu \times \nu) = \mu \times \nu.
\end{equation}
If $(X^0_n)_{n \geq 1}, (U^t_0)_{t \geq0}, \{\tilde{X}^t_n\}_{n\geq 1,t\geq 0}$ are independent random variables whose marginals are $\mu, \nu$ and $\tilde{\mu}$ respectively, then the random variables $(X^t_n,U^t_n)_{n\geq 1,t\geq0}$ defined by the relation \eqref{stochasticmodel} satisfy Burke's property for a stochastic model.
\end{prop}

Just as in the deterministic case, it is also possible to consider inhomogeneous stochastic models. For the purposes of the subsequent discussion, in this direction we suppose that $X^t_n$ are $\mathcal{X}_n$-valued random variables, $U^t_n$ are $\mathcal{U}_t$-valued random variables, $\tilde{X}^t_n$ are $\tilde{\mathcal{X}}^t_n$ random variables, and there exists a sequence of deterministic functions
\[ R_{n,t} : \tilde{\mathcal{X}}^t_n \times \mathcal{X}_n \times \mathcal{U}_t \to \mathcal{X}_n \times \mathcal{U}_t.\]
Given a (random) boundary condition $(X^0_n,U^t_0)_{n\geq1,t\geq0}$, we then define $(X^t_n,U^t_n)_{n\geq1,t\geq0}$ by
\begin{equation}\label{stochasticmodelin}
(X_n^{t+1},U_n^t)=R_{n,t}(\tilde{X}_n^t, X^{t}_n,U^t_{n-1}).
\end{equation}
For such dynamics, we define Burke's property as follows.

\begin{description}
  \item[Burke's property for an inhomogeneous stochastic model] We say that the distribution of the random variables $(X^t_n,U^t_n)_{n\geq 1,t\geq 0}$ satisfies \emph{Burke's property} if there exist a sequence of probability measures $(\mu_n)_{n\geq 1}$, with $\mu_n$ supported on $\mathcal{X}_n$, and $(\nu_t)_{t\geq0}$, with $\nu_t$ supported on $\mathcal{U}_t$ such that:
      \begin{itemize}
      \item $X^t_n \sim \mu_n$ for all $n\geq 1,t\geq 0$;
      \item $U^t_n \sim \nu_t$ for all $n\geq 1,t\geq 0$;
      \item for any $m,s \in \Z_+$, $(X^{s}_{m+n})_{n \ge1}, (U^{s+t}_{m})_{ t \ge 0}$ are independent random variables.
      \end{itemize}
\end{description}
\bigskip

The above notion of Burke's property was discussed in \cite{IMS} in the study of the stochastic higher spin six vertex model introduced in \cite{CP} (see Subsection \ref{hvsec} below). We can prove the following in the same way as the homogeneous case.

\begin{prop}[Burke's property for an inhomogeneous stochastic model]\label{Burkein}
Suppose $\tilde{\mu}^t_n$, $\mu_n$, $\nu_t$ are probability measures on $\tilde{\mathcal{X}}_n^t$, $\mathcal{X}_n$, $\mathcal{U}_t$, respectively, satisfying
\[R_{n,t}(\tilde{\mu}^t_n \times \mu_n \times \nu_t) = \mu_n \times \nu_t.\]
If $(X^0_n)_{n \geq1}$, $(U^t_0)_{t \geq0}$, $(\tilde{X}^t_n)_{n\geq1,t\geq0}$ are independent random variables whose marginals are $\mu_n$, $\nu_t$ and $\tilde{\mu}^t_n$, respectively, then the random variables $(X^t_n,U^t_n)_{n\geq 1,t\geq0}$ defined by
the relation \eqref{stochasticmodelin} satisfy Burke's property for an inhomogeneous stochastic model.
\end{prop}

The type I and type II models considered in the earlier deterministic part of the article can be understood as special cases of the stochastic models in the following ways.
\begin{itemize}
  \item Firstly, the local dynamics of a type I model clearly match those of a homogeneous stochastic model for which the map $R$ at \eqref{stochasticmodel} does not depend on $\tilde{X}_{n}^t$. More generally, one could apply Proposition \ref{Burkein} to study an inhomogeneous deterministic model. For example, if we set $F_{n,t}=F^{J_n,K_t}_{udK}$, then we have that $\mu_n=\mathrm{stExp}(\lambda, c, J_n-c)$ and $\nu_t=\mathrm{stExp}(\lambda, c, K_t-c)$ satisfy
\[F^{J_n,K_t}_{udK}\left(\mu_n\times\nu_t\right)=\mu_n\times\nu_t,\]
and so there is a distribution on $(x^t_n,u^t_n)_{n\geq 1,t \geq 0}$ that satisfies the inhomogeneous version of Burke's property.
  \item Secondly, to connect to type II models, we observe that the condition $R(\tilde{\mu} \times \mu \times \nu) = \mu \times \nu$ at \eqref{rdb} matches the condition on $F$ of Proposition \ref{twothree}(a). Hence, if $R$ is given by a map
      \begin{equation}\label{rmap}
      R(a,b,c)=R_*^{-1}\left(a, R_*^{(2)}(b,c)\right),
      \end{equation}
      where $R_* : \mathcal{X}_0 \times \mathcal{U}_0 \to \tilde{\mathcal{X}}_0 \times \tilde{\mathcal{U}}_0$ is a bijection (i.e.\ similarly to \eqref{threeinvolution} with $R=F^{(2,3)}$ and $R_*=F_*$), then the detailed balance condition for the type II model given by $R_*$ is equivalent to \eqref{rdb}. Consequently, for any type II model, we can construct stochastic counterpart by \eqref{rmap}, and the detailed balance condition for $F_*=R_*$ implies the existence of distributions satisfying Burke's property both for the deterministic and stochastic models. Note that the configuration for the deterministic model is $(x^t_n, u^t_n)_{n,t \in \Z}$, where $x^t_n \in \mathcal{X}_0$, $u^t_n  \in\tilde{ \mathcal{U}}_0$ for $n+t =0$ (mod 2), and $x^t_n \in \tilde{\mathcal{X}}_0$, $u^t_n  \in {\mathcal{U}}_0$ for $n+t=1$ (mod 2), whereas, for the stochastic model, $(X^t_n, U^t_n)_{n\geq 1,t\geq 0}$ satisfies $X^t_n  \in \mathcal{X}_0$ and $U^t_n  \in \mathcal{U}_0$  for all $n,t$.
\end{itemize}
We next proceed to discuss a number of examples of stochastic integrable systems. In particular, we will observe that \[R_{DLPP}=F_{udT}^{(2,3)},\qquad R_{RPs}=F_{dT}^{(2,3)},\]
where $R_{DLPP}$ is the function $R$ for directed last passage percolation, and $R_{RPs}$ is the function $R$ for the directed polymer with site weights (see Subsections \ref{dlppsec} and \ref{drpsec}, respectively). We will further see that $R_{RPe}$, the function $R$ for the directed polymer with edge weights, can also written in terms of a bijection $R_*=R_{RPe^*}$. For the latter model, the solutions of the detailed balance equation were, up to a regularity condition, characterized in \cite{CN}.

\subsection{Directed last passage percolation in two dimensions}\label{dlppsec} In the study of directed last passage percolation on $\N^2$, a key quantity of interest is the partition function
\[Z_{n,m}=\max_{\pi : (1,1) \to (n,m)}\left\{\sum_{(k,\ell) \in \pi}X_{k,\ell}\right\},\qquad m,n\in\mathbb{N},\]
where the maximum is taken over `up-right paths' $\pi : (1,1) \to (n,m)$ on $\N^2$, and $(X_{n,m})_{n,m \in \N}$ are i.i.d.\ random variables. One readily sees that this partition function satisfies the following recursion:
\begin{equation}\label{LLP}
Z_{n,m}=X_{n,m}+\max\left\{Z_{n-1,m}, Z_{n,m-1}\right\}.
\end{equation}
% The lattice needs not be $\N^2$. We only need the above relation.
By setting $U_{n,m}:=Z_{n,m}-Z_{n-1,m}$ and $V_{n,m}:=Z_{n,m}-Z_{n,m-1}$, the recursive equation at \eqref{LLP} can be rewritten as
\[R_{DLPP}(X_{n,m}, U_{n,m-1},V_{n-1,m})=(U_{n,m}, V_{n,m}),\]
where
\[R_{DLPP}\left(a,b,c\right)=\left(a+b-\min\{b,c\}, a+c-\min\{b,c\}\right).\]
In particular, $R_{DLPP}=F_{udT}^{(2,3)}$, and we obtain from Proposition \ref{twothree} that
\[R_{DLPP}(\tilde{\mu} \times \mu \times \nu)= \mu \times \nu\qquad\Leftrightarrow\qquad F_{udT^*}(\mu \times \nu)= \tilde{\mu} \times \tilde{\nu}\mbox{ for some }\nu.\]
Apart from trivial solutions, we have from Proposition \ref{UDTODAmeas} that the above identities imply that $\tilde{\mu}$ is a (possibly scaled and shifted) exponential/geometric distribution; note that when $(X_{n,m})_{n,m\in\mathbb{N}}$ has an i.i.d.\ exponential/geometric distribution, the directed last passage percolation model is known to be exactly solvable. Moreover, the solution of the detailed balance equation for $F_{udT^*}$ of Proposition \ref{UDTODAmeas} further yields the existence of the stationary distribution $U_{n,m} \sim \mathrm{Exp}(\la_1)$, $V_{n,m}  \sim \mathrm{Exp}(\la_2)$, $X_{n,m} \sim \mathrm{Exp}(\la_1+\la_2)$ and its geometric distribution version, cf. \cite{BCS}.

\subsection{Directed random polymer with site weights}\label{drpsec} For this model, which is a positive temperature version of directed last passage percolation, the partition function is given by
\[Z_{n,m}=\sum_{\pi : (1,1) \to (n,m)}\left\{\prod_{(k,\ell) \in \pi}X_{k,\ell}\right\},\qquad m,n\in\mathbb{N},\]
where the sum is taken over `up-right paths' $\pi : (1,1) \to (n,m)$ on $\N^2$, and $(X_{n,m})_{n,m \in \N}$ are i.i.d.\ random variables. In this case, we have a recursive equation for the partition function of the form
\begin{equation}\label{RP}
Z_{n,m}=X_{n,m}\left(Z_{n-1,m}+Z_{n,m-1}\right).
\end{equation}
Letting $U_{n,m}=Z_{n,m}/Z_{n-1,m}$, $V_{n,m}=Z_{n,m}/Z_{n,m-1}$, the recursive equation \eqref{RP} can be rewritten as
\[R_{RPs}\left(X_{n,m}^{-1}, U_{n,m-1}^{-1},V_{n-1,m}^{-1}\right)=\left(U_{n,m}^{-1}, V_{n,m}^{-1}\right),\]
where
\[R_{RPs}\left(a,b,c\right)=\left(\frac{ab}{b+c}, \frac{ac}{b+c}\right).\]
We thus see that $R_{RPs}=F_{dT}^{(2,3)}$, and we obtain from Proposition \ref{twothree} that
\[R_{RPs}(\tilde{\mu} \times \mu \times \nu)= \mu \times \nu\qquad\Leftrightarrow\qquad F_{dT^*}(\mu \times \nu)= \tilde{\mu} \times \tilde{\nu}\mbox{ for some }\nu.\]
From Proposition \ref{DTODAmeas}, we have that the only non-trivial solution to these equations has $\tilde{\mu}$ being a gamma distribution, and, similarly to the comment in the previous example, it is of note to observe that when $(X_{n,m})_{n,m\in\mathbb{N}}$ has an i.i.d.\ inverse gamma distribution, the model is exactly solvable. Furthermore, it also follows from Proposition \ref{DTODAmeas} that we have the existence of a stationary distribution with $U_{n,m}^{-1} \sim \mathrm{Gam}(\la_1,c)$, $V_{n,m}^{-1}  \sim \mathrm{Gam}(\la_2,c)$, $X_{n,m}^{-1} \sim \mathrm{Gam}(\la_1+\la_2,c)$, cf. \cite{S,Serr}.

\subsection{Directed random polymer with edge weights} Similarly to the previous subsection, the model we next consider has partition function
\[  Z_{n,m}=\sum_{\pi : (0,0) \to (n,m)}\left\{\prod_{e_i \in \pi}Y_{e_i}\right\},\qquad m,n\in\mathbb{N},\]
where again the sum is taken over `up-right paths' $\pi : (1,1) \to (n,m)$ on $\N^2$, and
\[Y_{e_i}:=
\left\{
  \begin{array}{ll}
    X_{k,l},&\hbox{if } e_i=((k-1,\ell), (k,\ell)),\\
    h(X_{k,\ell}), & \hbox{if } e_i=((k,\ell-1), (k,\ell)),
  \end{array}
\right.\]
where $h$ is a positive function on $\R_+$, and $(X_{n,m})_{n,m \in \N}$ are i.i.d.\ random variables. This partition function satisfies
\begin{equation}\label{RPe}
Z_{n,m}=X_{n,m}Z_{n-1,m}+h(X_{n,m})Z_{n,m-1},
\end{equation}
and by letting $U_{n,m}:=Z_{n,m}/Z_{n-1,m}$, $V_{n,m}:= Z_{n,m}/Z_{n,m-1}$, the recursive equation \eqref{RPe} can be rewritten as
\[R_{RPe}\left(X_{n,m}, U_{n,m-1},V_{n-1,m}\right)=\left(U_{n,m}, V_{n,m}\right),\]
where
\[R_{RPe}\left(a,b,c\right)=\left(a+\frac{h(a)b}{c},h(a)+\frac{a c}{b}\right).\]
Note that, whilst in the previous example we wrote $R_{RPs}$ in terms of $(X_{n,m}^{-1}, U_{n,m-1}^{-1},V_{n-1,m}^{-1})$ in order to fit closely with the map $F_{dT^*}$, here we write $R_{RPe}$ in terms of $(X_{n,m}, U_{n,m-1},V_{n-1,m})$ to better fit the discussion in \cite{CN}. In particular, in \cite{CN}, up to technical conditions, the authors characterize distributions $\tilde{\mu}$, $\mu$ and $\nu$ such that $R_{RPe}(\tilde{\mu} \times \mu \times \nu)=\mu \times \nu$. To expand on this, under the assumptions of \cite{CN}, whenever $b$ and $c$ are in the support of $\mu \times \nu$, the function $H_s(a):=as+h(a)$, where $s=\frac{c}{b}$, has an inverse function $H_s^{-1}$ on the support of $\tilde{\mu}$. It follows that the function
\[R_{RPe^*}(x,u)=\left(H^{-1}_{\frac{u}{x}}(u),\frac{u}{x}\right),\]
is a bijection (on the support of $\mu\times \nu$), with inverse function given by
\[R_{RPe^*}^{-1}(x,u)=\left(\frac{1}{u}H_u(x),H_u(x)\right)=\left(x+\frac{h(x)}{u}, h(x)+xu\right),\]
and putting these together yields
\[R_{RPe}\left(a,b,c\right)= R_{RPe^*}^{-1}\left(a, R_{RPe^*}^{(2)}(b,c)\right).\]
Hence the condition $R_{RPe}(\tilde{\mu} \times \mu \times \nu)=\mu \times \nu$ is equivalent to $R_{*, RPe}(\mu \times \nu)=\tilde{\mu} \times \tilde{\nu}$ for some $\tilde{\nu}$, and also to
\[F_{RPe}(\tilde{\mu} \times \mu \times \nu)=\tilde{\mu} \times \mu \times \nu,\]
where
\[F_{RPe}(a,b,c)= \left(R_{RPe^*}^{(1)}(b,c),  R_{RPe^*}^{-1}\left(a, R_{*,RPe}^{(2)}(b,c)\right)\right).\]
In \cite{CN}, the authors show that Burke's property holds for the directed random polymer with edge weights only if $h(x)=Ax+B$ for some $A,B \in \R$ such that $\max\{A,B\} >0$. NB.\ In this case, the above map is of the form
\[\xymatrix@C-15pt@R-15pt{ &\frac{(a-B)b}{Ab+c}&\frac{a(Ab+c)}{c}+\frac{Bb}{c}&\\
           c\ar[rrr]& & & \frac{a(Ab+c)}{b}+B.\\
     & b\ar[uu]&a\ar[uu]&}\]
Moreover, they characterize all distributions satisfying Burke's property. Up to linear transformations, these fall into one of the following four classes:
\begin{description}
  \item[Inverse gamma] For $A=1$, $B=0$, i.e.\ $h(x)=x$,
\[U_{n,m} \sim \mathrm{IG}(\la_1,c),\qquad V_{n,m}  \sim \mathrm{IG}(\la_2,c),\qquad X_{n,m} \sim \mathrm{IG}(\la_1+\la_2,c);\]
  \item[Gamma] For $A=0$, $B=1$, i.e.\ $h(x)=1$,
\[U_{n,m} \sim \mathrm{Gam}(\la_1+\la_2,c),\qquad V_{n,m}  \sim \mathrm{Be}^{-1}(\la_1,\la_2),\qquad X_{n,m} \sim \mathrm{Gam}(\la_2,c);\]
  \item[Beta] For $A=-1$, $B=1$, i.e.\ $h(x)=1-x$,
\[U_{n,m} \sim \mathrm{Be}(\la_1+\la_2,\la_3),\qquad V_{n,m}^{-1}  \sim \mathrm{Be}(\la_1,\la_2),\qquad X_{n,m} \sim \mathrm{Be}(\la_2,\la_3);\]
  \item[Inverse beta] For $A=1$, $B=-1$, i.e.\ $h(x)=x-1$,
\[U_{n,m}^{-1} \sim \mathrm{Be}(\la_1,\la_3),\qquad (V_{n,m}+1) ^{-1} \sim \mathrm{Be}(\la_2,\la_1+\la_3) ,\qquad X_{n,m}^{-1} \sim \mathrm{Be}(\la_1+\la_2,\la_3).\]
\end{description}
To obtain the results in the cases $h(x)=x$ and $h(x)=1$, the well-known characterization of gamma distributions from \cite{L} was applied, cf.\ our argument characterising the invariant measures for the discrete Toda lattice. (Note that if $h(x)=x$, then, up to inversion of the variables, the dynamics of $R_{RPe}$ matches that of $R_{RPs}$.) In the cases $h(x)=1-x$ and $h(x)=x-1$, a similar result for the beta distribution is used, see \cite{SW}.

\begin{rem}
The equation \eqref{RPe} with $h(x)=1-x$ corresponds to a recursion equation for the distribution function of the random walk in a Beta-distributed random environment, as studied in \cite{BI}. Specifically, the environment of the latter model is given by an i.i.d.\ collection of $\mathrm{Be}(\alpha,\beta)$ random variables $(B_{n,t})_{n\in\mathbb{Z},t\geq 0}$, and conditional on this, the process $(Y_t)_{t\geq 0}$ is the (discrete-time) Markov chain with transition probabilities given by
\[\mathbf{P}^{B}\left(Y_{t+1}=n+1\:\vline\:Y_t=n\right)=B_{n,t}=1-\mathbf{P}^{B}\left(Y_{t+1}=n-1\:\vline\:Y_t=n\right).\]
It is readily checked that $\tilde{Z}(t,n):=P^B(Y_t\geq t-2n+2)$ satisfies
\[\tilde{Z}(t,n)=B_{t,n}\tilde{Z}(t-1,n)+(1-B_{t,n})\tilde{Z}(t-1,n-1).\]
Reparameterising by setting $Z_{n,m}:=\tilde{Z}(n+m,m)$, $X_{n,m}:=B_{n+m,m}$, we obtain \eqref{RPe}.
\end{rem}

\subsection{Higher spin vertex models}\label{hvsec} In this subsection, we explain how Proposition \ref{Burkein} applies to higher spin vertex models. The state spaces for such models are given by $\mathcal{X}_0:=\{0,1,2,\dots,\}$, and $\mathcal{U}_0:=\{0,1,\dots, J\}$ for some $J \in \N$. In the case $J=1$, the dynamics of the model are given by the probabilities
\begin{align*}
\mathbf{P}\left((X^{t+1}_n, U^t_n)=(i,0)\:\vline\: (X^{t}_n, U^t_{n-1}=(i,0)\right)&=\frac{1+\alpha q^i}{1+\alpha}=:c_{i,0}, \\
\mathbf{P}\left((X^{t+1}_n, U^t_n)=(i-1,1)\:\vline\: (X^{t}_n, U^t_{n-1}=(i,0)\right)&=\frac{\alpha(1- q^i)}{1+\alpha},\\
\mathbf{P}\left((X^{t+1}_n, U^t_n)=(i+1,0)\:\vline\: (X^{t}_n, U^t_{n-1}=(i,1)\right)&=\frac{1-\nu q^i}{1+\alpha}=:c_{i,1}, \\
\mathbf{P}\left((X^{t+1}_n, U^t_n)=(i,1)\:\vline\: (X^{t}_n, U^t_{n-1}=(i,1)\right)&=\frac{\alpha+ \nu q^i}{1+\alpha},
\end{align*}
for an appropriate choice of $\alpha, \nu, q$, see \cite{CP} for details. For simplicity, we consider the case $\alpha \ge 0$ and $\nu,q\in [0,1)$. If
\[R_{HSV}^{\alpha,\nu,q}\left(u,i,j\right):=\left(i+j -\mathbf{1}_{\{u \ge c_{i,j}\}}, \mathbf{1}_{\{u \ge c_{i,j}\}}\right),\]
and $(\tilde{X}_n^t)_{n\geq 1,t\geq 0}$ is an i.i.d.\ collection of uniform random variables on $(0,1)$, we then have that
\[\left(X^{t+1}_n, U^t_n\right)=R_{HSV}^{\alpha,\nu,q}  \left(\tilde{X}_n^t, X^{t}_n, U^t_{n-1}\right).\]
By direct computation, one can check that
\[R_{HSV}^{\alpha,\nu,q}\left( \mathrm{Uni} (0,1) \times\mathrm{qNB}\left(\nu,\frac{p}{\alpha}\right) \times  \mathrm{qNB}\left(q^{-1},-qp\right)\right)=\mathrm{qNB}\left(\nu,\frac{p}{\alpha}\right) \times  \mathrm{qNB}\left(q^{-1},-qp\right)
\]
for any $0 \le p \le \alpha$, where $\mathrm{Uni} (0,1)$ is the uniform distribution on $(0,1)$,  and qNB is a $q$-negative binomial distribution (see the appendix for details). Note in particular that $\mathrm{qNB}(q^{-1},-qp)$ is a Bernoulli distribution with parameter $\frac{p}{1+p}$. In \cite{IMS}, the authors introduce a change of parameters from $(\nu, \alpha, p)$ to $(s, \xi, u, v)$ with $0 \le s <1$, $\xi >0$, $u <0$, $0 \le v < s\xi$, so that $\alpha= -s\xi u$, $\nu=s^2$, $p=-uv$. With this, we have
\[R_{HSV}^{-s\xi u,s^2,q} \left( \mathrm{Uni} (0,1) \times \mathrm{qNB}\left(s^2,\frac{v}{s\xi}\right) \times  \mathrm{qNB}\left(q^{-1},quv\right)\right)=\mathrm{qNB}\left(s^2,\frac{v}{s\xi}\right) \times  q\mathrm{NB}\left(q^{-1},quv\right).\]
Moreover, in \cite{IMS}, the parameters $(s, \xi, u)$ are allowed to be inhomogeneous, so that $s=s_n$, $\xi=\xi_n$ and $u=u_t$. To align with this framework, we set $R_{n,t}=R_{HSV}^{-s_n\xi_n u_t,s_n^2,q}$. It then follows that, for any fixed $0 \le v < \inf_n{s_n\xi_n}$,
\[R_{n,t}  \left( \mathrm{Uni} (0,1) \times \mu_n \times  \nu_t\right)=  \mu_n \times  \nu_t,\]
where $\mu_n=\mathrm{qNB}(s_n^2,\frac{v}{s_n\xi_n})$, $\nu_t=\mathrm{qNB}(q^{-1},qu_tv)$.

For more general $J \in \N$, the model is defined by a fusion operation, see \cite{CP}. This gives the stochastic matrix
\[\mathbf{P}\left((X^{t+1}_n, U^t_n)=(i',j')\:\vline\: (X^{t}_n, U^t_{n-1}=(i,j)\right)=\mathbf{1}_{\{i'+j'=i+j\}}p_{i',j'}\]
for $i,i' \in \{0,1,2,\dots,\}$, $j,j' \in\{0,1,\dots,J\}$, and so there exists $R_{HSV}^{J,\alpha,\nu,q}$ such that
\[(X^{t+1}_n, U^t_n)=R_{HSV}^{J, \alpha,\nu,q}  \left(\tilde{X}_n^t, X^{t}_n, U^t_{n-1}\right)\]
with $(\tilde{X}_n^t)_{n\geq 1,t\geq0}$ i.i.d.\ uniform random variables on $(0,1)$. Noting that a random variable $X \sim q\mathrm{NB}(q^{-J},-q^Jp)$ can be written as $X=Y_1+Y_2+ \dots Y_J$, where $Y_i \sim q\mathrm{NB}(q^{-1},-q^ip)=\mathrm{Ber}(\frac{q^{i-1}p}{1+p})$ (see Proposition 2.3 of \cite{IMS}), the fusion procedure gives that
\[R_{HSV}^{J, \alpha,\nu,q} \left( \mathrm{Uni} (0,1) \times \mathrm{qNB}\left(\nu,\frac{p}{\alpha}\right) \times  \mathrm{qNB}\left(q^{-J},-q^Jp\right)\right)=\mathrm{qNB}\left(\nu,\frac{p}{\alpha}\right) \times  \mathrm{qNB}\left(q^{-J},-q^{J}p\right).\]
In \cite{IMS}, the inhomogeneous version was also studied in the same way as above. Namely, for $R_{n,t}=R_{HSV}^{J, -s_n\xi_n u_t,s_n^2,q}$ for $n,t \in \N^2$, for any fixed $0 \le v < \inf_n{s_n\xi_n}$,
\[R_{n,t}  \left( \mathrm{Uni} (0,1) \times \mu_n \times  \nu_t\right)=  \mu_n \times  \nu_t,\]
where $\mu_n=\mathrm{qNB}(s_n^2,\frac{v}{s_n\xi_n})$, $\nu_t=\mathrm{qNB}(q^{-J},q^Ju_tv)$. Hence Proposition \ref{stochasticmodelin} applies.

As a final remark, we note that the role of the distribution of $\tilde{X}$ and the function $R$ are different in the higher spin vertex model and the other models discussed here. Indeed, for models other than the higher spin vertex model, the function $R$ reflects the structure of the model, or more precisely the recursion equation of the partition function, independent of the distribution of $\tilde{X}$. On the other hand, for the higher spin vertex model, the function $R$ and the distribution of $\tilde{X}$ do not have any meaning in themselves, but rather the pair together determines the stochastic matrix from the input $(X_n^{t},U^t_{n-1})$ to $(X_n^{t+1},U^t_{n})$, which determines the model.

\section{Iterated random functions}\label{irfsec}

As noted in the introduction, our models can be understood as a special class of iterated random functions. In this section, we discuss how our contributions relate to some known results in the literature regarding such systems. To introduce iterated random functions, we will follow the notation of Diaconis and Freedman's article \cite{DF}, which is a comprehensive survey on this subject (up to its year of writing). Let $S$ be a topological space equipped with its Borel $\sigma$-algebra, $(\Theta, \mathcal{F})$ be a measurable space, $\{f_{\theta}:\:\theta \in \Theta\}$ be a collection of continuous maps $f_\theta: S \to S$, and $\mu$ be a probability measure on $\Theta$. Let $(\theta_n)_{n \in \Z}$ be an  i.i.d.\ sequence with marginal $\mu$. The object of interest is the Markov chain $X_n$ constructed by iterating random functions on the state space $S$, that is
\[X_n:=f_{\theta_n}\left(X_{n-1}\right)=\left(f_{\theta_n} \circ f_{\theta_{n-1}} \circ f_{\theta_2} \dots f_{\theta_1}\right)(X_0),\]
where $X_0=s$ for some $s \in S$. Diaconis and Freedman showed that when `$(f_{\theta})_{\theta \in \Theta}$ is contracting on average' (see \cite{DF} for a precise definition), $X_n$ has a unique stationary distribution, which is independent of $s$. We highlight that a key ingredient in the proof of this theorem is the proposition that the backward iteration defined by
\[Y_n:=\left(f_{\theta_1} \circ f_{\theta_2} \circ f_{\theta_3} \dots f_{\theta_n}\right)(s)\]
converges almost surely, at an exponential rate, to a random variable $Y_{\infty}$ that does not depend on $s$ (see \cite[Proposition 5.1]{DF}).

We now explain how our setting is embedded into the iterated random function framework, starting with type I models. Recall in this case, we have an involution $F: \mathcal{X}_0 \times \mathcal{U}_0 \to  \mathcal{X}_0 \times \mathcal{U}_0$, and that, for a given $(x_n)_{n\in\mathbb{Z}} \in \mathcal{X}_0^\mathbb{Z}$, we are interested in the existence and uniqueness of $(u_n)_{n\in\Z} \in \mathcal{U}_0^\Z$ such that
\[u_{n}=F^{(2)}(x_n, u_{n-1}).\]
(Cf.\ \eqref{uniqueu}.) Letting $S:=\mathcal{U}_0$, $\Theta:=\mathcal{X}_0$ and $f_{\theta}:=F^{(2)}(\theta,\cdot)$ for $\theta \in \Theta$, it is clear that if $(x_n)_{n\in\mathbb{Z}}$ is an i.i.d.\ sequence with marginal $\mu$ and we are given $u_N$, then $(u_n)_{n\geq N}$ is the Markov chain constructed by the iterating random functions $f_{x_n}$. If we know that the backward iteration $Y_n$ converges almost surely to a limit which does not depend on $s$, then for any $n \in \Z$, the limit
\[Z_n:= \lim_{m \to \infty} (f_{\theta_{n+1}} \circ f_{\theta_{n+2}}   \dots  f_{\theta_m}) (s)\]
also exists almost surely and does not depend on $s$ (cf.\ \cite[Section 4]{KO}). In particular, $Z_n$ is measurable with respect to $(\theta_{m})_{m \ge 1+n}$, and $Z_n=f_{\theta_{n+1}}(Z_{n+1})$ for all $n$. Setting $x_n:=\theta_{1-n}$ and $u_n:=Z_{-n}$, it follows that
\[u_{n}=f_{x_n}(u_{n-1}),\]
and $u_n$ is measurable with respect to $(x_m)_{m \le n}$. In conclusion, for $\mu^{\Z}$-a.e.\ realization of $(x_n)_{n\in\mathbb{Z}}$, there exists at least one $(u_n)_{n \in \Z}$ satisfying $u_{n}=F^{(2)}(x_n, u_{n-1})$ and $u_n$ is measurable with respect to $(x_m)_{m \le n}$. Moreover, the distribution of $(u_n)_{n \in \Z}$ is translation invariant, being given by the stationary distribution for the Markov chain constructed by the iterated random functions $f_{x_n}$.

For type II models, the story is similar. In this case we have a bijection $F_*:  \mathcal{X}_0 \times \mathcal{U}_0 \to  \tilde{\mathcal{X}}_0 \times \tilde{\mathcal{U}}_0$, and taking $S=\mathcal{U}_0$, $\Theta=\mathcal{X}_0 \times \tilde{\mathcal{X}}_0$, and $f_{\theta}$ as
\[f_{x,\tilde{x}}(s)=(F^{-1}_*)^{(2)}\left(\tilde{x}, F_*^{(2)}(x,s)\right),\]
we can repeat the discussion of the preceding paragraph. For the ultra-discrete Toda model in particular, we have that
\[f^{udT}_{b,a}(c)=a+\max\{c-b,0\},\]
which can be analysed in the same way as the $G/G/1$ queue considered in \cite[Section 4]{DF}. More specifically, in the latter example, the map of interest is given by
\[f^{G/G/1}_{\theta}(s)=\max\{x+\theta,0\}.\]
Although this is not a strict contraction, it is nonetheless shown in \cite{DF} that, under a certain condition on the distribution $\mu$, which includes the case when $\int \theta d\mu<0$, the backward iteration converges almost surely to a limit which does not depend on $s$. To transfer the argument to the ultra-discrete Toda case, we first make the elementary observation that
\[f^{udT}_{E_{-1},Q_0}\circ f^{udT}_{E_{-2},Q_{-1}}\circ\cdots\circ f^{udT}_{E_{-n},Q_{-n+1}}(s)=Q_0+f^{G/G/1}_{\theta_1}{\circ}f^{G/G/1}_{\theta_2}\circ\cdots \circ f^{G/G/1}_{\theta_{n-1}}(s-E_{-n}),\]
where $\theta_i:=Q_{-i}-E_{-i}$. We can then apply the identity given at \cite[(4.4)]{DF} (that originally appeared in \cite{Fel}) to obtain that the above expressions are both equal to
\[Q_0+\max\left\{0,\theta_1,\theta_1+\theta_2,\dots,\theta_1+\theta_2+\dots+\theta_{n-1}, \theta_1+\theta_2+\dots+\theta_{n-1}+s-E_{-n}\right\}.\]
It readily follows that if $\mathbf{E}(Q_n-E_n)<0$ (cf.\ the requirement on configurations in $\mathcal{X}_{udT}$ in Section \ref{todasec}), then this backward iteration converges almost-surely, for any $s$, to the finite random variable $Q_0+\max\{0,\theta_1,\theta_1+\theta_2,\dots\}$. As is shown in \cite[Theorem 2.3]{CSTkdv}, this precisely corresponds to the value of $U^0_0$ given by the unique solution to the initial value problem for \eqref{UDTODA} with initial condition $(Q_n,E_n)_{n\in\mathbb{Z}}$. One can similarly reconstruct $(U^0_n)_{n\in\mathbb{Z}}$, and indeed the dynamics for all time using this iterated random function approach.

\begin{rem}
The connection between the ultra-discrete Toda lattice and queueing theory is further highlighted by a comparison of the framework and results of the present paper with those of \cite{QST}. Indeed, in the latter work, the local dynamics of the model studied precisely correspond to those given by the map $F_{udT}$, with the variables $(Q_n,E_n,U_n,\mathcal{T}_{udT}Q_n,\mathcal{T}_{udT}E_n)$ in our notation being the analogues of $(s_n,a_n,w_n+s_n,r_n,d_n)$ in that of \cite{QST}. In particular, \cite{QST} gives a version of Burke's theorem for the queuing process in question, with exponential/geometric invariant measures. (Cf.\ the discussion concerning directed last passage percolation in Subsection \ref{dlppsec}.)
\end{rem}

Next, returning to type I models, if the backward iteration converges, one can further consider the question of invariance. Namely, when is it the case that $(\mathcal{T}(x)_n)_{n\in\mathbb{Z}}$, as defined by $\mathcal{T}(x)_n:=F^{(1)}(x_n,u_{n-1})$, has the same distribution as $(x_n)_{n\in\mathbb{Z}}$, where $u_n$ is defined by the backward iteration? On this issue, in \cite{KO}, it is shown that when:
\begin{enumerate}
\item[(i)] the Markov chain $(u_n)_{n\in\mathbb{Z}}$ has reversible transition probabilities,
\item[(ii)] for each $s\in S$, the map $\theta \mapsto (s,f_{\theta}(s))$ is injective,
\end{enumerate}
if we set
\[\tilde{\mathcal{T}}(x)_n=\phi(u_n,u_{n-1}),\]
where $\phi$ is the inverse function of $\theta \mapsto (s,f_{\theta}(s))$, then $(\tilde{\mathcal{T}}(x)_n)_{n\in\mathbb{Z}}$ is an i.i.d.\ sequence with marginal $\mu$. NB. It is straightforward to check that, for a type I model and a measure $\mu$ on $\mathcal{X}_0$ such that $\mu^\mathbb{Z}(\mathcal{X}^*)=1$, we almost-surely have that $\mathcal{T}=\tilde{\mathcal{T}}$. It is moreover shown in \cite[Theorem 4.1]{KO} that $u_0$ is independent of ${\mathcal{T}}(x)_{0},{\mathcal{T}}(x)_{-1},\dots$, which yields that $(u^t_0)_{t \in\Z}$ is an i.i.d.\ sequence, where $(u^t_n)_{n\in\mathbb{Z},t\in\mathbb{Z}}$ is defined recursively. In addition, if $(x_n)_{n \in \Z}$ and $(u^t_0)_{t \in\Z}$ are one-to-one almost surely, then the dynamics given by ${\mathcal{T}}$ are ergodic (actually Bernoulli) with respect to $\mu^{\Z}$ (cf. \cite[Theorem 2.2]{KO} and Theorem \ref{ergodicthm} above). As an example, the authors of \cite{KO} study a discrete-time version of the M/M/1 queue, the dynamics of which are equivalent to BBS$(1,\infty)$ started from an i.i.d.\ configuration. The aim of their paper was to establish the ergodicity of the dynamics, and it was left as an open problem to identify under what conditions $\mathcal{T}$ is ergodic more generally. Whilst we do not address that question here, we do provide further examples of models satisfying the various conditions, namely the ultra-discrete and discrete KdV equations with appropriate i.i.d.\ marginals, as described in Section \ref{kdvsec}.

\begin{rem}
The conditions (i) and (ii) above imply that there exists an involution $F: \Theta \times S \to \Theta \times S$ (at least, on the support of appropriate measures) that is an extension of the map $(\theta,s)\mapsto f_\theta(s)$. More precisely, if we assume that $\theta \to (s,f_{\theta}(s))$ is injective for each $s\in S$, and that the set $ \{(s,f_{\theta}(s)) \in S^2 :\: \theta \in \Theta\} \subseteq S^2$ is symmetric in the two coordinates, then such an $F : \Theta \times S \to \Theta \times S$ is given by $F(\theta,s):=(\phi(f_{\theta}(s),s), f_{\theta}(s))$. Note that, even if there exists such an extension, however, we can not expect that $(f_{\theta})_{\theta\in\Theta}$ is contracting on average in general. Indeed, although the relevant backward iteration converges, the example studied in \cite{KO} does not satisfy the latter property.
\end{rem}

Another approach to demonstrating convergence of the backward iteration for a certain iterated random function system is set out in \cite{Yano,YanoY}. In the latter works, a key notion is that of a `synchronizing sequence', which represents a finite string $\theta_1,\theta_2,\dots,\theta_n$ such that the image of $f_{\theta_1}\circ f_{\theta_2}\circ \cdots \circ f_{\theta_n}$ contains exactly one point. If such a string occurs infinitely often under the measure $\mu^\mathbb{Z}$, then it is easy to see that the backward iteration converges. Observe that we have applied the same idea in the proof of Lemma \ref{udkdvlem}, with the conditions on $x_n+x_{n+1}$ given in \eqref{sync1} and \eqref{sync2} being `synchronizing' for the ultra-discrete KdV system with $J>K$.

Finally, we further note that there has also been a series of works on the stochastic equation:
\[\eta_k=\xi_k \eta_{k-1},\qquad k\in\mathbb{Z},\]
where $(\xi_k)_{k \in\mathbb{Z}}$ is the `evolution process', and $(\eta_k)_{k \in \mathbb{Z}}$ is an unknown process, with both taking values in a compact group $G$ (see the survey \cite{YY} and the references therein). It is clear that this model is in the setting of iterated random functions with $\Theta=S=G$ and $f_{\theta}(s)=\theta s$. Moreover, it is obvious that in this case there exists an involution $F : \Theta \times S \to  \Theta\times S $ such that $F^{(2)}( \theta,s)=f_{\theta}(s)$, as given by $F(\theta,s)=(\theta^{-1}, \theta s)$. These studies are motivated by Tsirelson’s equation, and in particular, it is shown that the Markov chains given by this type of iterated random function system can have a quite different behaviour to the models discussed above. Namely, depending on the distribution of $\theta_n$, the Markov chain might or might not have a unique stationary distributional solution or a strong solution (i.e.\ for which $\eta_k$ is measurable with respect to $(\xi_m)_{m \le k}$), and surprisingly, when the uniqueness of the stationary distributional solution holds, then there does not exist a strong solution, and on the other hand, when there is a strong solution, then there exist multiple strong solutions (for details, see \cite{YY}).

\section{Open problems and conjectures}\label{oqsec}

\subsection{Problems for KdV- and Toda-type discrete integrable systems}

\begin{prob}
Completely characterize the detailed balance solutions for \eqref{UDKDV}, i.e.\ remove the technical conditions from Proposition \ref{udkdvmeasall}.
\end{prob}

\begin{prob}
Completely characterize the detailed balance solutions for \eqref{DKDV}, i.e.\ extend the final claim of Proposition \ref{dkdvmeas} to general $\alpha,\beta\geq 0$ (see Conjecture \ref{con82} below for our expectation in this direction). Moreover, describe a reasonable subset of $\mathcal{X}^*$ when $\alpha\beta>0$, so that the invariance and ergodicity results can be extended to these cases. (As commented above, the results of \cite{CSTkdv} do not apply.)
\end{prob}

\begin{prob}
Give an argument for establishing the ergodicity of invariant measures for type II models, and in particular apply this in the case of the discrete and ultra-discrete Toda lattice equations. (Ergodicity of a polymer model related to the discrete Toda lattice, cf.\ Subsection \ref{drpsec}, is studied in \cite{JRA}.)
\end{prob}

\begin{prob}
In Section \ref{stochsec}, we presented some basic connections between the ultra-discrete/ discrete
% forced line-break - remove if not needed in formatting.
Toda lattices and certain stochastic integrable systems that explain why the invariant measures of the corresponding systems match up. In the last few decades, an important aspect of research in stochastic integrable systems has been the development of machinery to study models in the Kardar-Parisi-Zhang (KPZ) universality class (see \cite{Corwinsurvey} for background). Remarkably, it has recently been seen that the KPZ fixed point can be linked to the Kadomtsev-Petviashivili (KP) equation, which is a two-dimensional version of the KdV equation \cite{QR}. These observations naturally lead one to wonder where else there might be parallels between deterministic integrable systems of KdV/Toda-type, and stochastic integrable systems in the KPZ universality class, and to what extent these might be used to transfer knowledge between the two areas.
\end{prob}

\subsection{Characterizations of standard distributions}

In the course of this work, and in particular when solving the various detailed balance equations, we have applied several classical results of the form: if $X$ and $Y$ are independent, then so are $U$ and $V$, where $(U,V)=F(X,Y)$ for a given $F$, if and only if the distribution of $(X,Y)$ falls into a certain class. Perhaps the most famous result in this direction is that first proved by Kac in 1939: `if $X$ and $Y$ are independent, then so are $X+Y$ and $X-Y$ if and only if both $X$ and $Y$ have normal distributions with a common variance' (see \cite{Kac}, as described in \cite{Fe1}). In this subsection, alongside recalling other known results for specific involutions or bijections $F$, we formulate a number of natural conjectures that arise from our study. NB. In what follows, we say that random variables are `non-trivial' if they are non-Dirac).

As a first example, we recall the characterization of the product of GIG and gamma distributions from \cite{LW}. Similar results are sometimes described in the literature as being of `Matsumoto-Yor type', after \cite{MY2}, where the `if' part of the result was established (see \cite{KV}, for example).

\begin{thm}[\cite{LW}]\label{GIG} Let $F:(0,\infty)^2\rightarrow(0,\infty)^2$ be the involution given by
\[F(a,b)=\left(\frac{1}{a+b}, \frac{1}{a}-\frac{1}{a+b}\right).\]
Let $X$ and $Y$ be non-trivial $(0,\infty)$-valued independent random variables. It is then the case that $(U,V):=F(X,Y)$ are independent if and only if there exist $\la, c_1, c_2 > 0$ such that
\[X \sim \mathrm{GIG}(\la,c_1,c_2),\qquad Y \sim \mathrm{Gam}(\la, c_1),\]
and in this case, $U \sim \mathrm{GIG}( \la,c_2,c_1 )$ and $V \sim  \mathrm{Gam}(\la, c_2)$. Hence, if moreover $(U,V)$ has the same distribution as $(X,Y)$, then $X \sim \mathrm{GIG}( \la,c,c) $ and  $Y \sim \mathrm{Gam}(\la, c)$ for some $\lambda,c>0$.
\end{thm}

As a direct corollary, by making the change of variables $(a,b) \to (a,b^{-1})$, one can check a similar result for the involution $F:(0,\infty)^2\rightarrow(0,\infty)^2$ given by
\begin{equation}\label{fmap}
F(a,b)=\left(\frac{b}{1+ab}, a(1+ab)\right).
\end{equation}
In this case, the random variables $X$ and $U$ have the same distribution as in Theorem \ref{GIG}, but $Y\sim \mathrm{IG}(\la ,c_1)$ and $V\sim \mathrm{IG}(\la ,c_2)$. Now, the above map is precisely $F_{dK}^{(1,0)}$, and indeed it was the conclusion of \cite{LW} that we applied in the proof of Proposition \ref{dkdvmeas} to characterize the solutions of the detailed balance equation for $F_{dK}^{(\alpha,\beta)}$ with $\alpha\beta=0$. In light of the conclusion of Proposition \ref{dkdvmeas}, we conjecture that for general $\alpha, \beta\geq 0$, a similar result holds.

\begin{conj}\label{con82} Let $\alpha,\beta\geq 0$ with $\alpha\neq \beta$, and recall the definition of $F_{dK}^{(\alpha,\beta)}$ from \eqref{DKDV}. Let $X$ and $Y$ be non-trivial $(0,\infty)$-valued independent random variables. It is then the case that $(U,V):=F_{dK}^{(\alpha,\beta)}(X,Y)$ are independent if and only if there exist $\la, c_1,c_2 > 0$ such that
\[X \sim \mathrm{GIG}(\la,c_1\alpha ,c_2),\qquad Y \sim \mathrm{GIG}(\la,c_2\beta ,c_1),\]
and in this case $U \sim \mathrm{GIG}( \la,c_2\alpha, c_1 )$ and $V \sim \mathrm{GIG}(\la, c_1\beta, c_2)$. Hence, if moreover $(U,V)$ has the same distribution as $(X,Y)$, then $X \sim \mathrm{GIG}(\la, c\alpha ,c) $, $Y \sim \mathrm{GIG}(\la, c\beta,c)$ for some $\la,c>0$.
\end{conj}

The next statement was applied in the proof of Proposition \ref{DTODAmeas} when characterising the solutions of the detailed balance equation for the discrete Toda system. Moreover, this and the subsequent two results were used in \cite{CN} to characterize directed random polymer models having stationary measures satisfying Burke's property. We note that Corollary \ref{cor84} is a direct consequence of Theorem \ref{Gam}.

\begin{thm}[\cite{L}]\label{Gam} Let $F:(0,\infty)^2\rightarrow(0,\infty)\times(0,1)$ be the bijection given by
\[F(a,b)=\left({a+b}, \frac{a}{a+b}\right).\]
NB. $F^{-1}(a,b)=(ab,a(1-b))$. Let $X$ and $Y$ be non-trivial $(0,\infty)$-valued independent random variables. It is then the case that $(U,V):=F(X,Y)$ are independent if and only if there exist $\la, c_1, c_2 > 0$ such that
\[X \sim \mathrm{Gam}(\la_1,c) ,\qquad Y \sim \mathrm{Gam}(\la_2 ,c),\]
and in this case, $U \sim \mathrm{Gam}(\la_1+\la_2,c) $ and $V \sim \mathrm{Be}(\la_1 ,\la_2)$.
\end{thm}

\begin{cor}\label{cor84}
Let $F:(0,\infty)\times(0,1)\rightarrow(0,\infty)^2$ be the bijection given by
\[F(a,b)=\left(ab,a(1-b)\right).\]
Let $X$ and $Y$ be non-trivial $(0,\infty)$-valued and $(0,1)$-valued, respectively, independent random variables. It is then the case that $(U,V):=F(X,Y)$ are independent if and only if there exist $\la, c_1, c_2 > 0$ such that
\[X \sim \mathrm{Gam}(\la_1+\la_2,c) ,\qquad Y \sim \mathrm{Be}(\la_1 ,\la_2),\]
and in this case, $U \sim \mathrm{Gam}(\la_1,c)$ and $V \sim\mathrm{Gam}(\la_2 ,c)$.
\end{cor}

\begin{thm}[\cite{SW}]\label{t85} Let $F:(0,1)^2\rightarrow(0,1)^2$ be the involution given by
\[F(a,b)=\left(\frac{1-b}{1-ab}, 1-ab\right).\]
Let $X$ and $Y$ be non-trivial $(0,1)$-valued independent random variables. It is then the case that $(U,V):=F(X,Y)$ are independent if and only if there exist $p,q,r>0$ such that
\[X \sim \mathrm{Be}(p,q),\qquad Y \sim \mathrm{Be}(p+q ,r),\]
and in this case, $U \sim \mathrm{Be}(r,q)$ and $V \sim \mathrm{Be}(q+r,p)$. Hence, if moreover $(U,V)$ has the same distribution as $(X,Y)$, then $X \sim \mathrm{Be}(p,q)$, $Y \sim \mathrm{Be}(p+q,p)$.
\end{thm}

Just as we related solutions of the detailed balance equations for the discrete and ultra-discrete KdV- and Toda-type systems in Section \ref{linksec}, it is possible to ultra-discretize the above statements, and this leads to a number of further conjectures. To do this, we transform variables taking values in $(0,1)$ to $(0,\infty)$ via the bijection $x \mapsto \frac{1}{x^{-1}-1}$ (the inverse of which is $x \mapsto \frac{1}{1+x^{-1}}$). The ultra-discretization procedure is then given by applying the limit
\[F(a,b) \mapsto \lim_{ \varepsilon \to 0} \left(\iota\varepsilon \log F^{(1)}\left(e^{\iota a \varepsilon^{-1}}, e^{\iota b \varepsilon^{-1}}\right), \iota\varepsilon \log F^{(2)}\left(e^{\iota a \varepsilon^{-1}}, e^{\iota b \varepsilon^{-1}}\right)\right),\]
where we take $\iota=1$ for Conjectures \ref{c86} and \ref{c87}, and $\iota=-1$ in the remaining cases. Precisely, we arrive at Conjecture \ref{c86} from the map at \eqref{fmap}, Conjecture \ref{c87} from Conjecture \ref{con82}, Theorem \ref{Exp}/Corollary \ref{Expcor} from Theorem \ref{Gam}/Corollary \ref{cor84}, and Conjecture \ref{c810} from Theorem \ref{t85}.
%MS - Maybe it is useful to say that $(B, 1-B) \to (X^+,X^-)$ where B is a Beta(a,b) random variable and X is a AL(a,b) random variable.

\begin{conj}\label{c86} If $F(a,b)=F^{(0,\infty)}_{udK}(a,b)$, then $F  : \R^2 \to \R^2$ is an involution. For any $c >0$, $F : [-c, c] \times [-c, \infty) \to [-c, c] \times [-c, \infty)$ is an involution, and for any $c_1, c_2 >0$, $F : [-c_1, c_2] \times [-c_2, \infty) \to [-c_2, c_1] \times [-c_1, \infty)$ is a bijection. Let $X$ and $Y$ be absolutely continuous $\mathbb{R}$-valued independent random variables satisfying $P(X >0) P(X<0) \neq 0$. It is then the case that $(U,V):=F(X,Y)$ are independent if and only if there exist $\la ,c_1, c_2>0 $ such that
\[X \sim \mathrm{stExp}(\la , -c_1 ,c_2 ),\qquad Y \sim \mathrm{sExp}(\la, -c_2),\]
and in this case, $U \sim \mathrm{stExp}(\la , -c_2 ,c_1 )$, $V \sim \mathrm{sExp}(\la, -c_1)$. Hence, if moreover $(U,V)$ has the same distribution as $(X,Y)$, then $X \sim \mathrm{stExp}(\la , -c ,c) $, $Y \sim \mathrm{sExp}(\la, -c)$ for some $c>0$.
\end{conj}

\begin{conj}\label{c87} If $F(a,b)=F^{(J,K)}_{udK}(a,b)$ for some $-\infty < J, K <\infty$, then $F  : \R^2 \to \R^2$ is an involution. Also, for any $c < \min\{ \frac{J}{2}, \frac{K}{2} \}$, $F : [c, J-c] \times [c, K-c] \to [c, J-c] \times [c, K-c]$ is an involution, and
for any $c_1, c_2 < \min\{ \frac{J}{2}, \frac{K}{2} \}$, $F : [c_1, J-c_2] \times [c_2, K-c_1]  \to [c_2, J-c_1] \times [c_1, K-c_2]$ is a bijection. Let $X$ and $Y$ be absolutely continuous $\mathbb{R}$-valued independent random variables satisfying $P(X >\frac{J}{2})P(X<\frac{J}{2})P(Y >\frac{K}{2}) P(Y<\frac{K}{2}) \neq 0$.  It is then the case that $(U,V):=F(X,Y)$ are independent if and only if there exist $\la >0$ and $c_1, c_2 < \min\{ \frac{J}{2}, \frac{K}{2} \}$ such that
\[X \sim \mathrm{stExp}(\la ,c_1 ,J-c_2 ), \qquad Y \sim \mathrm{stExp}(\la, c_2,K-c_1 ),\]
and in this case, $U \sim \mathrm{stExp}(\la , c_2 , J-c_1 )$, $V \sim \mathrm{stExp}(\la, c_1,K-c_2)$. Hence, if moreover $(U,V)$ has the same distribution as $(X,Y)$, then $X \sim \mathrm{stExp}(\la, c ,J-c) $, $Y \sim \mathrm{stExp}(\la, c, K-c)$ for some $c < \min\{ \frac{J}{2}, \frac{K}{2} \}$.
\end{conj}

\begin{rem}
It is also possible to write down discrete versions of the previous two conjectures, replacing the $\mathrm{stExp}$ distribution with the $\mathrm{sstbGeo}$ one, cf.\ Proposition \ref{udkdvmeas}. The appearance of the bipartite version in the discrete case of these results is an interesting consequence of the particular structure of the ultra-discrete KdV system. Similarly, one might also make a discrete version of Conjecture \ref{c810} below involving the $\mathrm{sdAL}$ distribution.
\end{rem}

\begin{thm}[\cite{Cr}]\label{Exp} Let $F:\mathbb{R}^2\rightarrow\mathbb{R}^2$ be the bijection given by
\[F(a,b)=\left(\min\{a,b\}, a-b\right).\]
NB. $F^{-1}(a,b)=(a+\max\{b,0\},a-\min\{b,0\})$. Let $X$ and $Y$ be non-trivial $\mathbb{R}$-valued independent random variables. It is then the case that $(U,V):=F(X,Y)$ are independent if and only if there exist $\la_1, \la_2, c > 0$ such that
\[X \sim \mathrm{sExp}(\la_1,c),\qquad Y \sim \mathrm{sExp}(\la_2 ,c)\]
or $\theta_1, \theta_2 \in (0,1), m>0, M \in \Z$ such that
\[X \sim \mathrm{ssGeo}(1-\theta_1,M,m),\qquad Y \sim  \mathrm{ssGeo}(1-\theta_2,M,m),\]
and in this case  $U \sim \mathrm{stExp}(\la_1+\la_2,c)$, $V \sim \mathrm{AL}(\la_1 ,\la_2)$, or $U \sim \mathrm{ssGeo}(1-\theta_1\theta_2,M,m)$, $V \sim \mathrm{sdAL}(1-\theta_1,1-\theta_2,m)$, respectively.
\end{thm}

\begin{cor}\label{Expcor} Let $F:\mathbb{R}^2\rightarrow\mathbb{R}^2$ be the bijection given by
\[F(a,b)=\left(a+\max\{b,0\},a-\min\{b,0\}\right).\]
Let $X$ and $Y$ be non-trivial $\mathbb{R}$-valued independent random variables. It is then the case that $(U,V):=F(X,Y)$ are independent if and only if there exist $\la_1, \la_2, c > 0$ such that
\[X \sim\mathrm{sExp(\la_1+\la_2,c)},\qquad Y \sim \mathrm{AL}(\la_1 ,\la_2),\]
or $\theta_1, \theta_2 \in (0,1), m>0, M \in \Z$ such that
\[X \sim \mathrm{ssGeo}(1-\theta_1\theta_2,M,m),\qquad Y \sim  \mathrm{sdAL}(1-\theta_1,1-\theta_2,m),\]
and in this case  $U \sim \mathrm{sExp}(\la_1,c)$, $V \sim \mathrm{sExp}(\la_2,c)$, or $U \sim \mathrm{ssGeo}(1-\theta_1,M,m)$, $V\sim \mathrm{ssGeo}(1-\theta_2,M,m)$, respectively.
\end{cor}

\begin{conj}\label{c810} Let $F:\mathbb{R}^2\rightarrow\mathbb{R}^2$ be the involution given by
\[F(a,b)=\left(\min\{a, 0\}-b, \min\{a,b,0\}-a-b\right).\]
Let $X$ and $Y$ be absolutely continuous $\mathbb{R}$-valued independent random variables. It is then the case that $(U,V):=F(X,Y)$ are independent if and only if there exist $p,q,r > 0$ such that
\[X \sim \mathrm{AL}(p,q),\qquad Y \sim \mathrm{AL}(p+q, r),\]
and in this case, $U \sim \mathrm{AL}(r,q)$, $V \sim \mathrm{AL}(q+r, p)$. Hence, if moreover $(U,V)$ has the same distribution as $(X,Y)$, then $X \sim \mathrm{AL}(p,q) $, $Y \sim \mathrm{AL}(p+q, p)$.
\end{conj}
% I could check at least these distributions satisfy the independence.  Discrete version should be added : $X \sim dAL(p ,q, m) $, $Y \sim dAL(p+q, r,m)$.

\begin{rem}
Since this article was completed, some of the above conjectures have been addressed in \cite{BaoNoack}. In particular, under technical conditions, Theorems 1.1, 1.2 and 1.3 of \cite{BaoNoack} confirm Conjectures \ref{con82}, \ref{c810} and \ref{c86}, respectively. It remains to check discrete versions of the latter two claims.
\end{rem}

\section*{Acknowledgements}

This research was supported by JSPS Grant-in-Aid for Scientific Research (B), 19H01792. The research of DC was also supported by JSPS Grant-in-Aid for Scientific Research (C), 19K03540, and the Research Institute for Mathematical Sciences, an International Joint Usage/Research Center located in Kyoto University. This work was completed while MS was kindly being hosted by the Courant Institute, New York University.

\appendix

\section{Probability distributions}

In the following list, we give definitions of the various probability distributions that appear within this article.

\begin{description}
  \item[Shifted truncated exponential distribution] For $\lambda,c_1,c_2\in\mathbb{R}$ with $c_1<c_2$, the \emph{shifted truncated exponential} distribution with parameters $(\lambda,c_1,c_2)$, which we denote $\mathrm{stExp}(\lambda,c_1,c_2)$, has density
\[\frac{1}{Z}e^{-\la x}\mathbf{1}_{[c_1,c_2]}(x),\qquad x\in\mathbb{R},\]
where $Z$ is a normalizing constant.
  \item[Shifted exponential distribution] For $\lambda >0$, $c\in \R$,  the \emph{shifted exponential} distribution with parameters $(\lambda,c)$, which we denote $\mathrm{sExp}(\lambda,c)$, has density
\[\frac{1}{Z}e^{-\la x}\mathbf{1}_{[c,\infty)}(x),\qquad x\in\R,\]
where $Z$ is a normalizing constant. We use the convention that $\mathrm{stExp}(\lambda,c,\infty)=\mathrm{sExp}(\lambda,c)$ when $\lambda >0$.
  \item[Shifted scaled (truncated bipartite) geometric distribution] For $\theta>0$, $M\in\mathbb{Z}$, $N\in\mathbb{Z}\cup\{\infty\}$ such that $M\leq N$, $\kappa>0$ and $m\in(0,\infty)$, we say a random variable $X$ has \emph{shifted scaled truncated bipartite geometric distribution} with parameters $1-\theta$, $M$, $N$, $\kappa$ and $m$ if
\[\mathbf{P}\left(X=mx\right)=\frac{1}{Z}\theta^{x}\kappa^{\iota(x)},\qquad x\in\{M,M+1,\dots,N\},\]
where $\iota(2x)=0, \iota(2x+1)=1$ and $Z$ is a normalising constant; in this case we write $X\sim \mathrm{sstbGeo}(1-\theta,M,N,\kappa,m)$. Note that, if $N=\infty$, then we require that $\theta<1$ for the distribution to be defined. We observe that $\mathrm{sstbGeo}(1-\theta,0,N,1,1)$ is simply the distribution of the usual parameter $1-\theta$ geometric distribution conditioned to take a value in $\{0,1,\dots,N\}$. In the special case when $\theta<1$, $N=\infty$, $\kappa=1$, we say that $X$ has \emph{shifted scaled geometric distribution} with parameters $1-\theta$, $M$ and $m$, and write $X\sim \mathrm{ssGeo}(1-\theta,M,m)$.
\item[Asymmetric Laplace distribution] For $\lambda_1,\lambda_2\in(0,\infty)$, the \emph{asymmetric Laplace} distribution with parameters $(\lambda_1,\lambda_2)$, which we denote $\mathrm{AL}(\lambda_1,\lambda_2)$, has density
\[\frac{1}{Z}\left(e^{-\lambda_1 x}\mathbf{1}_{(0,\infty)}(x)+e^{\lambda_2 x}\mathbf{1}_{(-\infty,0)}(x)\right),\qquad x\in\R,\]
where $Z$ is a normalizing constant.
\item[Scaled discrete asymmetric Laplace distribution] For $\theta_1,\theta_2\in(0,1)$ and $m \in (0,\infty)$, we say a random variable $X$ has \emph{scaled discrete asymmetric Laplace} distribution with parameters $(1-\theta_1,1-\theta_2,m)$ if
\[\mathbf{P}\left(X=mx\right) =\left\{\begin{array}{ll}
                    \frac{1}{Z}\theta_1^{x}, & x \in \{0,1,2,\dots\}, \\
                    \frac{1}{Z}\theta_2^{-x}, & x \in \{ \dots, -2,-1\},
                  \end{array}\right.\]
where $Z$ is a normalizing constant; in this case we write $X \sim \mathrm{sdAL}(1-\theta_1,1-\theta_2,m)$.
\item[Gamma distribution] For $\lambda,c\in(0,\infty)$, the \emph{gamma} distribution with parameters $(\lambda,c)$, which we denote $\mathrm{Gam}(\lambda,c)$, has density
\[\frac{1}{Z}x^{\lambda-1}e^{-cx}\mathbf{1}_{(0,\infty)}(x),\qquad x\in\mathbb{R},\]
where $Z$ is a normalizing constant.
\item[Inverse gamma distribution] For $\lambda,c\in(0,\infty)$, the \emph{inverse gamma} distribution with parameters $(\lambda,c)$, which we denote $\mathrm{IG}(\lambda,c)$, has density
\[\frac{1}{Z}x^{-\lambda-1}e^{-cx^{-1}}\mathbf{1}_{(0,\infty)}(x),\qquad x\in\mathbb{R},\]
where $Z$ is a normalizing constant.
\item[Generalized inverse Gaussian distribution] For $\lambda \in \R$, $c_1,c_2 \in(0,\infty)$, the \emph{generalized inverse Gaussian} distribution with parameters $(\lambda,c_1,c_2)$, which we denote $\mathrm{GIG}(\lambda,c_1,c_2)$, has density
\[\frac{1}{Z}x^{-\lambda-1}e^{-c_1x-c_2x^{-1}}\mathbf{1}_{(0,\infty)}(x),\qquad x\in\mathbb{R},\]
where $Z$ is a normalizing constant.  We use the convention that $\mathrm{GIG}(\lambda,0,c)=\mathrm{IG}(\lambda,c)$.
\item[Beta distribution] For $\lambda_1,\lambda_2 \in(0,\infty)$, the \emph{beta} distribution with parameters $(\lambda_1,\lambda_2)$, which we denote $\mathrm{Be}(\lambda_1,\lambda_2)$, has density
\[\frac{1}{Z} x^{\lambda_1-1}(1-x)^{\lambda_2-1}\mathbf{1}_{(0,1)}(x),\qquad x\in\mathbb{R},\]
where $Z$ is a normalizing constant.
\item[$q$-negative binomial distribution] Fix $q\in[0,1)$. For $p,b \in [0,1)$ or $p <0$, $b=q^{-L}$ for some $L \in \Z$, we say a random variable $X$ has \emph{$q$-negative binomial} distribution with parameters $(p,b)$ if
\[\mathbf{P}\left(X=n\right)=\frac{1}{Z}p^n\frac{(b;q)_n}{(q;q)_n}, \qquad  n \in \{0,1,2,\dots\},\]
where $(a;q)_n:=(1-a)(1-aq)\dots(1-aq^{n-1})$ for $n\geq 1$, $(a;q)_0:=1$, and $Z$ is a normalising constant, which can be given explicitly as $Z=\frac{(pb;q)_{\infty}}{(b;q)_{\infty}}$; in this case we write $X\sim \mathrm{qNB}(b,p)$. Note that, if $p,b \in [0,1)$, then the support of $X$ is $\Z_+$, and if $p<0$ and $b=q^{-L}$ for some $L \in \Z$, then the support of $X$ is $\{0,1,2,\dots, L\}$.
\end{description}

\bibliography{irf}
\bibliographystyle{amsplain}

\end{document}